\documentclass{amsart}

\pdfoutput=1

\usepackage{amssymb,amsmath,amsthm,enumerate,fullpage,float}
\usepackage[usenames,dvipsnames]{color}
\usepackage{bm}
\usepackage{stmaryrd}
\usepackage{enumitem}
\usepackage{graphicx}
\usepackage{gensymb}
\usepackage{tikz}
\usetikzlibrary{matrix}
\usetikzlibrary{calc,decorations.pathmorphing,decorations.markings,
decorations.pathreplacing,patterns,shapes,arrows}

\newtheorem{lem}{Lemma}
\newtheorem{thm}{Theorem}
\newtheorem{conj}{Conjecture}
\newtheorem{cor}{Corollary}
\newtheorem{rmk}{Remark}
\let\oldrmk\rmk
\renewcommand{\rmk}{\oldrmk\normalfont}

\newcommand{\eref}[1]{(\ref{#1})}

\renewcommand{\b}[1]{\bar{#1}}
\renewcommand{\leq}{\leqslant}
\renewcommand{\geq}{\geqslant}

\definecolor{pink}{RGB}{255,128,128}
\definecolor{lred}{RGB}{255,0,0}
\definecolor{dred}{RGB}{176,0,0}
\definecolor{burg}{RGB}{43,0,0}

\def\ie{{\it i.e.}\ }

\def\ihat{I}

\def\fs{\footnotesize}

\def\u{\scalebox{0.65}{$\blacktriangle$}}
\def\d{\scalebox{0.65}{$\blacktriangledown$}}
\def\l{\scalebox{0.65}{$\blacktriangleleft$}}
\def\r{\scalebox{0.65}{$\blacktriangleright$}}

\newcommand{\binomQ}{\genfrac[]{0pt}{}}

\newcommand{\pf}{\mathop{\rm Pf}}

\begin{document}

\title{Refined Cauchy/Littlewood identities and six-vertex model partition functions: II. Proofs and new conjectures}

\author{D.~Betea, M.~Wheeler and P.~Zinn-Justin}

\address{Laboratoire de Physique Th\'eorique et Hautes \'Energies, CNRS UMR 7589 and Universit\'e Pierre et Marie Curie (Paris 6), 4 place Jussieu, 75252 Paris cedex 05, France}
\email{betea,mwheeler,pzinn@lpthe.jussieu.fr}


\begin{abstract}
We prove two identities of Hall--Littlewood polynomials, which appeared recently in \cite{bw}. We also conjecture, and in some cases prove, new identities which relate infinite sums of symmetric polynomials and partition functions associated with symmetry classes of alternating sign matrices. These identities generalize those already found in \cite{bw}, via the introduction of additional parameters. The left hand side of each of our identities is a simple refinement of a relevant Cauchy or Littlewood identity. The right hand side of each identity is (one of the two factors present in) the partition function of the six-vertex model on a relevant domain. 
\end{abstract}

\maketitle

\section{Introduction}

In this paper we continue our study of Cauchy and Littlewood type identities, initiated in \cite{bw}, and their relationship with partition functions of the six-vertex model. The principal results studied in \cite{bw} were three infinite sum identities for Hall--Littlewood polynomials:
\begin{multline}
\label{HL-ASM}
\sum_{\lambda}
\prod_{i=0}^{\infty}
\prod_{j=1}^{m_i(\lambda)}
(1-t^j)
P_{\lambda}(x_1,\dots,x_n;t)
P_{\lambda}(y_1,\dots,y_n;t)
=
\\
\frac{\prod_{i,j=1}^{n} (1-t x_i y_j)}
{\prod_{1 \leq i<j \leq n} (x_i - x_j) (y_i - y_j)}
\det_{1\leq i,j \leq n}
\left[
\frac{(1-t)}{(1-x_i y_j)(1-t x_i y_j)}
\right]
\end{multline}
\begin{multline}
\label{HL-OSASM}
\sum_{\substack{\lambda\ \text{with} \\ \text{even columns}}} \ \ 
\prod_{i=0}^{\infty}\ \prod_{j=2,4,6,\dots}^{m_i(\lambda)} 
(1-t^{j-1}) 
P_{\lambda}(x_1,\dots,x_{2n};t)
=
\\
\prod_{1 \leq i<j \leq 2n} 
\frac{(1-t x_i x_j)}{(x_i-x_j)} 
\pf_{1\leq i < j \leq 2n} 
\left[ 
\frac{(x_i-x_j)(1-t)}{(1-x_i x_j) (1-t x_i x_j)} 
\right]
\end{multline}
\begin{multline}
\label{HL-UASM}
\sum_{\lambda} 
\prod_{i=0}^{\infty} 
\prod_{j=1}^{m_i(\lambda)} 
(1-t^j)
P_{\lambda}(x_1,\dots,x_n;t) 
K_{\lambda}(y_1,\dots,y_n;t) 
= 
\\
\frac{\prod_{i,j=1}^{n} (1-t x_i y_j) (1-\frac{t x_i}{y_j}) }
{\prod_{1\leq i<j \leq n} (x_i-x_j) (y_i-y_j) (1 - t x_i x_j) (1 - \frac{1}{y_i y_j}) }
\det_{1\leq i,j \leq n}
\left[ 
\frac{(1-t)}{(1- x_i y_j)(1-t x_i y_j)(1- \frac{x_i}{y_j})(1-\frac{t x_i}{y_j})} 
\right]
\end{multline}
where $P_{\lambda}$ and $K_{\lambda}$ denote Hall--Littlewood polynomials of type $A_n$ \cite{mac} and $BC_n$ \cite{ven}, respectively. In \eref{HL-ASM} and \eref{HL-UASM} the sum is taken over all partitions of maximal length $n$, while in \eref{HL-OSASM} the sum is over all partitions of maximal length $2n$ and whose 
Young diagrams have even-length columns. In all equations $m_i(\lambda)$ is multiplicity of the part $i$ in $\lambda$, assuming all partitions are suffixed by $m_0(\lambda)$ zeros to augment them to their maximal length.

Equation \eref{HL-ASM} is due to Warnaar \cite{war}, based on earlier results of Kirillov and Noumi \cite{kn}. In \cite{bw}, we exposed a combinatorial interpretation of \eref{HL-ASM}: the left hand side can be viewed as a generating series of {\it path-weighted} plane partitions \cite{vul}, while the right hand side is the partition function of the six-vertex model under domain wall boundary conditions \cite{kor,ize}, and thus a generating series of alternating sign matrices (ASMs) \cite{kup1,kup2}. 

Equations \eref{HL-OSASM} and \eref{HL-UASM}, both conjectured in \cite{bw}, are further examples of such a relationship. In both of these equations, the right hand side is the partition function of the six-vertex model on a certain lattice (in \eref{HL-OSASM} the underlying lattice has off-diagonal symmetry \cite{kup2}; the partition function in \eref{HL-UASM} arises from reflecting domain wall boundary conditions \cite{tsu}), and may be viewed as a multi-parameter generating series of a symmetry class of ASMs (off-diagonally symmetric ASMs in the case of \eref{HL-OSASM}; U-turn ASMs in the case of \eref{HL-UASM}). Although we are able to view the left hand side of \eref{HL-OSASM} as a generating series of path-weighted symmetric plane partitions \cite{bw}, for the moment there is no known combinatorial interpretation of the left hand side of \eref{HL-UASM} in terms of plane partitions or other tableaux-related objects.

The goals of the present work are as follows. Firstly, we provide a new proof of \eref{HL-ASM} by applying the 
{\it Izergin--Korepin technique}\footnote{We use this as an umbrella term for any proof that involves: {\bf 1}. Writing down a set of properties, one of which is a simple recursion relation, that uniquely determine an object, and {\bf 2.} Showing that a certain determinant or Pfaffian Ansatz satisfies these properties.} \cite{kor,ize} to the left hand side of the equation, before adapting this method to prove \eref{HL-OSASM}. We remain unable to prove \eref{HL-UASM} by such methods, due to the absence of a simple combinatorial (tableau) formula for the $BC_n$-symmetric Hall--Littlewood polynomials. 

Secondly, we shall generalize all three identities by the introduction of additional parameters. It was already demonstrated in \cite{war} that \eref{HL-ASM} may be refined by two extra parameters $q$ and $u$, with the original identity being recovered when $q=0$ and $u=t$. The introduction of the $q$ parameter elevates the participating symmetric functions to Macdonald polynomials, and the equation itself comes from acting on the Cauchy identity with a generating series of Macdonald difference operators \cite{mac} (where $u$ is the indeterminate of the generating series). We prove that even in the presence of the two extra parameters, the right hand side of the identity remains a determinant (a fact which was not explicit in either \cite{war} or \cite{kn}). In a similar vein, we find that it is possible to refine \eref{HL-OSASM} by the introduction of the parameters $q$ and $u$. To round off, we conjecture a deformed version of \eref{HL-UASM} involving $u$ and four parameters $t_0,t_1,t_2,t_3$ which elevates it to the level of {\it lifted Koornwinder polynomials} \cite{rai}.

Finally, we investigate the meaning of the deformation parameters thus introduced in the setting of the 
six-vertex model. Surprisingly, the inclusion of the indeterminate $u$ in our equations does not break the correspondence with partition functions of the six-vertex model: the $u$-deformed versions of 
\eref{HL-ASM}--\eref{HL-UASM} all lead to determinants/Pfaffians which have appeared in \cite{kup2} in the context of further symmetry classes of ASMs. We will not comment on the role of $q$ in this scheme, since it appears to play only a trivial role.\footnote{On the right hand side of the $q$-deformed version of \eref{HL-ASM} and \eref{HL-OSASM}, the dependence on $q$ is completely factorized. In particular, $q$ does not appear within the determinant/Pfaffian, which rules out the possibility that it plays any non-trivial role within the six-vertex model itself.}

The paper is organized as follows. In Section \ref{sec:proofs} we give proofs of two results: identity \eqref{HL-ASM} (using an independent method to that of \cite{war}) and \eqref{HL-OSASM} (conjectured in \cite{bw}). In Section \ref{sec:mac} we discuss the generalization of \eqref{HL-ASM} to Macdonald polynomials (obtained in \cite{war}), and conjecture a companion generalization of \eqref{HL-OSASM} to this level. $u$-deformations of the Cauchy, Littlewood and $BC$-type Cauchy identities are discussed in Sections \ref{sec:HTASM}, \ref{sec:OOASM} and \ref{sec:UUASM}, and their connection with partition functions of ASM symmetry classes is exposed. The main result in Section \ref{sec:HTASM} is that a $u$-deformed version of \eqref{HL-ASM} is closely related to the partition function of half-turn symmetric ASMs (for a particular value of $u$). The main result in Section \ref{sec:OOASM} is Theorem \ref{u-def-little-hl}, a $u$-generalization of equation \eqref{HL-OSASM}. In this case, for an appropriate value of $u$, we obtain a close connection with the partition function of off-diagonally/off-anti-diagonally symmetric ASMs. In Section \ref{sec:UUASM} we conjecture a $u$-generalization of \eqref{HL-UASM} (Conjecture \ref{u-def-cau-bc-hl}). We prove a simpler, companion identity involving symplectic Schur polynomials (Theorem \ref{symp-schur-thm}) but are unable to prove the conjecture (due to the lack of a suitable branching rule for the lifted Koornwinder polynomials which participate). The conjecture has been verified for small partitions using Mathematica and Sage. Once again, a certain value of $u$ leads to a correspondence with a six-vertex model partition function (in this case, the partition function of double U-turn ASMs). Finally, following Rains \cite{rai}, in the Appendix we present a few results on $BC$-type interpolation and Koornwinder polynomials (and their symmetric function analogues) that we use. 

Throughout the paper, $\b{x} := \frac{1}{x}$. An $n$-tuple of variables $(x_1,\dots,x_n)$ will sometimes be denoted $\vec{x}_n$. We reserve letters $\lambda, \mu, \dots$ for partitions. A partition $\lambda$ is either the empty partition ($0$) or a sequence of strictly positive numbers listed in decreasing order: $\lambda_1 \geq \lambda_2 \geq \cdots \geq \lambda_k > 0$. We call each $\lambda_i$ a \textit{part} and $\ell(\lambda):=k$ the \textit{length} (number of non-zero parts) of $\lambda$. If all parts of $\lambda$ are even, we call the partition \textit{even}. $m_i(\lambda)$ stands for the number of parts in $\lambda$ equal to $i$. If for some prespecified $n$ we have $\ell(\lambda) \leq n$, we abuse notation and define $m_0(\lambda) = n-\ell(\lambda)$ to be the number of zeros we need to append to $\lambda$ to get a vector of length $n$. Moreover, we call $|\lambda|:=\sum_{i=1}^{\ell(\lambda)} \lambda_i$ the \textit{weight} of the partition. For any $\lambda$ we have a \textit{conjugate partition} $\lambda'$ whose parts are defined as $\lambda'_i := |\{j: \lambda_j \geq i\}|$. We finally define the notion of interlacing partitions. Let $\lambda$ and $\mu$ be two partitions with $|\lambda| \geq |\mu|$. They are said to be \textit{interlacing} and we write $\lambda \succ \mu$ if and only if
\begin{align*} 
\lambda_1 \geq \mu_1 \geq \lambda_2 \geq \mu_2 \geq \lambda_3 \geq \cdots
\end{align*}
In the language of \cite{mac}, the interlacing property is equivalent to saying that the skew diagram $\lambda-\mu$ forms a {\it horizontal strip,} meaning that 
$\lambda'_i-\mu'_i \leq 1$ for all $i \geq 1$.

\section{Proofs}
\label{sec:proofs}

The primary goal of this section is to prove equation \eref{HL-OSASM}, effectively by using the Izergin--Korepin technique familiar from quantum integrable models. As a warm-up, we begin by providing a new proof of \eref{HL-ASM} along such lines. This approach to proving \eref{HL-ASM} and \eref{HL-OSASM} may not be the most elegant (indeed, in the case of \eref{HL-ASM} there is a much simpler proof using Macdonald difference operators -- see Section \ref{sec:mac}), but it is powerful since it only assumes two standard properties of Hall--Littlewood polynomials: their branching rule, and a Pieri identity.

\subsection{Branching rule for Hall--Littlewood polynomials}

The branching rule allows a Hall--Littlewood polynomial 
$P_{\lambda}(x_1,\dots,x_n;t)$ to be written as a sum over Hall--Littlewood polynomials $P_{\mu}(x_1,\dots,x_{n-1};t)$ in a smaller alphabet. From the definition of skew Hall--Littlewood polynomials (see Section 5, Chapter III of \cite{mac}), one has 
\begin{align*}
P_{\lambda}(x_1,\dots,x_n;t)
=
\sum_{\mu}
P_{\lambda/\mu}(x_n;t)
P_{\mu}(x_1,\dots,x_{n-1};t).
\end{align*}
Since the skew Hall--Littlewood polynomial $P_{\lambda/\mu}(x_n;t)$ in a single variable satisfies
\begin{align*}
P_{\lambda/\mu}(x_n;t)
=
\psi_{\lambda/\mu}(t)
x_n^{|\lambda-\mu|},
\end{align*}
the branching rule can be expressed as
\begin{align}
\label{branch}
P_{\lambda}(x_1,\dots,x_n;t)
=
\sum_{\mu}
\psi_{\lambda/\mu}(t)
x_n^{|\lambda-\mu|}
P_{\mu}(x_1,\dots,x_{n-1};t),
\end{align}
where the function $\psi_{\lambda / \mu}(t)$ is given by\footnote{We have departed slightly from the convention of \cite{mac} for the function $\psi_{\lambda/\mu}(t)$, by incorporating the Kronecker delta into its definition (so that it is defined for all partitions $\lambda$, $\mu$). This turns out to be convenient in many of the equations that follow, since it avoids keeping track of interlacing conditions when writing sums.}
\begin{align*}
\psi_{\lambda/\mu}(t)
=
\delta_{\lambda \succ \mu}
\left(
\prod_{\substack{i \geq 1 \\ m_i(\mu) = m_i(\lambda)+1}}
(1-t^{m_i(\mu)})
\right)
\end{align*}
with $\psi_{\lambda/\mu}(t) = 0$ unless $\lambda \succ \mu$. In the sequel we will often find it convenient to rephrase \eref{branch} in terms of horizontal strips, by writing
\begin{align}
\label{branch2}
P_{\lambda}(x_1,\dots,x_n;t)
=
\sum_{\mu:\mu' = \lambda' - \epsilon} 
\prod_{\substack{i \geq 1 \\ \epsilon_i=0 \\ \epsilon_{i+1}=1}}
(1-t^{m_i(\mu)})
x_n^{|\lambda-\mu|}
P_{\mu}(x_1,\dots,x_{n-1};t),
\end{align} 
where the notation $\mu' = \lambda' - \epsilon$ means that $\mu'_i  = \lambda'_i - \epsilon_i$ for all $i \geq 1$, for some $\epsilon_i \in \{0,1\}$. 

\subsection{A Pieri identity for Hall--Littlewood polynomials}

Pieri rules allow one to compute the product of a fundamental symmetric function (such as complete symmetric functions, or elementary symmetric functions) and a more advanced symmetric function (such as Schur, Hall--Littlewood, or Macdonald polynomials). Several types of Pieri rules exist for Hall--Littlewood polynomials, but in the section we will only make use of the identity
\begin{align}
\label{pieri}
e_k(x_1,\dots,x_n)
P_{\mu}(x_1,\dots,x_n;t)
=
\sum_{\lambda:|\lambda - \mu| = k}
\psi'_{\lambda/\mu}(t)
P_{\lambda}(x_1,\dots,x_n;t),
\end{align}
where $e_k(x_1,\dots,x_n)$ is an elementary symmetric function (see Section 2, Chapter I of \cite{mac}) and $\psi'_{\lambda/\mu}(t)$ is given by
\begin{align*}
\psi'_{\lambda/\mu}(t)
=
\prod_{i = 1}^{\infty}
\binomQ{\lambda'_i - \lambda'_{i+1}}{\lambda'_i - \mu'_i}_t
=
\prod_{i = 1}^{\infty}
\binomQ{m_i(\lambda)}{\lambda'_i - \mu'_i}_t
\end{align*}
with the $t$-binomial coefficient being defined as
\begin{align*}
\binomQ{a}{b}_t
=
\frac{(1-t) \dots (1-t^a)}{(1-t) \dots (1-t^b).(1-t) \dots (1-t^{a-b})},
\quad
\forall\ 0 \leq b \leq a,
\quad
a,b \in \mathbb{Z},
\end{align*}
and $\binomQ{a}{b}_t = 0$ otherwise. We remark that the sum in \eref{pieri} can be considered to be taken over $\lambda$ such that $\lambda - \mu$ is a vertical strip (or equivalently, such that $\lambda' \succ \mu'$), since the coefficients $\psi'_{\lambda/\mu}(t)$ vanish when this is not the case.

\subsection{Proof of equation \eref{HL-ASM}}
\label{proof-ASM}

In this subsection we prove Theorem \ref{cauch-thm}, which is an alternative statement of equation \eref{HL-ASM}. Our strategy is to show that the left hand side of 
\eref{HL-ASM} satisfies four conditions, which are obvious properties of the right hand side (they are the usual four properties in the Izergin--Korepin approach to calculating the domain wall partition function). Since these conditions are uniquely-determining, it follows that the two sides of \eref{HL-ASM} must be equal.

\begin{thm}
\label{cauch-thm}
Define the function
\begin{align}
\label{F_n}
\mathcal{F}_n(x_1,\dots,x_n)
=
\sum_{\lambda}
w_{\lambda}(n,t)
P_{\lambda}(x_1,\dots,x_n;t)
P_{\lambda}(y_1,\dots,y_n;t),
\end{align}
where for convenience we have set
\begin{align}
w_{\lambda}(n,t)
=
\prod_{i=0}^{\infty}
\prod_{j=1}^{m_i(\lambda)}
(1-t^j),
\end{align}
and where the dependence on the variables $\{y_1,\dots,y_n\}$ and $t$ has been intentionally suppressed. $\mathcal{F}_n(x_1,\dots,x_n)$ satisfies the following properties:
\begin{enumerate}[label=\bf\arabic*.]

\item $\mathcal{F}_n(x_1,\dots,x_n)$ is symmetric in $\{x_1,\dots,x_n\}$.

\item The renormalized function 
$\prod_{i,j=1}^{n} (1-x_i y_j) \mathcal{F}_n(x_1,\dots,x_n)$ is a polynomial in $x_n$ of degree $n-1$.

\item Setting $x_n = 1/(t y_n)$, one obtains the recursion
\begin{align*}
\mathcal{F}_n
\Big|_{x_n = 1/(t y_n)}
=
-t^n
\mathcal{F}_{n-1}(x_1,\dots,x_{n-1}).
\end{align*}

\item $\mathcal{F}_1(x_1) = (1-t)/(1-x_1 y_1)$.

\end{enumerate}
\end{thm}
Since $\mathcal{F}_n(x_1,\dots,x_n)$ is a sum of Hall--Littlewood polynomials, each being symmetric in $\{x_1,\dots,x_n\}$, property {\bf 1} is immediate. The remaining properties {\bf 2}--{\bf 4} will be proved in Sections \ref{1-poly}--\ref{1-ic}, after we make a preliminary observation about the function $\mathcal{F}_n(x_1,\dots,x_n)$ in Section \ref{1-alt}.

\subsubsection{Alternative form for $\mathcal{F}_n(x_1,\dots,x_n)$}
\label{1-alt}

Let $\lambda$ be a length $n$ partition, and denote by $\lambda^{*}$ the partition obtained by removing the entire first column from the Young diagram of $\lambda$, \ie $\lambda^{*} = (\lambda_1-1,\dots,\lambda_n-1)$. Then one has the following identity between Hall--Littlewood polynomials:
\begin{align}
\label{imptt-id}
P_{\lambda}(x_1,\dots,x_n;t)
=
(x_1 \cdots x_n)
P_{\lambda^{*}}(x_1,\dots,x_n;t).
\end{align}
Since the function $w_{\lambda}(n,t)$ is effectively the same as the function 
$b_{\lambda}(t) := \prod_{i=1}^{\infty} \prod_{j=1}^{m_i(\lambda)} (1-t^j)$ 
(except that it treats parts of size zero as though they were of non-zero size), by using \eref{imptt-id} it follows immediately that
\begin{align}
\label{alt-form-cauchy}
(x_1 \cdots x_n) (y_1 \cdots y_n)
\mathcal{F}_n(x_1,\dots,x_n)
=
\sum_{\lambda: \ell(\lambda) = n}
b_{\lambda}(t)
P_{\lambda}(x_1,\dots,x_n;t)
P_{\lambda}(y_1,\dots,y_n;t).
\end{align}
This alternative way of writing $\mathcal{F}_n(x_1,\dots,x_n)$ will prove to be helpful in establishing the polynomiality property {\bf 2} and the recursive property {\bf 3}.

\subsubsection{Polynomiality}
\label{1-poly}

We begin by considering the Cauchy identity for Hall--Littlewood polynomials (see Section 4, Chapter III of \cite{mac}). While it is standard to write this identity so that the right hand side is a rational function in $\{x_1,\dots,x_m\}$ and $\{y_1,\dots,y_n\}$, here we adopt a polynomial normalization:
\begin{align}
\label{hl-cauchy}
\prod_{i=1}^{m}
\prod_{j=1}^{n}
(1-x_i y_j)
\sum_{\lambda}
b_{\lambda}(t)
P_{\lambda}(x_1,\dots,x_m;t)
P_{\lambda}(y_1,\dots,y_n;t)
=
\prod_{i=1}^{m}
\prod_{j=1}^{n}
(1-t x_i y_j).
\end{align}
Thanks to this identity, we deduce that the sum on the left hand side is in fact a polynomial in $x_m$ of degree $n$. To make full use of this fact, we now seek to rearrange the left hand side so that the dependence on $x_m$ is fully extracted. By applying both the branching rule \eref{branch} and the Pieri identity \eref{pieri}, we find that
\begin{align*}
&
\prod_{i=1}^{m}
\prod_{j=1}^{n}
(1-x_i y_j)
\sum_{\lambda}
b_{\lambda}(t)
P_{\lambda}(\vec{x}_m;t)
P_{\lambda}(\vec{y}_n;t)
\\
&
=
\prod_{i=1}^{m-1}
\prod_{j=1}^{n}
(1-x_i y_j)
\sum_{k=0}^{n}
e_{k}(\vec{y}_n)
(-x_m)^k
\sum_{\lambda}
\ 
\sum_{\mu}
b_{\lambda}(t)
\psi_{\lambda/\mu}(t) x_m^{|\lambda-\mu|}
P_{\mu}(\vec{x}_{m-1};t)
P_{\lambda}(\vec{y}_n;t)
\\
&
=
\prod_{i=1}^{m-1}
\prod_{j=1}^{n}
(1-x_i y_j)
\sum_{\lambda}
\ 
\sum_{\mu}
\ 
\sum_{\nu}
b_{\lambda}(t)
\psi_{\lambda/\mu}(t)
\psi'_{\nu/\lambda}(t)
(-1)^{|\nu-\lambda|} 
x_m^{|\nu-\mu|}
P_{\mu}(\vec{x}_{m-1};t)
P_{\nu}(\vec{y}_n;t), 
\end{align*}
where we have used the generating series $\prod_{i=1}^{n} (1+y_i z) = \sum_{k=0}^{n} e_k(y_1,\dots,y_n) z^k$ for the elementary symmetric polynomials to produce the first equality. From the linear independence of the Hall--Littlewood polynomials, the fact that the previous expression is polynomial in $x_m$ of degree $n$ means that:
\newline

\noindent
\fbox{
\begin{minipage}{6.3in}
\begin{equation}
\label{1-final-sum}
\sum_{\lambda}
(-1)^{|\lambda|}
b_{\lambda}(t)
\psi_{\lambda/\mu}(t)
\psi'_{\nu/\lambda}(t)
=
0
\end{equation}
for all partitions $\mu$ of length $\ell(\mu) \leq m-1$ and $\nu$ of length 
$\ell(\nu) \leq n$, such that $|\nu-\mu| > n$.
\end{minipage}
}

\medskip
Of course the value of $m$ is arbitrary, so one can state 
\eref{1-final-sum} with no constraint imposed on $\mu$.

\begin{rmk}
Equation \eqref{1-final-sum} is a special case of the following more general formula\footnote{We thank O. Warnaar for bringing equations \eref{q-pfaff} and \eref{q-pfaff-restrict} to our attention.}:
\begin{align}
\label{q-pfaff}
\sum_{\lambda:\mu \subseteq \lambda \subseteq \nu} 
(-1)^{|\lambda|}
b_{\lambda}(t) 
\psi_{\lambda / \mu}(t) 
\psi'_{\nu / \lambda}(t) 
=
(-1)^{|\mu|} 
t^{|\nu-\mu|}
b_{\mu}(t) 
\psi'_{\nu / \mu} (t).
\end{align}
Indeed, if $\ell(\nu) \leq n$ then when $|\nu - \mu| > n$ it is not possible for $\nu - \mu$ to be a vertical strip, meaning that the right hand side of \eref{q-pfaff} vanishes. Equation \eref{q-pfaff} follows from the (multivariate) $q$-Pfaff-Saalsch\"{u}tz-Rains Macdconald summation formula (Corollary 4.9 of \cite{rai} with $P$ replaced by $Q$) with (in the notation of that paper): 
\begin{align*}
b=a/q,\ \ c=bt,\ \ a \to 0,\ \ q \to 0.
\end{align*}
The limits are taken in the prescribed order after making the substitutions and using the homogeneity of Macdonald polynomials to cancel all powers of $a$ and $q$ appearing. For details about the simplifications that occur to obtain \eqref{q-pfaff}, see \cite{war2}.  
\end{rmk}

\medskip
Returning to the proof of the polynomiality property {\bf 2}, using the identity 
\eref{alt-form-cauchy} it is sufficient to show that
\begin{align*}
\prod_{i,j=1}^{n}
(1-x_i y_j)
\sum_{\lambda : \ell(\lambda) = n}
b_{\lambda}(t)
P_{\lambda}(x_1,\dots,x_n;t)
P_{\lambda}(y_1,\dots,y_n;t)
\end{align*}
is a polynomial in $x_n$ of degree $n$. The similarity between this sum and the sum appearing on the left hand side of the Cauchy identity \eref{hl-cauchy} is apparent: indeed the only difference is the constraint $\ell(\lambda)=n$, and the fact that we now consider the case $m=n$. Hence by following the same steps as those outlined above (modulo length constraints which are now imposed on the partitions being summed over) we see find that property {\bf 2} is equivalent to proving that: 
\newline

\noindent
\fbox{
\begin{minipage}{6.3in}
\begin{equation}
\label{1-final-sum-res}
\sum_{\lambda : \ell(\lambda) = n}
(-1)^{|\lambda|}
b_{\lambda}(t)
\psi_{\lambda/\mu}(t)
\psi'_{\nu/\lambda}(t)
=
0
\end{equation}
for all partitions $\mu$ of length $\ell(\mu) = n-1$ and $\nu$ of length $\ell(\nu) = n$, such that 
$|\nu-\mu| > n$.
\end{minipage}
}

\medskip
Let us define the sums
\begin{align*}
\mathcal{S}_{\leq n}(\mu,\nu)
=
\sum_{\lambda: \ell(\lambda) \leq n}
(-1)^{|\lambda|}
b_{\lambda}(t)
\psi_{\lambda/\mu}(t)
\psi'_{\nu/\lambda}(t),
\quad\quad
\mathcal{S}_{=n}(\mu,\nu)
=
\sum_{\lambda: \ell(\lambda) = n}
(-1)^{|\lambda|}
b_{\lambda}(t)
\psi_{\lambda/\mu}(t)
\psi'_{\nu/\lambda}(t),
\end{align*}
where we fix partitions $\mu$, $\nu$ that satisfy $\ell(\mu)=n-1$, $\ell(\nu)=n$, and $|\nu-\mu| > n$. We can clearly write
\begin{align}
\label{triv-eq}
\mathcal{S}_{=n}(\mu,\nu)
=
\mathcal{S}_{\leq n}(\mu,\nu)
-
\mathcal{S}_{\leq n-1}(\mu,\nu),
\end{align}
where we already know that $\mathcal{S}_{\leq n}(\mu,\nu) = 0$ using equation \eref{1-final-sum}. It remains only to show that
\begin{align}
\label{S_(n-1)}
\mathcal{S}_{\leq n-1}(\mu,\nu)
=
\sum_{\lambda: \ell(\lambda) \leq n-1}
(-1)^{|\lambda|}
b_{\lambda}(t)
\psi_{\lambda/\mu}(t)
\psi'_{\nu/\lambda}(t)
\end{align}
vanishes. If the final part of $\nu$ satisfies $\nu_n > 1$, then \eref{S_(n-1)} is zero (since 
$\nu-\lambda$ will never be a vertical strip). Hence we restrict our attention to the case 
$\nu = \tilde{\nu} \cup 1$, where $\ell(\tilde{\nu}) = n-1$. Furthermore, since 
$\ell(\mu) = n-1$, the only partitions $\lambda$ which give a non-zero contribution are those such that $\ell(\lambda) = n-1$ (otherwise $\lambda - \mu$ is not a horizontal strip). Hence all non-zero $\psi'_{\nu/\lambda}(t)$ in \eref{S_(n-1)} satisfy
\begin{align*}
\psi'_{\nu/\lambda}(t)
=
\binomQ{m_1(\nu)}{1}_t
\psi'_{\tilde{\nu}/\lambda}(t),
\end{align*}
and we obtain
\begin{align*}
\mathcal{S}_{\leq n-1}(\mu,\nu)
=
\binomQ{m_1(\nu)}{1}_t
\sum_{\lambda: \ell(\lambda) \leq n-1}
(-1)^{|\lambda|}
b_{\lambda}(t)
\psi_{\lambda/\mu}(t)
\psi'_{\tilde{\nu}/\lambda}(t).
\end{align*}
But this final sum is zero, using \eqref{1-final-sum} (since $|\tilde{\nu} - \mu| > n-1$). So everything on the right hand side of \eqref{triv-eq} is zero, which proves \eref{1-final-sum-res}.  

\begin{rmk}
As before, we comment that \eref{1-final-sum-res} is a special case of a more general equation:
\begin{align}
\label{q-pfaff-restrict}
\sum_{\substack{\lambda:\ell(\lambda) = \ell(\nu) \\ \mu \subseteq \lambda \subseteq \nu}} 
(-1)^{|\lambda|}
b_{\lambda}(t) 
\psi_{\lambda / \mu}(t) 
\psi'_{\nu / \lambda}(t) 
=
\left\{
\begin{array}{ll}
(-1)^{|\mu|} 
t^{|\nu-\mu|}
b_{\mu}(t) 
\psi'_{\nu / \mu} (t),
&
\quad
\ell(\mu) = \ell(\nu),
\\ \\
(-1)^{|\mu|} 
t^{|\nu-\mu|}
(1-t^{-1})
b_{\mu}(t) 
\psi'_{\nu / \mu} (t),
&
\quad
\ell(\mu) = \ell(\nu)-1,
\end{array}
\right.
\end{align}
with all other cases being trivially zero.

\end{rmk}

\subsubsection{Recursion relation}
\label{1-rec}

Applying the branching rule \eref{branch2} to both $P_{\lambda}(x_1,\dots,x_n;t)$ and 
$P_{\lambda}(y_1,\dots,y_n;t)$, equation \eref{alt-form-cauchy} becomes
\begin{align*}
&
\prod_{i=1}^{n} (x_i y_i)
\mathcal{F}_n(x_1,\dots,x_n)
=
\\
&
\sum_{\lambda: \ell(\lambda) = n}
\
\sum_{\substack{\mu: \ell(\mu) = n-1 \\ \mu' = \lambda' - \delta}}
\
\sum_{\substack{\nu: \ell(\nu) = n-1 \\ \nu' = \lambda' - \epsilon}}
b_{\lambda}(t)
\psi_{\lambda/\mu}(t)
\psi_{\lambda/\nu}(t)
x_n^{|\lambda-\mu|}
y_n^{|\lambda-\nu|}
P_{\mu}(\vec{x}_{n-1};t)
P_{\nu}(\vec{y}_{n-1};t)
=
\\
&
\sum_{\lambda: \ell(\lambda) = n}
\
\sum_{\substack{\mu: \ell(\mu) = n-1 \\ \mu' = \lambda' - \delta}}
\
\sum_{\substack{\nu: \ell(\nu) = n-1 \\ \nu' = \lambda' - \epsilon}}
b_{\lambda}(t)
\prod_{\substack{\delta_i = 0 \\ \delta_{i+1} = 1}}
(1-t^{m_i(\mu)})
\prod_{\substack{\epsilon_j = 0 \\ \epsilon_{j+1} = 1}}
(1-t^{m_j(\nu)})
x_n^{|\lambda-\mu|}
y_n^{|\lambda-\nu|}
P_{\mu}(\vec{x}_{n-1};t)
P_{\nu}(\vec{y}_{n-1};t)
.
\end{align*}
Setting $x_n = 1/(t y_n)$, we obtain
\begin{align*}
&
\prod_{i=1}^{n-1}
(x_i y_i)
t^{-1}
\mathcal{F}_n
\Big|_{x_n = 1/(t y_n)}
=
\\
&
\sum_{\lambda: \ell(\lambda) = n}
\
\sum_{\substack{\mu: \ell(\mu) = n-1 \\ \mu' = \lambda' - \delta}}
\
\sum_{\substack{\nu: \ell(\nu) = n-1 \\ \nu' = \lambda' - \epsilon}}
b_{\lambda}(t)
\prod_{\substack{\delta_i = 0 \\ \delta_{i+1} = 1}}
(1-t^{m_i(\mu)})
\prod_{\substack{\epsilon_j = 0 \\ \epsilon_{j+1} = 1}}
(1-t^{m_j(\nu)})
t^{|\mu-\lambda|}
y_n^{|\mu-\nu|}
P_{\mu}(\vec{x}_{n-1};t)
P_{\nu}(\vec{y}_{n-1};t).
\end{align*}
We isolate the coefficient of $P_{\mu}(x_1,\dots,x_{n-1};t) P_{\nu}(y_1,\dots,y_{n-1};t) y_n^{|\mu-\nu|}$ in the preceding expression, and denote it by $\mathcal{C}(\mu,\nu)$:
\begin{align*}
\mathcal{C}(\mu,\nu)
=
\sum_{\substack{\lambda: \ell(\lambda) = n 
\\ \lambda' = \mu' + \delta \\ \lambda' = \nu' + \epsilon}}
t^{-|\delta|}
b_{\lambda}(t)
\prod_{\substack{\delta_i = 0 \\ \delta_{i+1} = 1}}
(1-t^{m_i(\mu)})
\prod_{\substack{\epsilon_j = 0 \\ \epsilon_{j+1} = 1}}
(1-t^{m_j(\nu)}).
\end{align*}
To prove the required recursion relation, we wish to show that
\begin{align*}
\mathcal{C}(\mu,\nu)
=
\left\{
\begin{array}{ll}
-t^{n-1} b_{\mu}(t),
\quad
&
\mu = \nu,
\\
0,
\quad
&
\mu \not= \nu.
\end{array}
\right.
\end{align*}
We start by considering the diagonal case $\mu = \nu$, for which
\begin{align*}
\mathcal{C}(\mu,\mu)
=
\sum_{j=1}^{\infty}
\sum_{\delta_j \in \{0,1\}}
t^{-|\delta|}
b_{\lambda}(t)
\prod_{\substack{\delta_k = 0 \\ \delta_{k+1} = 1}}
(1-t^{m_k(\mu)})^2
\end{align*}
where we now sum over all $\delta_j \in \{0,1\}$, with $\lambda$ given by $\lambda' = \mu' + \delta$. At first it seems that we must exclude the possibility $\delta_j = 0, \delta_{j+1} = 1$ when $\mu'_j = \mu'_{j+1}$, since this gives rise to $\lambda$ which is not a partition. In fact we can ignore this constraint entirely, since $\mu'_j = \mu'_{j+1}$ implies that $m_j(\mu) = 0$ and the term $(1-t^{m_j(\mu)})$ vanishes, meaning 
$\delta_j=0,\delta_{j+1}=1$ gives no contribution to the summation in any case. We define the partial coefficients
\begin{align}
\label{part-coeff1}
\mathcal{C}_{i,\delta_i}(\mu)
=
\sum_{j=1}^{i-1}
\sum_{\delta_j \in \{0,1\}}
t^{-\sum_{k=1}^{i} \delta_k}
\prod_{k=1}^{i-1}
\prod_{l=1}^{m_k(\lambda)}
(1-t^l)
\prod_{\substack{1 \leq k \leq i-1 \\ \delta_k = 0 \\ \delta_{k+1} = 1}}
(1-t^{m_k(\mu)})^2,
\end{align}
where $\delta_i$ can be either 0 or 1, and $\lambda$ is the length $n$ partition formed by taking $\lambda'_j=\mu'_j+\delta_j$ for all $1 \leq j \leq i$, $\lambda'_j = \mu'_j$ for all $j > i$. We proceed to establish some recurrence relations which these coefficients obey. Consider the coefficient $\mathcal{C}_{i+1,\delta_{i+1}}(\mu)$ in the case 
$\delta_{i+1} = 0$, and perform the summation on $\delta_i$ explicitly. This produces the recurrence
\begin{align}
\label{recur-1}
\mathcal{C}_{i+1,0}(\mu)
&=
\prod_{j=1}^{m_i(\mu)}
(1-t^j)
\left[
\mathcal{C}_{i,0}(\mu)
+
(1-t^{m_i(\mu)+1})
\mathcal{C}_{i,1}(\mu)
\right]
\end{align}
valid for all $i \geq 1$, with initial values $\mathcal{C}_{1,0}(\mu) = 0$ ($\delta_1 = 0$ is forbidden, because it would lead to $\ell(\lambda) = n-1$) and 
$\mathcal{C}_{1,1}(\mu) = t^{-1}$. Similarly, the $\delta_{i+1} = 1$ case of 
$\mathcal{C}_{i+1,\delta_{i+1}}(\mu)$ gives rise to the recurrence
\begin{align}
\label{recur-2}
t
\mathcal{C}_{i+1,1}(\mu)
&=
\prod_{j=1}^{m_i(\mu)}
(1-t^j)
\left[
(1-t^{m_i(\mu)})
\mathcal{C}_{i,0}(\mu)
+
\mathcal{C}_{i,1}(\mu)
\right]
\end{align}
valid for all $i \geq 1$, where we have again summed over both possible values of 
$\delta_i$ to obtain the result. Since $m_i(\mu) = 0$ for all $i > \mu_1$, the recurrence relations \eref{recur-1} and \eref{recur-2} eventually stabilize:
\begin{align*}
\mathcal{C}_{j,0}(\mu)
=
\mathcal{C}_{i,0}(\mu)
+
(1-t)
\sum_{k=i}^{j-1}
\mathcal{C}_{k,1}(\mu),
\quad\quad
\mathcal{C}_{j,1}(\mu)
=
t^{i-j}
\mathcal{C}_{i,1}(\mu),
\quad\quad
\forall\ j > i > \mu_1. 
\end{align*}
It follows immediately that the quantity that we wish to compute, $\mathcal{C}(\mu,\mu)$, is given by
\begin{align*}
\mathcal{C}(\mu,\mu)
=
\lim_{i \rightarrow \infty}
\mathcal{C}_{i,0}(\mu)
=
\mathcal{C}_{\mu_1+1,0}(\mu)
-
t
\mathcal{C}_{\mu_1+1,1}(\mu). 
\end{align*}
For this reason, we now consider the linear combination of coefficients 
$\mathcal{C}_{i,0}(\mu) - t \mathcal{C}_{i,1}(\mu)$. Subtracting equation \eref{recur-2} from \eref{recur-1}, we find that this linear combination satisfies the elementary recurrence
\begin{align}
\label{recur-final}
\mathcal{C}_{i+1,0}(\mu) - t \mathcal{C}_{i+1,1}(\mu)
=
t^{m_i(\mu)}
\prod_{j=1}^{m_i(\mu)}
(1-t^j)
\Big[
\mathcal{C}_{i,0}(\mu) 
- 
t 
\mathcal{C}_{i,1}(\mu)
\Big] 
\end{align}
with initial condition $\mathcal{C}_{1,0}(\mu) - t \mathcal{C}_{1,1}(\mu) = -1$. Solving the recurrence \eref{recur-final}, we find that
\begin{align*}
\mathcal{C}_{\mu_1+1,0}(\mu)
-
t
\mathcal{C}_{\mu_1+1,1}(\mu)
=
-t^{\sum_{i=1}^{\infty} m_i(\mu)}
\prod_{i=1}^{\infty}
\prod_{j=1}^{m_i(\mu)}
(1-t^j)
=
-t^{n-1}
b_{\mu}(t), 
\end{align*}
where we have used the fact that $m_i(\mu) = 0$ for all $i > \mu_1$ to produce the first equality, and the fact that the sum of all the multiplicities in $\mu$ is equal to $n-1$ to produce the second. This completes the proof of the diagonal case $\mu = \nu$.

In the non-diagonal case $\mu \not= \nu$, we follow a similar procedure to that outlined above. We wish to calculate
\begin{align*}
\mathcal{C}(\mu,\nu)
=
\sum_{j=1}^{\infty}
\sum_{\substack{\delta_j \in \{0,1\} \\ \epsilon_j \in \{0,1\}}}
t^{-|\delta|}
b_{\lambda}(t)
\prod_{\substack{\delta_k = 0 \\ \delta_{k+1} = 1}}
(1-t^{m_k(\mu)})
\prod_{\substack{\epsilon_k = 0 \\ \epsilon_{k+1} = 1}}
(1-t^{m_k(\nu)}),
\end{align*}
where $\lambda$ is the length $n$ partition given by 
$\lambda' = \mu' + \delta = \nu' + \epsilon$. Since 
$\delta_i, \epsilon_i \in \{0,1\}$ for all $i$, it follows that $|\mu'_i-\nu'_i| \leq 1$ for all $i$, or else $\mathcal{C}(\mu,\nu)$ is necessarily zero. We define a sequence of partial coefficients
\begin{align*}
\mathcal{C}_{i,\delta_i,\epsilon_i}(\mu,\nu)
=
\sum_{j=1}^{i-1}
\sum_{\substack{\delta_j \in \{0,1\} \\ \epsilon_j \in \{0,1\}}}
t^{-\sum_{k=1}^{i} \delta_k}
\prod_{k=1}^{i-1}
\prod_{l=1}^{m_k(\lambda)}
(1-t^l)
\prod_{\substack{1 \leq k \leq i-1 \\ \delta_k = 0 \\ \delta_{k+1} = 1}}
(1-t^{m_k(\mu)})
\prod_{\substack{1 \leq k \leq i-1 \\ \epsilon_k = 0 \\ \epsilon_{k+1} = 1}}
(1-t^{m_k(\nu)}),
\end{align*}
where both $\delta_i$ and $\epsilon_i$ take some value in $\{0,1\}$. We let $\ihat$ denote the largest $i$ such that $|\mu'_i-\nu'_i| = 1$, \ie $\mu'_i = \nu'_i$ for all $i > \ihat$. Then either 
$\delta_{\ihat} = 1, \epsilon_{\ihat} = 0$ or $\delta_{\ihat} = 0, \epsilon_{\ihat} = 1$, and the summation is constrained by $\delta_i = \epsilon_i$ for all $i > \ihat$. Given that the summation indices are forced in this way, it is easy to deduce the recurrences
\begin{align}
\label{non-diag-rec1}
\mathcal{C}_{\ihat+1,0,0}(\mu,\nu)
=
\prod_{j=1}^{m_{\ihat}(\nu)}
(1-t^j)
\mathcal{C}_{\ihat,1,0}(\mu,\nu),
\quad
t
\mathcal{C}_{\ihat+1,1,1}(\mu,\nu)
=
\prod_{j=1}^{m_{\ihat}(\nu)}
(1-t^j)
\mathcal{C}_{\ihat,1,0}(\mu,\nu),
\quad
\text{when}\ 
\delta_{\ihat} = 1, \epsilon_{\ihat} = 0,
\\
\label{non-diag-rec2}
\mathcal{C}_{\ihat+1,0,0}(\mu,\nu)
=
\prod_{j=1}^{m_{\ihat}(\mu)}
(1-t^j)
\mathcal{C}_{\ihat,0,1}(\mu,\nu),
\quad
t
\mathcal{C}_{\ihat+1,1,1}(\mu,\nu)
=
\prod_{j=1}^{m_{\ihat}(\mu)}
(1-t^j)
\mathcal{C}_{\ihat,0,1}(\mu,\nu),
\quad
\text{when}\ 
\delta_{\ihat} = 0, \epsilon_{\ihat} = 1,
\end{align}
while for $i > \ihat$, we recover the same recurrences already obtained when considering the diagonal case $\mu=\nu$:
\begin{align*}
\mathcal{C}_{i+1,0,0}(\mu,\nu)
&=
\prod_{j=1}^{m_i(\mu)}
(1-t^j)
\left[
\mathcal{C}_{i,0,0}(\mu,\nu)
+
(1-t^{m_i(\mu)+1})
\mathcal{C}_{i,1,1}(\mu,\nu)
\right],
\\
t
\mathcal{C}_{i+1,1,1}(\mu,\nu)
&=
\prod_{j=1}^{m_i(\mu)}
(1-t^j)
\left[
(1-t^{m_i(\mu)})
\mathcal{C}_{i,0,0}(\mu,\nu)
+
\mathcal{C}_{i,1,1}(\mu,\nu)
\right]
.
\end{align*}
Hence by applying precisely the same reasoning as above, we conclude that
\begin{align}
\mathcal{C}(\mu,\nu)
=
\lim_{i \rightarrow \infty}
\mathcal{C}_{i,0,0}(\mu,\nu)
=
\mathcal{C}_{M+1,0,0}(\mu,\nu)
-
t
\mathcal{C}_{M+1,1,1}(\mu,\nu), 
\end{align}
where $M = \max(\mu_1,\nu_1)$, to cater for the case where these may be different. Hence we are again required to compute $\mathcal{C}_{i,0,0}(\mu,\nu) - t \mathcal{C}_{i,1,1}(\mu,\nu)$, which can be done via a recurrence of the form \eref{recur-final}. However, in contrast to the above, the initial condition of this recurrence is $\mathcal{C}_{\ihat+1,0,0}(\mu,\nu) - t \mathcal{C}_{\ihat+1,1,1}(\mu,\nu) = 0$ (by virtue of \eref{non-diag-rec1} and \eref{non-diag-rec2}). Because of this trivial initial condition, it follows that 
$\mathcal{C}_{i,0,0}(\mu,\nu) - t \mathcal{C}_{i,1,1}(\mu,\nu) = 0$ for all $i>\ihat$, which is what we aimed to show.

\subsubsection{Initial condition}
\label{1-ic}

In the case $n=1$, the Hall--Littlewood polynomials appearing in the sum \eref{F_n} depend on a single variable. Hence the partitions in the sum are constrained by 
$\ell(\lambda) \leq 1$. It follows that
\begin{align}
\mathcal{F}_1(x_1)
=
\sum_{k=0}^{\infty}
(1-t)
P_{(k)}(x_1;t)
P_{(k)}(y_1;t)
=
(1-t)
\sum_{k=0}^{\infty}
x_1^k y_1^k
=
\frac{1-t}{1-x_1 y_1}.
\end{align}

\subsection{Proof of equation \eref{HL-OSASM}}
\label{proof-OSASM}

In this subsection we prove Theorem \ref{littlewood-thm}, which is equivalent to proving equation \eref{HL-OSASM}. Similarly to the previous proof, we will show that the left hand side of \eref{HL-OSASM} satisfies four conditions, which are basic properties of the right hand side. Since these conditions only admit a unique solution, it follows that the two sides of \eref{HL-OSASM} must be equal.

\begin{thm}
\label{littlewood-thm}
Let $N=2n$. Define the function
\begin{align}
\label{G_N}
\mathcal{G}_{N}(x_1,\dots,x_N)
=
\sum_{\substack{\lambda\ {\rm with} \\ {\rm even\ columns}}} \ \ 
w_{\lambda}^{\rm el}(N,t) 
P_{\lambda}(x_1,\dots,x_N;t),
\end{align}
where for convenience we have set
\begin{align}
w_{\lambda}^{\rm el}(N,t)
= 
\prod_{i=0}^{\infty}\ \prod_{j=2,4,6,\dots}^{m_i(\lambda)} 
(1-t^{j-1}).
\end{align}
Then $\mathcal{G}_{N}(x_1,\dots,x_N)$ satisfies the following list of properties:
\begin{enumerate}[label=\bf\arabic*.]

\item $\mathcal{G}_{N}(x_1,\dots,x_N)$ is symmetric in $\{x_1,\dots,x_N\}$.

\item The renormalized function $\prod_{1 \leq i<j \leq N} (1-x_i x_j) \mathcal{G}_{N}(x_1,\dots,x_N)$ is a polynomial in $x_N$ of degree $N-2$.

\item Setting $x_N = 1/(t x_{N-1})$, one obtains the recursion
\begin{align*}
\mathcal{G}_N \Big|_{x_N = 1/(t x_{N-1})}
=
-t^{N-1}
\mathcal{G}_{N-2}(x_1,\dots,x_{N-2}).
\end{align*} 

\item $\mathcal{G}_2(x_1,x_2) = (1-t)/(1-x_1 x_2)$.
\end{enumerate}
\end{thm}
The property {\bf 1} is obvious, since $\mathcal{G}_{N}(x_1,\dots,x_N)$ is a sum of Hall--Littlewood polynomials and therefore manifestly symmetric in $\{x_1,\dots,x_N\}$. As we did in the proof of Theorem \ref{cauch-thm}, we begin with an alternative expression for $\mathcal{G}_{N}(x_1,\dots,x_N)$ in Section \ref{2-alt}, before proving the remaining properties {\bf 2}--{\bf 4} in Sections \ref{2-poly}--\ref{2-ic}.

\subsubsection{Alternative form for $\mathcal{G}_N(x_1,\dots,x_N)$}
\label{2-alt}

The function $w_{\lambda}^{\rm el}(N,t)$ bears close resemblance to the function 
$b^{\rm el}_{\lambda}(t)$, which usually appears in the Littlewood identity for 
Hall--Littlewood polynomials\footnote{The superscript in $w_{\lambda}^{\rm el}(N,t)$ and $b^{\rm el}_{\lambda}(t)$ is for {\it even legs}, since in the Macdonald case 
$b^{\rm el}_{\lambda}(q,t)$ is defined as a product over all boxes in $\lambda$ with even leg-length \cite{mac}.}:
\begin{align*}
b^{\rm el}_{\lambda}(t)
=
\prod_{i=1}^{\infty}\ 
\prod_{j=2,4,6,\dots}^{m_i(\lambda)} 
(1-t^{j-1}).
\end{align*}
The only difference between the two functions is that $w_{\lambda}^{\rm el}(N,t)$ considers parts of the partition $\lambda$ of size zero, whereas 
$b^{\rm el}_{\lambda}(t)$ does not. By using the identity \eref{imptt-id} once again, it follows that $\mathcal{G}_N(x_1,\dots,x_N)$ can be expressed alternatively as
\begin{align}
\label{alt-form-little}
(x_1 \cdots x_N)
\mathcal{G}_N(x_1,\dots,x_N)
=
\sum_{\substack{\lambda: \ell(\lambda) = N \\ \lambda'\ {\rm even}}}
\
b^{\rm el}_{\lambda}(t) 
P_{\lambda}(x_1,\dots,x_N;t).
\end{align}
This new way of writing $\mathcal{G}_N(x_1,\dots,x_N)$ is useful in establishing the polynomiality property {\bf 2}, as we will see below.

\subsubsection{Polynomiality property}
\label{2-poly}

We start by considering a renormalized version of the Littlewood identity for Hall--Littlewood polynomials (see Section 5, Chapter III of \cite{mac}):
\begin{align}
\label{littlewood-hl}
\prod_{1 \leq i<j \leq N}
(1-x_i x_j)
\sum_{\lambda'\ {\rm even}}
\
b^{\rm el}_{\lambda}(t) 
P_{\lambda}(x_1,\dots,x_N;t)
=
\prod_{1 \leq i<j \leq N}
(1-t x_i x_j).
\end{align}
We deduce that the left hand side of \eref{littlewood-hl} is a polynomial in $x_N$ of degree $N-1$, a fact which is only obvious given its equality with the right hand side. In what follows it will be useful to reformulate this fact, which we do by isolating the $x_N$ dependence:
\begin{align*}
&
\prod_{1 \leq i<j \leq N}
(1-x_i x_j)
\sum_{\lambda'\ {\rm even}}
\
b^{\rm el}_{\lambda}(t) 
P_{\lambda}(\vec{x}_N;t)
\\
&
=
\prod_{1 \leq i<j \leq N-1}
(1-x_i x_j)
\sum_{k=0}^{N-1}
e_k(\vec{x}_{N-1})
(-x_N)^k
\sum_{\lambda'\ {\rm even}}
\
\sum_{\mu}
b^{\rm el}_{\lambda}(t)
\psi_{\lambda/\mu}(t)
x_N^{|\lambda-\mu|}
P_{\mu}(\vec{x}_{N-1};t)
\\
&
=
\prod_{1 \leq i<j \leq N-1}
(1-x_i x_j)
\sum_{\lambda'\ {\rm even}}
\
\sum_{\mu}
\
\sum_{\nu}
(-1)^{|\nu-\mu|}
b^{\rm el}_{\lambda}(t)
\psi_{\lambda/\mu}(t)
\psi'_{\nu/\mu}(t)
x_N^{|\lambda-\mu|+|\nu-\mu|}
P_{\nu}(\vec{x}_{N-1};t),
\end{align*}
where the first equality follows from the definition of the elementary symmetric functions and the branching rule \eref{branch}, and the second equality is obtained from the Pieri identity \eref{pieri}. Now we notice that the summation over $\lambda$ is constrained trivially, since it imposes that $\lambda'$ is even and that $\lambda \succ \mu$. Indeed, given any partition $\mu$, there is a unique way of adding to it a horizontal strip such that the resulting partition has even-length columns. Hence we can write
\begin{align*}
&
\prod_{1 \leq i<j \leq N}
(1-x_i x_j)
\sum_{\lambda'\ {\rm even}}
\
b^{\rm el}_{\lambda}(t) 
P_{\lambda}(\vec{x}_N;t)
\\
&
=
\prod_{1 \leq i<j \leq N-1}
(1-x_i x_j)
\sum_{\mu}
\
\sum_{\nu}
(-1)^{|\nu-\mu|}
b^{\rm el}_{\lambda}(t)
\psi_{\lambda/\mu}(t)
\psi'_{\nu/\mu}(t)
x_N^{|\nu-\mu|+n_{\rm\tiny oc}(\mu)}
P_{\nu}(\vec{x}_{N-1};t),
\end{align*}
where $n_{\rm\tiny oc}(\mu)$ denotes the number of odd-length columns in the partition $\mu$, and $\lambda$ is hereinafter understood as the even-columned partition obtained by adding a horizontal strip to $\mu$. From the linear independence of the Hall--Littlewood polynomials (and eliminating any overall factors which play no role), we obtain at last our reformulation of the polynomiality statement:
\newline

\noindent
\fbox{
\begin{minipage}{6.3in}
\begin{equation}
\label{2-final-sum}
\sum_{\mu}
\
(-1)^{|\mu|}
b^{\rm el}_{\lambda}(t)
\psi_{\lambda/\mu}(t)
\psi'_{\nu/\mu}(t)
x^{|\nu-\mu|+n_{\rm\tiny oc}(\mu)}
\end{equation}
is a polynomial in $x$ of degree $N-1$, for all partitions $\nu$ of length $\ell(\nu) \leq N-1$.
\end{minipage}
}

\medskip

Coming back to the proof of property {\bf 2}, because of the identity \eref{alt-form-little} it suffices to show that
\begin{align*}
\prod_{1 \leq i<j \leq N}
(1-x_i x_j)
\sum_{\substack{\lambda:\ell(\lambda) = N \\ \lambda'\ {\rm even} }}
\
b^{\rm el}_{\lambda}(t) 
P_{\lambda}(x_1,\dots,x_N;t)
\end{align*}
is a polynomial in $x_N$ of degree $N-1$. The strong similarity between the preceding quantity and the left hand side of the Littlewood identity \eref{littlewood-hl} means that we can inherit information from the latter. Indeed, by following precisely the same arguments already outlined above (but paying attention to the new restriction $\ell(\lambda) = N$), property {\bf 2} is equivalent to the statement
\newline

\noindent
\fbox{
\begin{minipage}{6.3in}
\begin{equation}
\label{2-final-sum-res}
\sum_{\mu: \ell(\mu) = N-1}
\
(-1)^{|\mu|}
b^{\rm el}_{\lambda}(t)
\psi_{\lambda/\mu}(t)
\psi'_{\nu/\mu}(t)
x^{|\nu-\mu|+n_{\rm\tiny oc}(\mu)}
\end{equation}
is a polynomial in $x$ of degree $N-1$, for all partitions $\nu$ of length $\ell(\nu) = N-1$.
\end{minipage}
}

\medskip
Letting $\mathcal{P}_{N}$ denote the proposition \eref{2-final-sum-res}, we prove it by induction on $N$. Although we are only interested in the case where $N$ is even, one can prove it for generic $N$. The base case $N=1$ is trivial:
\begin{align*}
\sum_{\mu: \ell(\mu) = 0}
\
(-1)^{|\mu|}
b^{\rm el}_{\lambda}(t)
\psi_{\lambda/\mu}(t)
\psi'_{\nu/\mu}(t)
x^{|\nu-\mu|+n_{\rm\tiny oc}(\mu)}
=
1
\end{align*}
since the only possibility is for the partition $\nu$ to be empty. Now let $N > 1$ and assume that $\mathcal{P}_1,\dots,\mathcal{P}_{N-1}$ are all true. It is clearly possible to write
\begin{multline*}
\sum_{\mu: \ell(\mu) = N-1}
\
(-1)^{|\mu|}
b^{\rm el}_{\lambda}(t)
\psi_{\lambda/\mu}(t)
\psi'_{\nu/\mu}(t)
x^{|\nu-\mu|+n_{\rm\tiny oc}(\mu)}
=
\\
\left(
\sum_{\mu}
-
\sum_{k=0}^{N-2}
\sum_{\mu: \ell(\mu) = k}
\right)
\
(-1)^{|\mu|}
b^{\rm el}_{\lambda}(t)
\psi_{\lambda/\mu}(t)
\psi'_{\nu/\mu}(t)
x^{|\nu-\mu|+n_{\rm\tiny oc}(\mu)},
\end{multline*}
where the first sum is already known to give a polynomial in $x$ of degree $N-1$, thanks to \eref{2-final-sum}. As for the remaining sums, they only give a non-zero result when $\nu = \kappa \cup 1^{N-k-1}$, where $\kappa$ is a partition with length $\ell(\kappa)=k$. Under such circumstances, and with $\ell(\mu) = k$, we have
\begin{align*}
\psi'_{\nu/\mu}(t)
=
\binomQ{m_1(\nu)}{N-k-1}_t
\psi'_{\kappa/\mu}(t).
\end{align*}
This allows us to conclude that
\begin{multline*}
\sum_{\mu: \ell(\mu) = N-1}
\
(-1)^{|\mu|}
b^{\rm el}_{\lambda}(t)
\psi_{\lambda/\mu}(t)
\psi'_{\nu/\mu}(t)
x^{|\nu-\mu|+n_{\rm\tiny oc}(\mu)}
=
\\
\binomQ{m_1(\nu)}{N-k-1}_t
x^{N-k-1}
\sum_{\mu: \ell(\mu) = k}
(-1)^{|\mu|}
b^{\rm el}_{\lambda}(t)
\psi_{\lambda/\mu}(t)
\psi'_{\kappa/\mu}(t)
x^{|\kappa-\mu|+n_{\rm\tiny oc}(\mu)},
\end{multline*}
which is also a polynomial in $x$ of degree $N-1$ from the inductive assumption. Hence $\mathcal{P}_N$ is true, completing the proof.

\subsubsection{Recursion relation}
\label{2-rec}

Applying the branching rule \eref{branch2} twice to $P_{\lambda}(x_1,\dots,x_N;t)$, equation \eref{alt-form-little} becomes
\begin{align*}
&
\prod_{i=1}^{N} (x_i)
\mathcal{G}_N(x_1,\dots,x_N)
=
\\
&
\sum_{\substack{\lambda: \ell(\lambda) = N \\ \lambda'\ \text{even}}}
\
\sum_{\substack{\mu: \ell(\mu) = N-1 \\ \mu' = \lambda' - \delta}}
\
\sum_{\substack{\nu: \ell(\nu) = N-2 \\ \nu' = \mu' - \epsilon}}
b_{\lambda}^{\rm el}(t)
\psi_{\lambda/\mu}(t)
\psi_{\mu/\nu}(t)
x_N^{|\lambda-\mu|}
x_{N-1}^{|\mu-\nu|}
P_{\nu}(\vec{x}_{N-2};t)
=
\\
&
\sum_{\substack{\lambda: \ell(\lambda) = N \\ \lambda'\ \text{even}}}
\
\sum_{\substack{\mu: \ell(\mu) = N-1 \\ \mu' = \lambda' - \delta}}
\
\sum_{\substack{\nu: \ell(\nu) = N-2 \\ \nu' = \mu' - \epsilon}}
b_{\lambda}^{\rm el}(t)
\prod_{\substack{\delta_i = 0 \\ \delta_{i+1} = 1}}
(1-t^{m_i(\mu)})
\prod_{\substack{\epsilon_j = 0 \\ \epsilon_{j+1} = 1}}
(1-t^{m_j(\nu)})
x_N^{|\lambda-\mu|}
x_{N-1}^{|\mu-\nu|}
P_{\nu}(\vec{x}_{N-2};t)
.
\end{align*}
Setting $x_N = 1/(t x_{N-1})$, we find that
\begin{align*}
&
\prod_{i=1}^{N-2} (x_i)
t^{-1}
\mathcal{G}_N
\Big|_{x_N = 1/(t x_{N-1})}
=
\\
&
\sum_{\substack{\lambda: \ell(\lambda) = N \\ \lambda'\ \text{even}}}
\
\sum_{\substack{\mu: \ell(\mu) = N-1 \\ \mu' = \lambda' - \delta}}
\
\sum_{\substack{\nu: \ell(\nu) = N-2 \\ \nu' = \mu' - \epsilon}}
b_{\lambda}^{\rm el}(t)
\prod_{\substack{\delta_i = 0 \\ \delta_{i+1} = 1}}
(1-t^{m_i(\mu)})
\prod_{\substack{\epsilon_j = 0 \\ \epsilon_{j+1} = 1}}
(1-t^{m_j(\nu)})
t^{|\mu-\lambda|}
x_{N-1}^{|\mu-\lambda|+|\mu-\nu|}
P_{\nu}(\vec{x}_{N-2};t).
\end{align*}
We isolate the coefficient of $P_{\nu}(x_1,\dots,x_{N-2};t)$ in the previous expression, and denote it by $\mathcal{D}(\nu)$:
\begin{align*}
\mathcal{D}(\nu)
=
\sum_{\substack{\mu:\ell(\mu) = N-1 \\ \mu' = \nu' + \epsilon}}
\
\sum_{\substack{\lambda:\ell(\lambda) = N 
\\ \lambda' = \mu' + \delta \\ \lambda'\ \text{even}}}
b_{\lambda}^{\rm el}(t)
\prod_{\substack{\delta_i = 0 \\ \delta_{i+1} = 1}}
(1-t^{m_i(\mu)})
\prod_{\substack{\epsilon_j = 0 \\ \epsilon_{j+1} = 1}}
(1-t^{m_j(\nu)})
t^{-|\delta|}
x^{|\epsilon|-|\delta|}
,
\end{align*}
where we write $x_{N-1} \equiv x$, since the subscript is irrelevant in what follows. To prove the required recursion relation, we need to show that
\begin{align*}
\mathcal{D}(\nu)
=
\left\{
\begin{array}{ll}
-t^{N-2},
&
\nu'\ \text{even},
\\
0,
&
\text{otherwise}.
\end{array}
\right.
\end{align*}
We start by considering the case $\nu'\ \text{even}$. Since $\lambda$ has even length columns, it follows by parity that $\delta_i = \epsilon_i$ for all $i \geq 1$, which simplifies the summation as follows:
\begin{align*}
\mathcal{D}(\nu)
&=
\sum_{j=1}^{\infty}
\sum_{\delta_j \in \{0,1\}}
t^{-|\delta|}
b_{\lambda}^{\rm el}(t)
\prod_{\substack{\delta_k = 0 \\ \delta_{k+1} = 1}}
(1-t^{m_k(\mu)})
(1-t^{m_k(\nu)})
\\
&=
\sum_{j=1}^{\infty}
\sum_{\delta_j \in \{0,1\}}
t^{-|\delta|}
b_{\lambda}^{\rm el}(t)
\prod_{\substack{\delta_k = 0 \\ \delta_{k+1} = 1}}
(1-t^{m_k(\nu)-1})
(1-t^{m_k(\nu)}),
\end{align*}
where in the final line we have used the fact that if $\delta_k = 0, \delta_{k+1} = 1$, then $m_k(\mu) = m_k(\nu)-1$, and where $\lambda$ is the partition satisfying $\lambda' = \nu' + 2 \delta$. Proceeding in direct analogy with Section \ref{1-rec}, we define the sequence of partial coefficients
\begin{align}
\label{D-partial-coeff}
\mathcal{D}_{i,\delta_i}(\nu)
=
\sum_{j=1}^{i-1}
\sum_{\delta_j \in \{0,1\}}
t^{-\sum_{k=1}^{i} \delta_k}
\prod_{k=1}^{i-1}
\prod_{l\ {\rm even}}^{m_k(\lambda)}
(1-t^{l-1})
\prod_{\substack{1 \leq k \leq i-1 \\ \delta_k = 0 \\ \delta_{k+1} = 1}}
(1-t^{m_k(\nu)-1})
(1-t^{m_k(\nu)}),
\end{align}
where $\lambda$ is the partition formed by taking $\lambda'_j = \nu'_j + 2 \delta_j$ for all $1 \leq j \leq i$, $\lambda'_j = \nu'_j$ for all $j > i$. Given the similarity of these coefficients to those defined in equation \eref{part-coeff1}, we expect that they will satisfy recurrence relations of an analogous form to \eref{recur-1} and \eref{recur-2}. Indeed, by taking $\mathcal{D}_{i+1,\delta_{i+1}}(\nu)$ and performing its summation over $\delta_i$ explicitly, we find that
\begin{align}
\label{recur-3}
\mathcal{D}_{i+1,0}(\nu)
&=
\prod_{j\ {\rm even}}^{m_i(\nu)}
(1-t^{j-1})
\left[
\mathcal{D}_{i,0}(\nu)
+
(1-t^{m_i(\nu)+1})
\mathcal{D}_{i,1}(\nu)
\right],
\\
\label{recur-4}
t \mathcal{D}_{i+1,1}(\nu)
&=
\prod_{j\ {\rm even}}^{m_i(\nu)}
(1-t^{j-1})
\left[
(1-t^{m_i(\nu)})
\mathcal{D}_{i,0}(\nu)
+
\mathcal{D}_{i,1}(\nu)
\right],
\end{align}
valid for all $i \geq 1$, with initial values $\mathcal{D}_{1,0}(\nu)=0$ ($\delta_1= 0$ is not allowed, since this would mean that $\ell(\lambda) = N-2$) and 
$\mathcal{D}_{1,1}(\nu)=t^{-1}$. Since the recursion relations \eref{recur-3} and \eref{recur-4} are virtually identical to \eref{recur-1} and \eref{recur-2}, all of the reasoning presented in Section \ref{1-rec} also goes through in the present instance. In particular, we find that
\begin{align*}
\mathcal{D}(\nu)
=
\lim_{i \rightarrow \infty}
\mathcal{D}_{i,0}(\nu)
=
\mathcal{D}_{\nu_1+1,0}(\nu)
-
t
\mathcal{D}_{\nu_1+1,1}(\nu),
\end{align*}
where $\mathcal{D}_{i,0}(\nu) - t\mathcal{D}_{i,1}(\nu)$ obeys the recurrence
\begin{align}
\label{recur-final2}
\mathcal{D}_{i+1,0}(\nu)
-
t
\mathcal{D}_{i+1,1}(\nu)
=
t^{m_i(\nu)}
\prod_{j\ {\rm even}}^{m_i(\nu)}
(1-t^{j-1})
\Big[
\mathcal{D}_{i,0}(\nu)
-
t
\mathcal{D}_{i,1}(\nu)
\Big]
\end{align}
with initial condition $\mathcal{D}_{1,0}(\nu)-t\mathcal{D}_{1,1}(\nu) = -1$. Solving this recurrence, we obtain
\begin{align*}
\mathcal{D}_{\nu_1+1,0}(\nu)
-
t
\mathcal{D}_{\nu_1+1,1}(\nu)
=
-t^{\sum_{i=1}^{\infty} m_i(\nu)}
\prod_{i=1}^{\infty}
\prod_{j\ {\rm even}}^{m_i(\nu)}
(1-t^{j-1})
=
-t^{N-2}
b_{\nu}^{\rm el}(t),
\end{align*}
completing the proof in the case $\nu'\ \text{even}$.

Turning to the case where $\nu$ has at least one column of odd length, our task is to calculate
\begin{align*}
\mathcal{D}(\nu)
=
\sum_{j=1}^{\infty}
\sum_{\substack{\delta_j \in \{0,1\} \\ \epsilon_j \in \{0,1\}}}
t^{-|\delta|}
x^{|\epsilon|-|\delta|}
b_{\lambda}^{\rm el}(t)
\prod_{\substack{\delta_k = 0 \\ \delta_{k+1} = 1}}
(1-t^{m_k(\mu)})
\prod_{\substack{\epsilon_k = 0 \\ \epsilon_{k+1} = 1}}
(1-t^{m_k(\nu)}),
\end{align*}
where $\mu$ is the length $N-1$ partition given by $\mu' = \nu' + \epsilon$, and 
$\lambda$ the length $N$ partition given by $\lambda' = \mu' + \delta$. We define partial coefficients
\begin{multline*}
\mathcal{D}_{i,\delta_i,\epsilon_i}(\nu)
=
\sum_{j=1}^{i-1}
\sum_{\substack{\delta_j \in \{0,1\} \\ \epsilon_j \in \{0,1\}}}
\left[
t^{-\sum_{k=1}^{i} \delta_k}
\right]
\left[
x^{\sum_{k=1}^{i} (\epsilon_k - \delta_k)}
\right]
\\
\times
\prod_{k=1}^{i-1}
\prod_{l\ {\rm even}}^{m_k(\lambda)}
(1-t^{l-1})
\prod_{\substack{1 \leq k \leq i-1 \\ \delta_k = 0 \\ \delta_{k+1} = 1}}
(1-t^{m_k(\mu)})
\prod_{\substack{1 \leq k \leq i-1 \\ \epsilon_k = 0 \\ \epsilon_{k+1} = 1}}
(1-t^{m_k(\nu)})
\end{multline*}
and let $\ihat$ denote the largest $i$ such that $\nu'_i$ is odd (meaning that $\nu'_i$ is even for all $i > \ihat$). Since $\lambda'_{\ihat}$ is necessarily even, it follows that either $\delta_{\ihat} = 1, \epsilon_{\ihat} = 0$ or $\delta_{\ihat} = 0, \epsilon_{\ihat} = 1$. Summing over these possibilities, we obtain the recurrences
\begin{align}
\label{non-diag-rec3}
\mathcal{D}_{\ihat+1,0,0}(\nu)
&=
\prod_{j\ {\rm even}}^{m_{\ihat}(\nu)+1}
(1-t^{j-1})
\Big[
\mathcal{D}_{\ihat,1,0}(\nu)
+
\mathcal{D}_{\ihat,0,1}(\nu)
\Big],
\\
\label{non-diag-rec4}
t
\mathcal{D}_{\ihat+1,1,1}(\nu)
&=
\prod_{j\ {\rm even}}^{m_{\ihat}(\nu)+1}
(1-t^{j-1})
\Big[
\mathcal{D}_{\ihat,1,0}(\nu)
+
\mathcal{D}_{\ihat,0,1}(\nu)
\Big],
\end{align}
where we are only obliged to consider $\delta_{\ihat+1} = \epsilon_{\ihat+1}$, since by the definition of $\ihat$ and using parity, $\delta_i = \epsilon_i$ for all $i > \ihat$. This ensures that for $i > \ihat$,
\begin{align*}
\mathcal{D}_{i+1,0,0}(\nu)
&=
\prod_{j\ {\rm even}}^{m_i(\nu)}
(1-t^{j-1})
\left[
\mathcal{D}_{i,0,0}(\nu)
+
(1-t^{m_i(\nu)+1})
\mathcal{D}_{i,1,1}(\nu)
\right],
\\
t \mathcal{D}_{i+1,1,1}(\nu)
&=
\prod_{j\ {\rm even}}^{m_i(\nu)}
(1-t^{j-1})
\left[
(1-t^{m_i(\nu)})
\mathcal{D}_{i,0,0}(\nu)
+
\mathcal{D}_{i,1,1}(\nu)
\right],
\end{align*}
which are the same recursion relations as \eref{recur-3} and \eref{recur-4}. We are thus in the same situation as in the non-diagonal part ($\mu \not= \nu$) of Section \ref{1-rec}. The quantity that we wish to calculate is
\begin{align*}
\mathcal{D}(\nu)
=
\lim_{i \rightarrow \infty}
\mathcal{D}_{i,0,0}(\nu)
=
\mathcal{D}_{\nu_1+1,0,0}(\nu)
-
t
\mathcal{D}_{\nu_1+1,1,1}(\nu),
\end{align*}
where $\mathcal{D}_{i,0,0}(\nu) - t \mathcal{D}_{i,1,1}(\nu)$ is given by a recurrence of the form \eref{recur-final2}, but with the trivial initial condition 
$\mathcal{D}_{\ihat+1,0,0}(\nu) - t\mathcal{D}_{\ihat+1,1,1}(\nu) = 0$ (by subtracting equation \eref{non-diag-rec4} from \eref{non-diag-rec3}). Since its initial value is zero, this recurrence has the solution $\mathcal{D}_{i,0,0}(\nu) - t \mathcal{D}_{i,1,1}(\nu) = 0$ for all $i > \ihat$, as we were required to show.

\subsubsection{Initial condition}
\label{2-ic}

The case $N=2$ can be computed explicitly without difficulty. Indeed, we find that
\begin{align*}
\mathcal{G}_2 (x_1,x_2)
=
(1-t)
\sum_{k=0}^{\infty}
P_{(k,k)}(x_1,x_2;t)
=
(1-t)
\sum_{k=0}^{\infty}
x_1^k x_2^k
=
\frac{1-t}{1-x_1 x_2}.
\end{align*}

\section{Identities at Macdonald level}
\label{sec:mac}

The identities \eref{HL-ASM}--\eref{HL-UASM} listed at the start of this paper apply at the level of Hall--Littlewood polynomials. Since Hall--Littlewood polynomials are the $q=0$ specialization of Macdonald polynomials, it is natural to suggest that these equations are special cases of yet more general identities involving extra parameters. 

In this section we show that this is indeed the case, by presenting Macdonald analogues of both equations \eref{HL-ASM} and \eref{HL-OSASM}. It turns out that these equations can be generalized by the introduction of two additional parameters, one being the $q$ from Macdonald theory. The Macdonald generalization of \eref{HL-ASM} has been known since the work of Warnaar in \cite{war}, and can be proved using Macdonald difference operators. Although we obtain a completely analogous generalization of \eref{HL-OSASM} to Macdonald level, it remains conjectural, since we lack an appropriate family of difference operators to expedite its proof.

As an aside, we remark that we do not know of an appropriate generalization of \eref{HL-UASM} to Macdonald level, even conjecturally. We are however able to deform it by the introduction of certain additional parameters, but we defer this result to Section \ref{sec:UUASM} since it does not pertain directly to symmetric polynomials at Macdonald level.

\subsection{$u$-deformed Macdonald Cauchy identity}

The following theorem can be deduced by acting on the Macdonald Cauchy identity
\begin{align}
\label{cauchy-mac}
\sum_{\lambda}
b_{\lambda}(q,t)
P_{\lambda}(x_1,\dots,x_n;q,t)
P_{\lambda}(y_1,\dots,y_n;q,t)
=
\prod_{i,j=1}^{n}
\frac{(t x_i y_j; q)}{(x_i y_j ; q)}
\end{align}
with the generating series
\begin{align}
\label{dops-mac}
D_n(u)
=
\sum_{r=0}^{n}
(-u)^r
\sum_{
\substack{
S \subseteq [n] \\ 
|S| = r
}
}
t^{r(r-1)/2}
\prod_{\substack{
i \in S \\ j \not\in S
}}
\frac{tx_i - x_j}{x_i - x_j}
\prod_{i \in S}
T_{q,x_i} 
\end{align}
of Macdonald difference operators \cite{mac}. It was partially discovered in \cite{kn} and discussed again in \cite{war} (see Equations (3.2) and (3.3) therein). The fact that the right hand side is a determinant for all values of the parameter $u$ was not made explicit in \cite{war}, but the procedure presented therein (for the case $u=t$) can be applied {\it mutatis mutandis} when $u$ is generic. For that reason, we attribute this theorem to Warnaar.

\begin{thm}[Warnaar]
\label{mac-cauch-thm}
\begin{multline}
\label{mac-asm}
\sum_{\lambda}
\prod_{i=1}^{n}
(1- u q^{\lambda_i} t^{n-i})
b_{\lambda}(q,t)
P_{\lambda}(x_1,\dots,x_n;q,t)
P_{\lambda}(y_1,\dots,y_n;q,t)
=
\\
\prod_{i,j=1}^{n}
\frac{(t x_i y_j; q)}{(x_i y_j ; q)}
\frac{ \prod_{i,j=1}^{n} (1 - x_i y_j) }
{\prod_{1 \leq i<j \leq n} (x_i - x_j) (y_i - y_j)}
\det_{1\leq i,j \leq n}
\left[
\frac{1-u + (u-t) x_i y_j}{(1-x_i y_j) (1-t x_i y_j)}
\right].
\end{multline}
\end{thm}

\begin{proof}
Acting on both sides of the Cauchy identity \eref{cauchy-mac} with the operator \eref{dops-mac}, one obtains
\begin{multline}
\label{det-id}
\sum_{\lambda}
\prod_{i=1}^{n}
(1- u q^{\lambda_i} t^{n-i})
b_{\lambda}(q,t)
P_{\lambda}(x_1,\dots,x_n;q,t)
P_{\lambda}(y_1,\dots,y_n;q,t)
=
\\
\prod_{i,j=1}^{n}
\frac{(t x_i y_j; q)}{(x_i y_j ; q)}
\sum_{r=0}^{n}
\sum_{\substack{S \subseteq [n] \\  |S| = r}}
(-u)^r t^{r(r-1)/2}
\prod_{\substack{ i \in S \\ j \not\in S }}
\frac{t x_i - x_j}{x_i - x_j}
\prod_{i \in S}
\prod_{j=1}^{n}
\frac{1 - x_i y_j} {1- t x_i y_j}. 
\end{multline}
It thus suffices to show that
\begin{multline}
\label{thm1-proof}
\sum_{r=0}^{n}
\sum_{\substack{S \subseteq [n] \\  |S| = r}}
(-u)^r t^{r(r-1)/2}
\prod_{\substack{ i \in S \\ j \not\in S }}
\frac{t x_i - x_j}{x_i - x_j}
\prod_{i \in S}
\prod_{j=1}^{n}
\frac{1 - x_i y_j} {1- t x_i y_j} 
=
\\
\frac{ \prod_{i,j=1}^{n} (1 - x_i y_j) }
{\prod_{1 \leq i<j \leq n} (x_i - x_j) (y_i - y_j)}
\det_{1\leq i,j \leq n}
\left[
\frac{1-u + (u-t) x_i y_j}{(1-x_i y_j) (1-t x_i y_j)}
\right].
\end{multline}
This can be done using Lagrange interpolation. We let $\mathcal{L}_n$ and $\mathcal{R}_n$ denote the left and right hand sides of \eref{thm1-proof}, having first multiplied this equation by 
$\prod_{i,j=1}^{n} (1-tx_i y_j)$. Both $\mathcal{L}_n$ and $\mathcal{R}_n$ are polynomials in $x_n$ of degree $n$, and manifestly symmetric in the variables $\{y_1,\dots,y_n\}$. We find that $\mathcal{L}_n$ satisfies two simple recursion relations:
\begin{align*}
\mathcal{L}_n
\Big|_{x_n = 1/y_n}
&=
(1-t) 
\prod_{i=1}^{n-1}
(1-t x_i y_n)
\prod_{j=1}^{n-1}
(1-t y_j/y_n)
\times
\prod_{i,j=1}^{n-1}
(1-t x_i y_j)
\times
\\
\sum_{r=0}^{n-1}
\sum_{\substack{S \subseteq [n-1] \\ |S| = r}}
&
(-u)^r t^{r(r-1)/2}
\prod_{\substack{ i \in S \\ j \not\in S }}
\left(
\frac{t x_i - x_j}{x_i - x_j}
\right)
\prod_{i \in S}
\left(
\frac{t x_i - 1/y_n}{x_i - 1/y_n}
\right)
\prod_{i \in S}
\prod_{j=1}^{n-1}
\left(
\frac{1 - x_i y_j} {1- t x_i y_j}
\right)
\prod_{i \in S}
\left(
\frac{1 - x_i y_n} {1- t x_i y_n}
\right)
\\
&=
(1-t) 
\prod_{i=1}^{n-1}
(1-t x_i y_n)
\prod_{j=1}^{n-1}
(1-t y_j/y_n)
\mathcal{L}_{n-1},
\end{align*}
\begin{align*}
\mathcal{L}_n
\Big|_{x_n = 1/(t y_n)}
&=
(1-1/t)
\prod_{i=1}^{n-1}
(1-t x_i y_n)
\prod_{j=1}^{n-1}
(1-y_j/y_n)
\times
\prod_{i,j=1}^{n-1}
(1-t x_i y_j)
\times
\\
\sum_{r=1}^{n}
\sum_{\substack{S \subseteq [n-1] \\ |S| = r-1}}
&
(-u)^{r} t^{r(r-1)/2}
\prod_{\substack{ i \in S \\ j \not\in S }}
\left(
\frac{t x_i - x_j}{x_i - x_j}
\right)
\prod_{j \not\in S}
\left(
\frac{1/y_n - x_j}{1/(ty_n) - x_j}
\right)
\prod_{i \in S}
\prod_{j=1}^{n}
\left(
\frac{1 - x_i y_j} {1- t x_i y_j}
\right)
\prod_{j=1}^{n-1}
\left(
\frac{1 - y_j/(t y_n)} {1- y_j/y_n}
\right)
\\
&=
u
t^{n-1}
(1/t-1)
\prod_{i=1}^{n-1}
(1- x_i y_n)
\prod_{j=1}^{n-1}
(1-y_j/(t y_n))
\mathcal{L}_{n-1}.
\end{align*}
Identical recursion relations are satisfied by $\mathcal{R}_n$. Thanks to the symmetry in $\{y_1,\dots,y_n\}$, these recursion relations prove that $\mathcal{L}_n = \mathcal{R}_n$ at $2n$ values of $x_n$ (more than sufficient for Lagrange interpolation), provided that they agree for $n=1$. It is immediate from their definitions that $\mathcal{L}_1 = 1-u+(u-t) x_1 y_1 = \mathcal{R}_1$.

\end{proof}

\subsection{$u$-deformed Macdonald Littlewood identity}

Throughout this subsection, we let $N=2n$. The following conjecture\footnote{We are grateful to E.~Rains for a comprehensive numerical test of the factorized dependence on $q$ and $u$ in \eref{mac-osasm}, and to O.~Warnaar for independently suggesting this conjecture to us while the manuscript was in preparation.} is motivated by the even columns Littlewood identity
\begin{align}
\sum_{\lambda'\ {\rm even}}
b^{\rm el}_{\lambda}(q,t)
P_{\lambda}(x_1,\dots,x_N;q,t)
=
\prod_{1 \leq i<j \leq N}
\frac{(t x_i x_j;q)}{(x_i x_j;q)} 
\label{little-mac}
\end{align}
for Macdonald polynomials (see Example 4, Section 7, Chapter VI of \cite{mac}). Although we do not have a proof of this conjecture, it is tempting to suggest that it can be deduced by acting on the Littlewood identity 
\eref{little-mac} with an appropriate family of difference operators, in much the same way that Theorem \ref{mac-cauch-thm} follows from the Cauchy identity \eref{cauchy-mac}. A preliminary step in this direction is given in the remark following the conjecture. 

\begin{conj}
\label{mac-conj}
\begin{multline}
\label{mac-osasm}
\sum_{\lambda'\ {\rm even}}
\ 
\prod_{i\ {\rm even}}^{N}
(1-u q^{\lambda_i} t^{N-i} )
b^{\rm el}_{\lambda}(q,t)
P_{\lambda}(x_1,\dots,x_{N};q,t)
=
\\
\prod_{1 \leq i<j \leq N}
\frac{(t x_i x_j;q)}{(x_i x_j;q)}
\prod_{1 \leq i<j \leq N}
\frac{(1-x_i x_j)}{(x_i - x_j)}
\pf_{1\leq i < j \leq N}
\left[
\frac{(x_i - x_j) (1 - u + (u-t) x_i x_j)}
{(1-x_i x_j) (1-t x_i x_j)}
\right].
\end{multline}
\end{conj}
\begin{rmk} 
One can express the Pfaffian on the right hand side of \eref{mac-osasm} as a sum over subsets of 
$\{1,\dots,N\}$, as in the following lemma.
\end{rmk}  
\begin{lem}
\begin{multline}
\label{pfaff-id}
\sum_{r=0}^{n}
(-u)^r
\sum_{
\substack{
S \subseteq [N] \\ 
|S| = 2r
}
}
t^{r(r-1)}
\prod_{\substack{
i \in S \\ j \not\in S
}}
\frac{1 - x_i x_j}{x_i - x_j}
\prod_{\substack{i<j \\ i,j \in S}}
\frac{1-x_i x_j}{1-t x_i x_j}
=
\\
\prod_{1 \leq i<j \leq N}
\frac{(1-x_i x_j)}{(x_i - x_j)}
\pf_{1\leq i < j \leq N}
\left[
\frac{(x_i - x_j) (1 - u + (u-t) x_i x_j)}
{(1-x_i x_j) (1-t x_i x_j)}
\right].
\end{multline}
\end{lem}

\begin{proof}
The proof proceeds along analogous lines to the proof of Theorem \ref{mac-cauch-thm}. We let $\mathcal{L}_N$ and $\mathcal{R}_N$ denote the left and right hand sides of \eref{pfaff-id}, after it is multiplied by 
$\prod_{1\leq i<j \leq N}(1-t x_i x_j)$. Both $\mathcal{L}_N$ and $\mathcal{R}_N$ are polynomials in $x_{N}$ of degree $N-1$, and symmetric in the set of variables $\{x_1,\dots,x_{N}\}$. $\mathcal{L}_{N}$ satisfies the following two recursion relations:
\begin{align*}
&
\mathcal{L}_{N}
\Big|_{x_{N} = \b{x}_{N-1}}
\\
&=
(1-t)
\prod_{i=1}^{N-2}
(1-tx_i x_{N-1})
(1-tx_i \b{x}_{N-1})
\sum_{r=0}^{n-1}
\sum_{\substack{S \subseteq [N-2] \\ |S| = 2r}}
(-u)^r
t^{r(r-1)}
\prod_{\substack{
i \in S \\ j \not\in S
}}
\left(
\frac{1 - x_i x_j}{x_i - x_j}
\right)
\prod_{\substack{i<j \\ i,j \in S}}
\left(
\frac{1-x_i x_j}{1-t x_i x_j}
\right)
\\
&=
(1-t)
\prod_{i=1}^{N-2}
(1-tx_i x_{N-1})
(1-tx_i \b{x}_{N-1})
\mathcal{L}_{N-2}, 
\end{align*}

\begin{align*}
\mathcal{L}_{N}
\Big|_{x_{N} = \b{x}_{N-1}/t}
=&
(1-1/t)
\prod_{i=1}^{N-2}
(1-tx_i x_{N-1})
(1-x_i \b{x}_{N-1})
\sum_{r=1}^{n}
\sum_{\substack{S \subseteq [N-2] \\ |S| = 2r-2}}
(-u)^r t^{r(r-1)}
\prod_{\substack{i \in S \\ j \not\in S}}
\left(
\frac{1-x_i x_j}{x_i-x_j}
\right)
\times
\\
\prod_{j\not\in S}
\left(
\frac{1-x_{N-1} x_j}{x_{N-1}-x_j}
\right)
&
\left(
\frac{1-\b{x}_{N-1} x_j/t}{\b{x}_{N-1}/t-x_j}
\right)
\prod_{1\leq i<j \leq N-2}
\left(
\frac{1-x_i x_j}{1-t x_i x_j}
\right)
\prod_{i \in S}
\left(
\frac{1-x_i x_{N-1}}{1-t x_i x_{N-1}}
\right)
\left(
\frac{1-x_i \b{x}_{N-1}/t}{1-x_i \b{x}_{N-1}}
\right)
\\
=&
-u
t^{N-2}
(1-1/t)
\prod_{i=1}^{N-2}
(1- x_i x_{N-1})
(1- x_i \b{x}_{N-1}/t)
\mathcal{L}_{N-2}.
\end{align*}
It is straightforward to show that both of these recursion relations are satisfied by 
$\mathcal{R}_{N}$. Due to the symmetry in $\{x_1,\dots,x_{N}\}$, these recursion relations prove that 
$\mathcal{L}_{N} = \mathcal{R}_{N}$ at $2N-2$ points, as long as they agree for $N=2$. It clear from their definitions that $\mathcal{L}_2 = (1-t x_1 x_2) - u (1-x_1 x_2) = \mathcal{R}_2$.

\end{proof}

The expression of the Pfaffian appearing on the right hand side of \eref{mac-osasm} as a sum over subsets of $\{1,\dots,N\}$, as achieved by equation \eref{pfaff-id}, would seem to be an important step towards the proof of \eref{mac-osasm}. Indeed, the analogous result \eref{det-id} was crucial in the proof of \eref{mac-asm}, since the type of sum arising in that case was manifestly related to the generating series 
\eref{dops-mac} of difference operators. 

Nevertheless, we do not yet know of a family of operators whose action on the right hand side of the Littlewood identity \eref{little-mac} produces the sum in \eref{pfaff-id}. The discovery of such operators would not only lead to the completed proof of \eref{mac-osasm}, but would constitute an important development in the theory of Macdonald polynomials in its own right.

\section{$u$-deformed Cauchy identity and half-turn symmetric alternating sign matrices}
\label{sec:HTASM}

The aim of this section is to study the $q=0$ specialization of equation \eref{mac-asm}, and its relation with the six-vertex model. In particular we will study the six-vertex model on a lattice with {\it half-turn symmetry}, and calculate its partition function as a product of two determinants following Kuperberg \cite{kup2}. One of the determinants is precisely the domain wall partition function (the right hand side of \eref{HL-ASM}), while the remaining determinant is equal to the right hand side of \eref{mac-asm} with $q=0$ and 
$u=-\sqrt{t}$.

\subsection{Six-vertex model in the bulk}

We begin with some preliminary material on the six-vertex model. We consider lattices formed by the intersection of horizontal and vertical lines. The points of intersection form vertices, and each of the four edges surrounding a vertex is assigned an arrow configuration, such that exactly two arrows point towards/away from the point of intersection. This gives rise to six legal vertices, which are illustrated in Figure 
\ref{fig:6vertices}.


\begin{figure}[H]
\begin{tabular}{ccc}
\begin{tikzpicture}[scale=0.6]
\draw[thick, smooth] (-1,0) -- (1,0);
\node at (-0.5,0) {$\r$}; \node at (0.5,0) {$\r$};
\node[label={left: \fs ${\color{red} \shortrightarrow} \ x$}] at (-1,0) {};
\draw[thick, smooth] (0,-1) -- (0,1);
\node at (0,-0.5) {$\u$}; \node at (0,0.5) {$\u$};
\node[label={below: \fs ${\color{red} \begin{array}{c} {\color{black} y} \\ \shortuparrow \end{array} }$}] at (0,-1) {};
\end{tikzpicture}
\quad\quad\quad&
\begin{tikzpicture}[scale=0.6]
\draw[thick, smooth] (-1,0) -- (1,0);
\node at (-0.5,0) {$\r$}; \node at (0.5,0) {$\r$};
\node[label={left: \fs ${\color{red} \shortrightarrow} \ x$}] at (-1,0) {};
\draw[thick, smooth] (0,-1) -- (0,1);
\node at (0,-0.5) {$\d$}; \node at (0,0.5) {$\d$};
\node[label={below: \fs ${\color{red} \begin{array}{c} {\color{black} y} \\ \shortuparrow \end{array} }$}] at (0,-1) {};
\end{tikzpicture}
\quad\quad\quad&
\begin{tikzpicture}[scale=0.6]
\draw[thick, smooth] (-1,0) -- (1,0);
\node at (-0.5,0) {$\r$}; \node at (0.5,0) {$\l$};
\node[label={left: \fs ${\color{red} \shortrightarrow} \ x$}] at (-1,0) {};
\draw[thick, smooth] (0,-1) -- (0,1);
\node at (0,-0.5) {$\d$}; \node at (0,0.5) {$\u$};
\node[label={below: \fs ${\color{red} \begin{array}{c} {\color{black} y} \\ \shortuparrow \end{array} }$}] at (0,-1) {};
\end{tikzpicture}
\\
\quad
$a_{+}(x,y)$
\quad\quad\quad&
\quad
$b_{+}(x,y)$
\quad\quad\quad&
\quad
$c_{+}(x,y)$
\\
\\
\begin{tikzpicture}[scale=0.6]
\draw[thick, smooth] (-1,0) -- (1,0);
\node at (-0.5,0) {$\l$}; \node at (0.5,0) {$\l$};
\node[label={left: \fs ${\color{red} \shortrightarrow} \ x$}] at (-1,0) {};
\draw[thick, smooth] (0,-1) -- (0,1);
\node at (0,-0.5) {$\d$}; \node at (0,0.5) {$\d$};
\node[label={below: \fs ${\color{red} \begin{array}{c} {\color{black} y} \\ \shortuparrow \end{array} }$}] at (0,-1) {};
\end{tikzpicture}
\quad\quad\quad&
\begin{tikzpicture}[scale=0.6]
\draw[thick, smooth] (-1,0) -- (1,0);
\node at (-0.5,0) {$\l$}; \node at (0.5,0) {$\l$};
\node[label={left: \fs ${\color{red} \shortrightarrow} \ x$}] at (-1,0) {};
\draw[thick, smooth] (0,-1) -- (0,1);
\node at (0,-0.5) {$\u$}; \node at (0,0.5) {$\u$};
\node[label={below: \fs ${\color{red} \begin{array}{c} {\color{black} y} \\ \shortuparrow \end{array} }$}] at (0,-1) {};
\end{tikzpicture}
\quad\quad\quad&
\begin{tikzpicture}[scale=0.6]
\draw[thick, smooth] (-1,0) -- (1,0);
\node at (-0.5,0) {$\l$}; \node at (0.5,0) {$\r$};
\node[label={left: \fs ${\color{red} \shortrightarrow} \ x$}] at (-1,0) {};
\draw[thick, smooth] (0,-1) -- (0,1);
\node at (0,-0.5) {$\u$}; \node at (0,0.5) {$\d$};
\node[label={below: \fs ${\color{red} \begin{array}{c} {\color{black} y} \\ \shortuparrow \end{array} }$}] at (0,-1) {};
\end{tikzpicture}
\\
\quad
$a_{-}(x,y)$
\quad\quad\quad &
\quad
$b_{-}(x,y)$
\quad\quad\quad &
\quad
$c_{-}(x,y)$
\end{tabular}
\caption{The vertices of the six-vertex model, with Boltzmann weights indicated beneath. The small red arrows indicate the orientation of the lines. In order to distinguish between $a$ and $b$ type vertices, and also between $c_{+}$ and $c_{-}$ vertices, the correct convention is to view every vertex such that its lines are oriented from south-west to north-east.}
\label{fig:6vertices}
\end{figure}
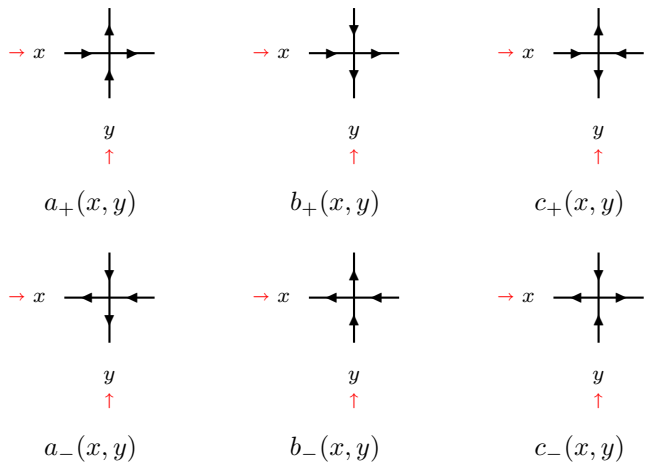
To each horizontal (respectively, vertical) line of the lattice one associates an orientation and a variable $x_i$ (respectively, $y_j$), called a {\it rapidity}. The six types of vertex are assigned Boltzmann weights, which are rational functions depending on the ratio 
$x/y$ of the rapidities incident on the vertex:
\begin{align}
\label{boltz}
a_{\pm}(x,y)
=
\frac{1-t x/y}{1-x/y},
\quad\quad
b_{\pm}(x,y)
=
\sqrt{t},
\quad\quad
c_{+}(x,y)
=
\frac{(1-t)}{1-x/y},
\quad\quad
c_{-}(x,y)
=
\frac{(1-t) x/y}{1-x/y}.
\end{align}
Note that in order to correctly determine the Boltzmann weight of a vertex it is necessary to place the vertex in some canonical orientation, which means rotating the vertex such that the orientation of its lines are from left to right and bottom to top. The crucial feature of the Boltzmann weights thus chosen is that they satisfy the Yang--Baxter equation, shown in Figure \ref{fig:YB}. The Yang--Baxter equation plays an essential role in the partition functions to be studied, since it ensures that these functions are symmetric in their rapidity variables. Without this fact, it would clearly not be possible to expand these objects with respect to Hall--Littlewood polynomials. 
\begin{figure}[H]
\begin{tikzpicture}[scale=0.8]
\draw[thick,smooth] (0,1.5) arc (-155:-90:2);
\draw[thick,smooth] (0,0) arc (155:90:2);
\draw[thick, smooth] 
({2*cos(25)},{1.5-(2-2*sin(25))})--({2*cos(25)+0.5},{1.5-(2-2*sin(25))});
\draw[thick, smooth] ({2*cos(25)},{(2-2*sin(25))})--({2*cos(25)+0.5},{(2-2*sin(25))});
\draw[thick,smooth] (1.3,-0.25)--(1.3,1.75);
\node at (5,0.75) {$=$};
\draw[thick,smooth] (10,1.5) arc (-25:-90:2);
\draw[thick,smooth] (10,0) arc (25:90:2);
\draw[thick, smooth] 
({10-2*cos(25)},{1.5-(2-2*sin(25))})--({10-2*cos(25)-0.5},{1.5-(2-2*sin(25))});
\draw[thick, smooth] ({10-2*cos(25)},{(2-2*sin(25))})--({10-2*cos(25)-0.5},{(2-2*sin(25))});
\draw[thick,smooth] (8.7,-0.25)--(8.7,1.75);

\node[left,rotate=295] at (0,1.5) {\fs ${\color{red} \shortrightarrow} \ x$};
\node[left,rotate=-295] at (0,0) {\fs ${\color{red} \shortrightarrow} \ y$};

\node[label={below: \fs ${\color{red} \begin{array}{c}  {\color{black} z} \\ \shortuparrow \end{array} }$}] at (1.3,-0.25) {};

\node[label={left: \fs ${\color{red} \shortrightarrow} \ y$}] at 
({10-2*cos(25)-0.5},{1.5-(2-2*sin(25))}) {};
\node[label={left: \fs ${\color{red} \shortrightarrow} \ x$}] at 
({10-2*cos(25)-0.5},{(2-2*sin(25))}) {};
\node[label={below: \fs ${\color{red} \begin{array}{c} {\color{black} z} \\ \shortuparrow \end{array} }$}] at (8.7,-0.25) {};
\end{tikzpicture}
\caption{The Yang--Baxter equation. One makes a definite choice for the arrows on the six external edges (which is consistent and fixed on both sides of the equation) and sums over the possible arrow configurations of the three internal edges. In this way, the figure actually implies $2^6$ equations involving the Boltzmann weights \eref{boltz}.}
\label{fig:YB} 
\end{figure}
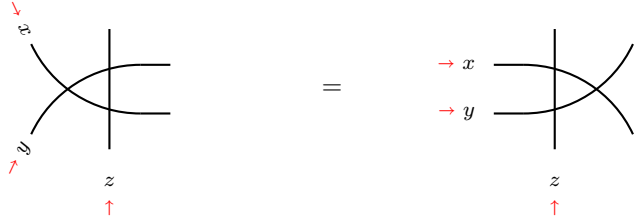

\subsection{Partition function on half-turn symmetric lattice}

We turn our attention to the six-vertex model under domain wall boundary conditions, with half-turn (or 180\degree\ rotational) symmetry imposed -- see Figure \ref{fig:HT}. Configurations on this lattice are in one-to-one correspondence with half-turn symmetric alternating sign matrices \cite{kup2}.


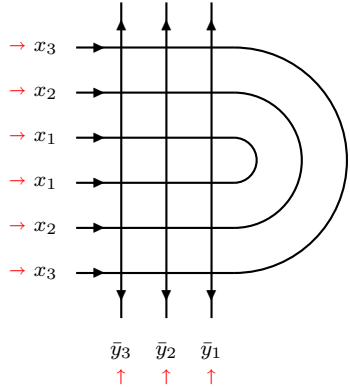
\begin{figure}[H]
\begin{tikzpicture}[scale=0.6]

\foreach\x in {1,...,6}
\draw[thick]
(0,\x) -- (3.5,\x);

\foreach\x in {1,...,3}
{\node[label={left: \fs ${\color{red} \shortrightarrow} \ x_{\x}$}] at (0,4-\x) {};
\node[label={left: \fs ${\color{red} \shortrightarrow} \ x_{\x}$}] at (0,3+\x) {};}

\draw[thick, smooth] (3.5,3) arc (-90:90:0.5);
\draw[thick, smooth] (3.5,2) arc (-90:90:1.5);
\draw[thick, smooth] (3.5,1) arc (-90:90:2.5);


\foreach\y in {1,...,6}
\node at (0.5,\y) {\r};

\foreach\x in {1,...,3}
\draw[thick]
(\x,0) -- (\x,7);

\foreach\x in {1,...,3}
\node[label={below: \fs ${\color{red} \begin{array}{c} {\color{black} \b{y}_{\x}} \\ \shortuparrow \end{array} }$}] at (4-\x,0) {};


\foreach\x in {1,...,3}
\node at (\x,6.5) {\u};

\foreach\x in {1,...,3}
\node at (\x,0.5) {\d};

\end{tikzpicture}
\caption{The partition function $Z_{\rm HT}$ of the six-vertex model with half-turn symmetric boundary conditions, in the case $n=3$. The semi-circular lines on the right side of the lattice indicate that the arrow on the final edge of the $i$-th horizontal line is paired (in a continuous fashion) with that of the $(2n-i+1)$-th horizontal line.}
\label{fig:HT} 
\end{figure}

\begin{lem}
\label{lem:HT-prop}
The partition function 
$Z^{(n)}_{\rm HT} = Z_{\rm HT}(x_1,\dots,x_n;y_1,\dots,y_n;t)$ 
as defined in Figure \ref{fig:HT} satisfies four properties:

\begin{enumerate}[label=\bf\arabic*.]

\item Multiplying by $\prod_{i,j=1}^{n} (1-x_i y_j)^2$, it is a polynomial in $x_n$ of degree $2n-1$.

\item It is symmetric in $\{y_1,\dots,y_n\}$.

\item It obeys the recursion relations
\begin{align}
\label{HT-rec1}
Z^{(n)}_{\rm HT}
\Big|_{x_n = \b{y}_n/t}
&=
-t^{2n-1/2}
Z^{(n-1)}_{\rm HT},
\\
\label{HT-rec2}
\lim_{x_n \rightarrow \b{y}_n}
\Big(
(1-x_n y_n)^2
Z^{(n)}_{\rm HT}
\Big)
&= 
(1-t)^2
\prod_{i=1}^{n-1}
\frac{(1-ty_i \b{y}_n)^2}{(1-y_i \b{y}_n)^2}
\frac{(1-tx_i y_n)^2}{(1-x_i y_n)^2}
Z^{(n-1)}_{\rm HT}.
\end{align}

\item When $n=1$, it is given explicitly by
\begin{align*}
Z^{(1)}_{\rm HT}
=
\frac{(1-t)(1+\sqrt{t})(1-\sqrt{t} x_1 y_1)}
{(1-x_1 y_1)^2}.
\end{align*}

\end{enumerate}
\end{lem}

\begin{proof}
Working directly from the lattice definition in Figure \ref{fig:HT}, we demonstrate these properties one by one.
\begin{enumerate}[label=\bf\arabic*.]

\item Multiplying the partition function by $\prod_{i,j=1}^{n} (1-x_i y_j)^2$ is equivalent to multiplying each individual Boltzmann weight by $(1-x_i y_j)$. After this renormalization, it is clear that every Boltzmann weight is a degree-1 polynomial in $x_i$, with the sole exception of the $c_{+}$ vertex (which is a constant).

Focusing attention on the top and bottom rows of the lattice in Figure \ref{fig:HT}, which are the only places which have dependence on $x_n$, one can easily deduce that exactly one $c_{+}$ vertex occurs in these two rows. It follows that the renormalized partition function is a polynomial in $x_n$ of degree $2n-1$. 

\item Symmetry in the $y$ variables is deduced using a standard argument involving the Yang--Baxter equation (see, for example, \cite{kzj}). Indeed, any two adjacent vertical lines can be exchanged using this procedure.

\item Setting $x_n = \b{y}_n/t$ eliminates the possibility that the top-left vertex of the lattice is an $a_{+}$ vertex. It follows that it must be a $c_{+}$ vertex. This forces a subset of the vertices into a {\it frozen configuration}, shown on the left of Figure \ref{fig:HT-freeze}. Studying the Boltzmann weights of the frozen region, we find that they contribute the total factor $-t(\sqrt{t})^{4n-3} = -t^{2n-1/2}$. The remaining (unfrozen) region is just $Z^{(n-1)}_{\rm HT}$. Hence we recover the first recursion relation \eref{HT-rec1}.


\begin{figure}[H]
\begin{tabular}{cc|cc}
\begin{tikzpicture}[scale=0.6]

\foreach\x in {1,...,6}
\draw[thick]
(0,\x) -- (3.5,\x);

\foreach\x in {1,...,2}
{\node[label={left: \fs ${\color{red} \shortrightarrow} \ x_{\x}$}] at (0,4-\x) {};
\node[label={left: \fs ${\color{red} \shortrightarrow} \ x_{\x}$}] at (0,3+\x) {};}

\node[label={left: \fs ${\color{red} \shortrightarrow} \ \b{y}_3/t$}] at (0,1) {};
\node[label={left: \fs ${\color{red} \shortrightarrow} \ \b{y}_3/t$}] at (0,6) {};

\draw[thick, smooth] (3.5,3) arc (-90:90:0.5);
\draw[thick, smooth] (3.5,2) arc (-90:90:1.5);
\draw[thick, smooth] (3.5,1) arc (-90:90:2.5);


\foreach\y in {1,...,6}
\node at (0.5,\y) {\r};

\node at (1.5,6) {\l};
\foreach\y in {1,...,5}
\node at (1.5,\y) {\r};

\foreach\x in {1,2}
{\node at (4.5-\x,6) {\l};
\node at (4.5-\x,1) {\r};}


\foreach\x in {1,...,3}
\draw[thick]
(\x,0) -- (\x,7);


\foreach\x in {1,...,3}
\node at (\x,6.5) {\u};

\node at (1,5.5) {\d};
\foreach\x in {2,3}
\node at (\x,5.5) {\u};

\foreach\y in {1,...,3}
\node at (1,5.5-\y) {\d};

\foreach\x in {1,...,3}
\node at (\x,1.5) {\d};

\foreach\x in {1,...,3}
\node at (\x,0.5) {\d};

\foreach\x in {1,...,3}
\node[label={below: \fs ${\color{red} \begin{array}{c} {\color{black} \b{y}_{\x}} \\ \shortuparrow \end{array} }$}] at (4-\x,0) {};
\end{tikzpicture}
&
&
&
\begin{tikzpicture}[scale=0.6]

\foreach\x in {1,...,6}
\draw[thick]
(0,\x) -- (3.5,\x);

\foreach\x in {1,...,2}
{\node[label={left: \fs ${\color{red} \shortrightarrow} \ x_{\x}$}] at (0,4-\x) {};
\node[label={left: \fs ${\color{red} \shortrightarrow} \ x_{\x}$}] at (0,3+\x) {};}

\node[label={left: \fs ${\color{red} \shortrightarrow} \ \b{y}_3$}] at (0,1) {};
\node[label={left: \fs ${\color{red} \shortrightarrow} \ \b{y}_3$}] at (0,6) {};

\draw[thick, smooth] (3.5,3) arc (-90:90:0.5);
\draw[thick, smooth] (3.5,2) arc (-90:90:1.5);
\draw[thick, smooth] (3.5,1) arc (-90:90:2.5);


\foreach\y in {1,...,6}
\node at (0.5,\y) {\r};

\foreach\y in {2,...,6}
\node at (1.5,\y) {\r};
\node at (1.5,1) {\l};

\foreach\x in {1,2}
{\node at (4.5-\x,6) {\r};
\node at (4.5-\x,1) {\l};}


\foreach\x in {1,...,3}
\draw[thick]
(\x,0) -- (\x,7);


\foreach\x in {1,...,3}
\node at (\x,6.5) {\u};

\foreach\x in {1,...,3}
\node at (\x,5.5) {\u};

\foreach\y in {1,...,3}
\node at (1,5.5-\y) {\u};

\node at (1,1.5) {\u};
\foreach\x in {2,3}
\node at (\x,1.5) {\d};

\foreach\x in {1,...,3}
\node at (\x,0.5) {\d};

\foreach\x in {1,...,3}
\node[label={below: \fs ${\color{red} \begin{array}{c} {\color{black} \b{y}_{\x}} \\ \shortuparrow \end{array} }$}] at (4-\x,0) {};
\end{tikzpicture}
\end{tabular}
\caption{The two recursion relations satisfied by $Z_{\rm HT}$, in the case $n=3$. On the left, the freezing procedure which gives rise to equation \eref{HT-rec1}. On the right, the freezing procedure which gives rise to equation \eref{HT-rec2}.}
\label{fig:HT-freeze} 
\end{figure}
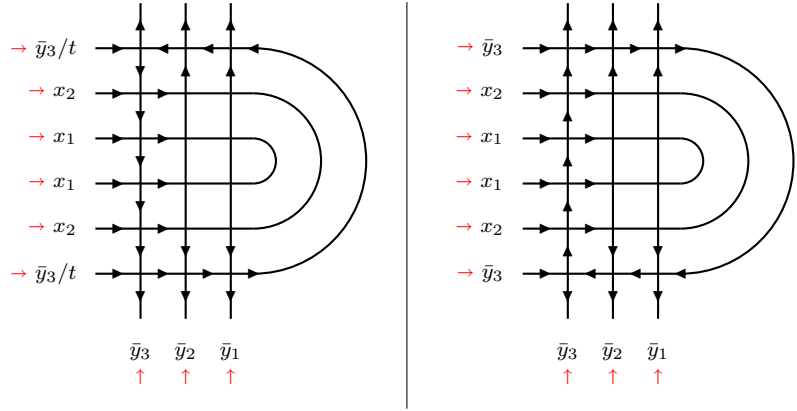

Multiplying by $(1-x_n y_n)^2$ and taking $x_n \rightarrow \b{y}_n$ eliminates the possibility that the bottom-left vertex of the lattice is a $b_{+}$ vertex. It must therefore be a $c_{+}$ vertex. Once again, this forces a subset of the vertices to freeze out, as is shown on the right of Figure \ref{fig:HT-freeze}. The Boltzmann weights of these frozen vertices contribute the total factor
\begin{align*}
(1-t)^2
\prod_{i=1}^{n-1}
\frac{(1-ty_i \b{y}_n)^2}{(1-y_i \b{y}_n)^2}
\frac{(1-tx_i y_n)^2}{(1-x_i y_n)^2},
\end{align*}
while the unfrozen region again represents $Z^{(n-1)}_{\rm HT}$. This yields the second recursion relation \eref{HT-rec2}.

\item The $n=1$ case is small enough to be calculated explicitly:
\begin{center}
\begin{tikzpicture}[scale=0.6]
\node at (-3,0.5) {$Z^{(1)}_{\rm HT}\ \ \ =$};
\draw[thick, smooth] (0,0)--(3/2,0);
\draw[thick, smooth] (0,1)--(3/2,1);
\node[label={left: \fs ${\color{red} \shortrightarrow} \ x_1$}] at (0,0) {};
\node[label={left: \fs ${\color{red} \shortrightarrow} \ x_1$}] at (0,1) {};
\node at (0.5,0) {$\r$}; \node at (1.5,0) {\r};
\node at (0.5,1) {$\r$}; \node at (1.5,1) {\l};
\draw[thick, smooth] (3/2,0) arc (-90:90:0.5);
\draw[thick, smooth] (1,-1)--(1,2);
\node[label={below: \fs ${\color{red} \begin{array}{c} {\color{black} \b{y}_1} \\ \shortuparrow \end{array} }$}] at (1,-1) {};
\node at (1,-0.5) {$\d$};
\node at (1,0.5) {\d};
\node at (1,1.5) {$\u$};
\node at (3,0.5) {$+$};
\end{tikzpicture}
\quad
\begin{tikzpicture}[scale=0.6]
\draw[thick, smooth] (0,0)--(3/2,0);
\draw[thick, smooth] (0,1)--(3/2,1);
\node[label={left: \fs ${\color{red} \shortrightarrow} \ x_1$}] at (0,0) {};
\node[label={left: \fs ${\color{red} \shortrightarrow} \ x_1$}] at (0,1) {};
\node at (0.5,0) {$\r$}; \node at (1.5,0) {\l};
\node at (0.5,1) {$\r$}; \node at (1.5,1) {\r};
\draw[thick, smooth] (3/2,0) arc (-90:90:0.5);
\draw[thick, smooth] (1,-1)--(1,2);
\node[label={below: \fs ${\color{red} \begin{array}{c} {\color{black} \b{y}_1} \\ \shortuparrow \end{array} }$}] at (1,-1) {};
\node at (1,-0.5) {$\d$};
\node at (1,0.5) {\u};
\node at (1,1.5) {$\u$};
\end{tikzpicture}
\end{center}
\noindent Substituting the Boltzmann weights into this expression, we obtain
\begin{align*}
Z^{(1)}_{\rm HT}
=
\frac{(1-t) \sqrt{t}}{1-x_1 y_1}
+
\frac{1-tx_1 y_1}{1-x_1 y_1}
\frac{(1-t)}{1-x_1 y_1}
=
\frac{(1-t)(1+\sqrt{t})(1-\sqrt{t} x_1 y_1)}{(1-x_1 y_1)^2}.
\end{align*}

\end{enumerate}
\end{proof}

\begin{thm}[Kuperberg]
The partition function on the half-turn symmetric lattice is given by a product of determinants:
\begin{multline}
\label{ZHT}
Z_{\rm HT}(x_1,\dots,x_n; y_1,\dots,y_n;t)
=
\\
\frac{\prod_{i,j=1}^{n} (1-t x_i y_j)^2}
{\prod_{1\leq i<j \leq n} (x_i-x_j)^2 (y_i-y_j)^2}
\det_{1 \leq i,j \leq n}\left[ \frac{(1-t)}{(1-x_i y_j)(1-t x_i y_j)}\right]
\det_{1 \leq i,j \leq n}\left[ \frac{1+\sqrt{t}-(\sqrt{t}+t) x_i y_j}{(1-x_i y_j)(1-t x_i y_j)} \right].
\end{multline}
\end{thm}

\begin{proof}
Taking the expression \eref{ZHT} as an Ansatz for the partition function, it is straightforward to show that it satisfies the four properties of Lemma \ref{lem:HT-prop}. Furthermore, these properties are uniquely-determining by the usual arguments of Lagrange interpolation.
\end{proof}

\subsection{$u$-deformed Cauchy identity at Schur and Hall--Littlewood level}
\label{ssec:u-def-ASM}

\begin{cor}

The Cauchy identity for Schur polynomials can be doubly refined, by the introduction of two deformation parameters $t$ and $u$:
\begin{align}
\label{ASM-cor1}
\sum_{\lambda}
\prod_{i=1}^{n}
(1 - u t^{\lambda_i - i + n})
s_{\lambda}(x_1,\dots,x_n)
s_{\lambda}(y_1,\dots,y_n)
=
\frac{1}{\Delta(x) \Delta(y)}
\det_{1 \leq i,j \leq n}
\left[
\frac{1-u + (u-t) x_i y_j}{(1-x_i y_j) (1-t x_i y_j)}
\right].
\end{align}

\end{cor}

\begin{cor}

The Cauchy identity for Hall--Littlewood polynomials can be refined by the introduction of a single deformation parameter $u$:
\begin{multline}
\label{ASM-cor2}
\sum_{\lambda}
\prod_{i=1}^{m_0(\lambda)}
(1 - u t^{i-1})
b_{\lambda}(t)
P_{\lambda}(x_1,\dots,x_n;t)
P_{\lambda}(y_1,\dots,y_n;t)
=
\\
\frac{\prod_{i,j=1}^{n} (1- t x_i y_j)}{\Delta(x) \Delta(y)}
\det_{1 \leq i,j \leq n}
\left[
\frac{1-u + (u-t) x_i y_j}{(1-x_i y_j) (1-t x_i y_j)}
\right].
\end{multline}

\end{cor}
These identities are recovered as the special cases $q=t$ and $q=0$ of Theorem \ref{mac-cauch-thm}. In view of the fact that the Schur polynomials on the left hand side of \eref{ASM-cor1} are determinants, it is possible to prove \eref{ASM-cor1} by completely elementary means via the Cauchy--Binet identity.\footnote{Dividing equation \eref{ASM-cor1} by $(1-t)^n$, letting $u = t^{-z}$ and taking the limit $t \rightarrow 1$, the left hand side of \eref{ASM-cor1} becomes
\begin{align*}
\sum_{\lambda}
\prod_{i=1}^{n}
(\lambda_i - i + n - z)
s_{\lambda}(x_1,\dots,x_n)
s_{\lambda}(y_1,\dots,y_n). 
\end{align*}
This limiting form was already investigated in \cite{agj}, although there the right hand side was not expressed in determinant form. We thank F. Jouhet for showing us this reference.}  
On the other hand \eref{ASM-cor2} remains a highly non-trivial identity, admitting no simple proof (that we know of) outside of the use of Macdonald difference operators, or a Lagrange interpolation style of proof similar to that of Section \ref{proof-ASM}.

One can consider further specializations of \eref{ASM-cor2}, by setting the free parameter $u$ to various values. Setting $u=0$ produces the Cauchy identity for 
Hall--Littlewood polynomials, while setting $u=t$ reproduces equation \eref{HL-ASM}. Finally, in the case $u=-\sqrt{t}$ we recover one half of the factors in equation \eref{ZHT} for $Z_{\rm HT}$ (the remaining factors being those of the domain wall partition function).

\section{$u$-deformed Littlewood identity and doubly off-diagonally symmetric alternating sign matrices}
\label{sec:OOASM}

This section proceeds largely in parallel with the previous one. The goal here is to study the $q=0$ specialization of the conjecture \eref{mac-osasm}, and its connection with the six-vertex model. The relevant domain in this case is the {\it doubly off-diagonally symmetric} lattice, whose configurations are in one-to-one correspondence with 
off-diagonally/off-anti-diagonally symmetric alternating sign matrices \cite{kup2}.

\subsection{Corner vertices}

In this section it is necessary to introduce {\it boundary vertices}, which consist of a single lattice line making a turn through a node. The type of boundary vertices of interest to us are the {\it corner vertices}, shown in Figure \ref{fig:corner}.

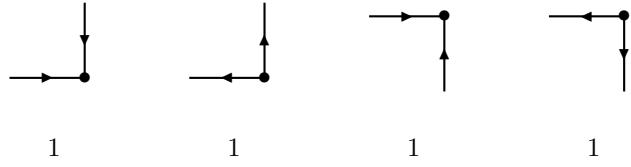
\begin{figure}[H]
\begin{tabular}{cccc}
\begin{tikzpicture}[>=stealth]
\draw[thick]
(0,0) to (1,0) to (1,1);
\node at (1,0) {$\bullet$};
\node at (0.5,0) {\r};
\node at (1,0.5) {\d};
\end{tikzpicture}
\quad\quad\quad&
\begin{tikzpicture}[>=stealth]
\draw[thick]
(0,0) to (1,0) to (1,1);
\node at (1,0) {$\bullet$};
\node at (0.5,0) {\l};
\node at (1,0.5) {\u};
\end{tikzpicture}
\quad\quad\quad&
\begin{tikzpicture}[rotate=90]
\draw[thick]
(0,0) to (1,0) to (1,1);
\node at (1,0) {$\bullet$};
\node at (0.5,0) {\u};
\node at (1,0.5) {\r};
\end{tikzpicture}
\quad\quad\quad&
\begin{tikzpicture}[rotate=90]
\draw[thick]
(0,0) to (1,0) to (1,1);
\node at (1,0) {$\bullet$};
\node at (0.5,0) {\d};
\node at (1,0.5) {\l};
\end{tikzpicture}
\\ \\
1 
\quad\quad\quad&
1 
\quad\quad\quad&
1 
\quad\quad\quad&
1 
\end{tabular}
\caption{The corner vertices, which in this work are all assigned equal Boltzmann weights. In \cite{kup2}, the corner vertices are assigned the Boltzmann weights $b,\b{b},c,\b{c}$ respectively, where $b$ and $c$ are arbitrary parameters (not to be confused with $b$ and $c$ type vertices). For simplicity, here we choose both of these parameters to be equal to 1.}
\label{fig:corner}
\end{figure}
Together with the bulk vertices of Figure \ref{fig:6vertices}, the corner vertices satisfy the {\it corner reflection} equations shown in Figure \ref{fig:corner-YB}. The corner reflection equations, in conjunction with the regular Yang--Baxter equation, ensure that the $Z_{\rm OO}$ partition function that we subsequently study is symmetric in its rapidities.

\begin{figure}[H]
\begin{tabular}{cc|cc}
\begin{tikzpicture}[scale=0.3]

\draw[thick,smooth] (-0.5,0)--(0,0);
\draw[thick,smooth] (-0.5,2)--(0,2);
\draw[thick, smooth] (0,2) arc (90:45:1);
\draw[thick, smooth] (0,0) arc (-90:-45:1);
\draw[thick, smooth] ({sqrt(2)/2},{1+sqrt(2)/2})--({3*sqrt(2)/2},{1-sqrt(2)/2});
\draw[thick, smooth] ({sqrt(2)/2},{1-sqrt(2)/2})--({3*sqrt(2)/2},{1+sqrt(2)/2});
\draw[thick, smooth] ({2*sqrt(2)},2) arc (90:135:1);
\draw[thick, smooth] ({2*sqrt(2)},0) arc (-90:-135:1);
\draw[thick, smooth] ({2*sqrt(2)}, 2)--({2*sqrt(2)+0.5}, 2);
\draw[thick, smooth] ({2*sqrt(2)}, 0)--({2*sqrt(2)+0.5}, 0);
\draw[thick, smooth] ({2*sqrt(2)+0.5}, 0)--({2*sqrt(2)+0.5}, 4);
\draw[thick, smooth] ({2*sqrt(2)+0.5+2}, 2)--({2*sqrt(2)+0.5+2}, 4);
\draw[thick, smooth] ({2*sqrt(2)+0.5}, 2)--({2*sqrt(2)+0.5+2}, 2);
\node at ({2*sqrt(2)+0.5}, 0) {$\bullet$};
\node at ({2*sqrt(2)+0.5+2}, 2) {$\bullet$};
\node[label={left: \fs ${\color{red} \shortrightarrow} \ x$}] at (-0.5, 0) {};
\node[label={left: \fs ${\color{red} \shortrightarrow} \ y$}] at (-0.5, 2) {};
\node[label={above: \fs ${\color{red} \begin{array}{c}  \shortuparrow \\ {\color{black} \b{y}} \end{array} }$}] at ({2*sqrt(2)+0.5},4) {};
\node[label={above: \fs ${\color{red} \begin{array}{c}  \shortuparrow \\ {\color{black} \b{x}} \end{array} }$}] at ({2*sqrt(2)+0.5+2},4) {};

\node at ({2*sqrt(2) + 2 + 0.5 + 4}, 2) {$=$};

\draw[thick, smooth] ({2*sqrt(2) + 2 + 0.5 + 8}, 0)--({2*sqrt(2) + 2 + 0.5 + 10}, 0);
\draw[thick, smooth] ({2*sqrt(2) + 2 + 0.5 + 8}, 2)--({2*sqrt(2) + 2 + 0.5 + 12}, 2);
\draw[thick, smooth] ({2*sqrt(2) + 2 + 0.5 + 10}, 0)--({2*sqrt(2) + 2 + 0.5 + 10}, 2);
\draw[thick, smooth] 
({2*sqrt(2)+2+0.5 + 2 + 4 + 6},{2  + 2*sqrt(2) + 1})--({2*sqrt(2)+2+0.5 + 2 + 4 + 6},{2-0.5 + 2*sqrt(2) + 1});
\draw[thick, smooth] 
({2*sqrt(2)+2+0.5 + 2 + 2 + 6},{2 + 2*sqrt(2) + 1})--({2*sqrt(2)+2+0.5 + 2 + 2 + 6},{2-0.5 + 2*sqrt(2) + 1});
\draw[thick, smooth] ({2*sqrt(2)+2+0.5 + 2 + 2 + 6},{2-0.5 + 2*sqrt(2) + 1}) arc (-180:-135:1);
\draw[thick, smooth] ({2*sqrt(2)+2+0.5 + 2 + 4 + 6},{2-0.5 + 2*sqrt(2) + 1}) arc (0:-45:1);
\draw[thick, smooth] 
({2*sqrt(2)+2+0.5 + 2 + 2 + 1 + 6 - sqrt(2)/2},{2-0.5-sqrt(2)/2 + 2*sqrt(2) + 1})--
({2*sqrt(2)+2+0.5 + 2 + 2 + 1 + 6 + sqrt(2)/2},{2-0.5-3*sqrt(2)/2 + 2*sqrt(2) + 1});
\draw[thick, smooth] 
({2*sqrt(2)+2+0.5 + 2 + 2 + 1 + 6 - sqrt(2)/2},{2-0.5-3*sqrt(2)/2 + 2*sqrt(2) + 1})--
({2*sqrt(2)+2+0.5 + 2 + 2 + 1 + 6 + sqrt(2)/2},{2-0.5-sqrt(2)/2 + 2*sqrt(2) + 1});
\draw[thick, smooth] ({2*sqrt(2)+2+0.5 + 2 + 2 + 6},{2-0.5-2*sqrt(2) + 2*sqrt(2) + 1}) arc (180:135:1);
\draw[thick, smooth] ({2*sqrt(2)+2+0.5 + 2 + 4 + 6},{2-0.5-2*sqrt(2) + 2*sqrt(2) + 1}) arc (0:45:1);
\draw[thick, smooth]  ({2*sqrt(2)+2+0.5 + 2 + 2 + 6},{2-0.5-2*sqrt(2) + 2*sqrt(2) + 1})
--({2*sqrt(2)+2+0.5 + 2 + 2 + 6},{2-1-2*sqrt(2) + 2*sqrt(2) + 1});
\draw[thick, smooth]  ({2*sqrt(2)+2+0.5 + 2 + 4 + 6},{2-0.5-2*sqrt(2) + 2*sqrt(2) + 1})
--({2*sqrt(2)+2+0.5 + 2 + 4 + 6},{2-1-2*sqrt(2) + 2*sqrt(2) + 1});
\node at ({2*sqrt(2) + 2 + 0.5 + 10}, 0) {$\bullet$};
\node at ({2*sqrt(2) + 2 + 0.5 + 12}, 2) {$\bullet$};

\node[label={left: \fs ${\color{red} \shortrightarrow} \ x$}] at ({2*sqrt(2) + 2 + 0.5 + 8}, 0) {};
\node[label={left: \fs ${\color{red} \shortrightarrow} \ y$}] at ({2*sqrt(2) + 2 + 0.5 + 8}, 2) {};
\node[label={above: \fs ${\color{red} \begin{array}{c} \shortuparrow \\ {\color{black} \b{y}} \end{array} }$}] at ({2*sqrt(2)+2+0.5 + 2 + 2 + 6},{2 + 2*sqrt(2) + 1}) {};
\node[label={above: \fs ${\color{red} \begin{array}{c} \shortuparrow \\ {\color{black} \b{x}} \end{array} }$}] at ({2*sqrt(2)+2+0.5 + 2 + 4 + 6},{2 + 2*sqrt(2) + 1}) {};

\end{tikzpicture}
&
&
&
\begin{tikzpicture}[scale=0.3]

\draw[thick, smooth] (0,{3+2*sqrt(2)})--(2,{3+2*sqrt(2)});
\draw[thick, smooth] (0,{1+2*sqrt(2)})--(4,{1+2*sqrt(2)});
\draw[thick, smooth] (2,{3+2*sqrt(2)})--(2,{0.5+2*sqrt(2)});
\draw[thick, smooth] (4,{1+2*sqrt(2)})--(4,{0.5+2*sqrt(2)});
\draw[thick, smooth] (2,{-0.5 + 2*sqrt(2) + 1}) arc (-180:-135:1);
\draw[thick, smooth] (4,{-0.5 + 2*sqrt(2) + 1}) arc (0:-45:1);
\draw[thick, smooth] ({3 - sqrt(2)/2},{-0.5-sqrt(2)/2 + 2*sqrt(2) + 1})--({3 + sqrt(2)/2},{-0.5-3*sqrt(2)/2 + 2*sqrt(2) + 1});
\draw[thick, smooth] ({3 - sqrt(2)/2},{-0.5-3*sqrt(2)/2 + 2*sqrt(2) + 1})--({3 + sqrt(2)/2},{-0.5-sqrt(2)/2 + 2*sqrt(2) + 1});
\draw[thick, smooth] (2,0.5) arc (180:135:1);
\draw[thick, smooth] (4,0.5) arc (0:45:1);
\draw[thick, smooth]  (2,0)-- (2,0.5);
\draw[thick, smooth]  (4,0)-- (4,0.5);
\node at (2,{3+2*sqrt(2)}) {$\bullet$};
\node at (4,{1+2*sqrt(2)}) {$\bullet$};
\node[label={left: \fs ${\color{red} \shortrightarrow} \ x$}] at (0, {3+2*sqrt(2)}) {};
\node[label={left: \fs ${\color{red} \shortrightarrow} \ y$}] at (0, {1+2*sqrt(2)}) {};
\node[label={below: \fs ${\color{red} \begin{array}{c} \shortuparrow \\ {\color{black} \b{y}} \end{array} }$}] at (2,0) {};
\node[label={below: \fs ${\color{red} \begin{array}{c} \shortuparrow \\ {\color{black} \b{x}} \end{array} }$}] at (4,0) {};

\node at (7,{(3+2*sqrt(2))/2}) {$=$};

\draw[thick, smooth] (10, {3+2*sqrt(2)})--(10.5, {3+2*sqrt(2)});
\draw[thick, smooth] (10, {1+2*sqrt(2)})--(10.5, {1+2*sqrt(2)});
\draw[thick, smooth] ({10.5 + 2*sqrt(2)}, {3+2*sqrt(2)})--({11 + 2*sqrt(2)}, {3+2*sqrt(2)});
\draw[thick, smooth] ({10.5 + 2*sqrt(2)}, {1+2*sqrt(2)})--({13 + 2*sqrt(2)}, {1+2*sqrt(2)});
\draw[thick, smooth] ({11 + 2*sqrt(2)}, {3+2*sqrt(2)})--({11 + 2*sqrt(2)}, {2*sqrt(2)-1});
\draw[thick, smooth] ({13 + 2*sqrt(2)}, {1+2*sqrt(2)})--({13 + 2*sqrt(2)}, {2*sqrt(2)-1});
\draw[thick, smooth] (10.5,{3+2*sqrt(2)}) arc (90:45:1);
\draw[thick, smooth] (10.5,{1+2*sqrt(2)}) arc (-90:-45:1);
\draw[thick, smooth] ({10.5 + sqrt(2)/2},{2+5*sqrt(2)/2})--({10.5 + 3*sqrt(2)/2},{2+3*sqrt(2)/2});
\draw[thick, smooth] ({10.5 + sqrt(2)/2},{2+3*sqrt(2)/2})--({10.5 + 3*sqrt(2)/2},{2+5*sqrt(2)/2});
\draw[thick, smooth] ({10.5+2*sqrt(2)},{3+2*sqrt(2)}) arc (90:135:1);
\draw[thick, smooth] ({10.5+2*sqrt(2)},{1+2*sqrt(2)}) arc (-90:-135:1);
\node[label={left: \fs ${\color{red} \shortrightarrow} \ x$}] at (10, {3+2*sqrt(2)}) {};
\node[label={left: \fs ${\color{red} \shortrightarrow} \ y$}] at (10, {1+2*sqrt(2)}) {};
\node[label={below: \fs ${\color{red} \begin{array}{c}  \shortuparrow \\ {\color{black} \b{y}} \end{array} }$}] at ({2*sqrt(2)+11},{2*sqrt(2)-1}) {};
\node[label={below: \fs ${\color{red} \begin{array}{c}  \shortuparrow \\ {\color{black} \b{x}} \end{array} }$}] at ({2*sqrt(2)+13},{2*sqrt(2)-1}) {};
\node at ({11+2*sqrt(2)},{3+2*sqrt(2)}) {$\bullet$};
\node at ({13+2*sqrt(2)},{1+2*sqrt(2)}) {$\bullet$};
\end{tikzpicture}

\end{tabular}
\caption{Reflection equations for corner vertices. Notice that, due to the orientation of the lattice lines, the equation on the right is not simply a 90\degree\ rotation of that on the left.}
\label{fig:corner-YB}
\end{figure}
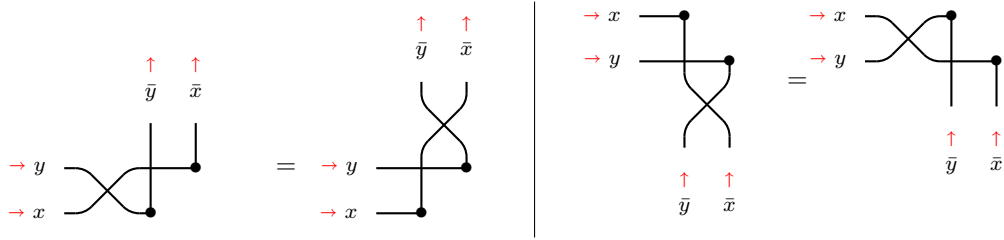 

\subsection{Partition function on doubly off-diagonally symmetric lattice}

We now study the six-vertex model on a doubly off-diagonally symmetric domain, as shown in Figure \ref{fig:OO}. The corresponding alternating sign matrices have off-diagonal/off-anti-diagonal symmetry \cite{kup2}.


\begin{figure}[H]
\begin{tikzpicture}[scale=0.6]

\foreach\x in {1,...,4}
{\draw[thick] (0,\x-1) -- (\x,\x-1);
\node[label={left: \fs ${\color{red} \shortrightarrow} \ x_{\x}$}] at (0,4-\x) {};
\draw[thick] (0,\x+3) -- (5-\x,\x+3);
\node[label={left: \fs ${\color{red} \shortrightarrow} \ x_{\x}$}] at (0,\x+3) {};
\node at (\x,\x-1) {$\bullet$}; \node at (5-\x,\x+3) {$\bullet$};
\node[label={below: \fs ${\color{red} \begin{array}{c}  {\color{black} \b{x}_{\x}} \\ \shortuparrow \end{array} }$}] at (5-\x,4-\x) {};}


\foreach\y in {0,...,7}
\node at (0.5,\y) {\r};

\foreach\x in {1,...,4}
\draw[thick, smooth] (\x,\x-1) -- (\x,8-\x);
\end{tikzpicture}
\caption{The partition function $Z_{\rm OO}$ of the six-vertex model with doubly 
off-diagonally symmetric boundary conditions, in the case $n=2$. Horizontal lattice lines are oriented from left to right, vertical lattice lines are oriented from bottom to top, and the variables associated to the vertical lines are reciprocated.}
\label{fig:OO}
\end{figure}
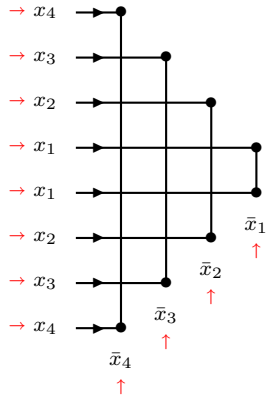

\begin{lem}
\label{lem:OO-prop}
The partition function 
$Z^{(2n)}_{\rm OO} = Z_{\rm OO}(x_1,\dots,x_{2n};t)$ as defined in Figure \ref{fig:OO} satisfies four properties:
\begin{enumerate}[label=\bf\arabic*.]

\item Multiplying by $\prod_{1\leq i<j \leq 2n} (1-x_i x_j)^2$, it is a polynomial in 
$x_{2n}$ of degree $4n-3$.

\item It is symmetric in $\{x_1,\dots,x_{2n}\}$.

\item It obeys the recursion relations
\begin{align}
\label{OO-rec1}
Z^{(2n)}_{\rm OO}
\Big|_{x_{2n} = \b{x}_{2n-1}/t}
&=
- t^{4n-5/2}
Z^{(2n-2)}_{\rm OO},
\\
\label{OO-rec2}
\lim_{x_{2n} \rightarrow \b{x}_{2n-1}}
\Big(
(1-x_{2n-1} x_{2n})^2
Z^{(2n)}_{\rm OO}
\Big)
&= 
(1-t)^2
\prod_{i=1}^{2n-2}
\frac{(1-tx_i \b{x}_{2n-1})^2}{(1-x_i \b{x}_{2n-1})^2}
\frac{(1-tx_i x_{2n-1})^2}{(1-x_i x_{2n-1})^2}
Z^{(2n-2)}_{\rm OO}.
\end{align}

\item When $n=1$, it is given explicitly by
\begin{align*}
Z^{(2)}_{\rm OO}
=
\frac{(1-t)(1+\sqrt{t})(1 - \sqrt{t} x_1 x_2)}
{(1-x_1 x_2)^2}.
\end{align*}

\end{enumerate}
\end{lem}

\begin{proof}

The proof is again based closely on the lattice definition of $Z_{\rm OO}$. 

\begin{enumerate}[label=\bf\arabic*.]

\item Multiplying the partition function by $\prod_{1\leq i<j \leq 2n} (1-x_i x_j)^2$ is equivalent to renormalizing every vertex by $(1-x_i x_j)$. This makes all Boltzmann weights degree-1 polynomials in $x_i$, with the sole exception of the $c_{+}$ vertex (which is a constant). Examining the left-most vertical line of Figure \ref{fig:OO}, which gives rise to all $x_{2n}$ dependence, we see that exactly one $c_{+}$ vertex will occur on this line. Hence $Z^{(2n)}_{\rm OO}$ is a polynomial in $x_{2n}$ of degree $4n-3$. 

\item Symmetry in the $x$ variables can be deduced using both the Yang--Baxter equation and the two corner reflection equations in Figure \ref{fig:corner-YB}. These equations, in combination, allow for any two lattice lines bearing the labels $x_i$ and $x_j$ to be interchanged.

\item Consider the top-most bulk vertex in Figure \ref{fig:OO}. Setting 
$x_{2n} = \b{x}_{2n-1}/t$ rules out the possibility that this is an $a_{+}$ vertex. It must therefore be a $c_{+}$ vertex, and this causes a subset of the vertices to be in a frozen configuration, as shown on the left of Figure \ref{fig:OO-freeze}. The total contribution from these frozen vertices is the weight 
$-t (\sqrt{t})^{8n-7} = -t^{4n-5/2}$, while the surviving region is simply $Z^{(2n-2)}_{\rm OO}$. Hence we obtain equation \eref{OO-rec1}.


\begin{figure}[H]
\begin{tabular}{cc|cc}
\begin{tikzpicture}[scale=0.6]
\foreach\x in {1,...,4}
{\draw[thick] (0,\x-1) -- (\x,\x-1);
\draw[thick] (0,\x+3) -- (5-\x,\x+3);
\node at (\x,\x-1) {$\bullet$}; \node at (5-\x,\x+3) {$\bullet$};}

\foreach\x in {1,...,3}
{\node[label={left: \fs ${\color{red} \shortrightarrow} \ x_{\x}$}] at (0,4-\x) {};
\node[label={left: \fs ${\color{red} \shortrightarrow} \ x_{\x}$}] at (0,\x+3) {};
\node[label={below: \fs ${\color{red} \begin{array}{c}  {\color{black} \b{x}_{\x}} \\ \shortuparrow \end{array} }$}] at (5-\x,4-\x) {};}

\node[label={left: \fs ${\color{red} \shortrightarrow} \ \b{x}_{3}/t$}] at (0,0) {};
\node[label={left: \fs ${\color{red} \shortrightarrow} \ \b{x}_{3}/t$}] at (0,7) {};
\node[label={below: \fs ${\color{red} \begin{array}{c}  {\color{black} t x_3 } \\ \shortuparrow \end{array} }$}] at (1,0) {};


\foreach\y in {0,...,7}
\node at (0.5,\y) {\r};

\node at (1.5,6) {\l};
\foreach\y in {1,...,5}
\node at (1.5,\y) {\r};

\foreach\y in {2,...,5}
\node at (2.5,\y) {\r};

\foreach\x in {1,...,4}
\draw[thick, smooth] (\x,\x-1) -- (\x,8-\x);


\node at (1,0.5) {\d};

\foreach\y in {2,...,6}
{\foreach\x in {1,2}
\node at (\x,-0.5+\y) {\d};}

\node at (1,6.5) {\u};

\end{tikzpicture}
&
&
&
\begin{tikzpicture}[scale=0.6]
\foreach\x in {1,...,4}
{\draw[thick] (0,\x-1) -- (\x,\x-1);
\draw[thick] (0,\x+3) -- (5-\x,\x+3);
\node at (\x,\x-1) {$\bullet$}; \node at (5-\x,\x+3) {$\bullet$};}

\foreach\x in {1,...,3}
{\node[label={left: \fs ${\color{red} \shortrightarrow} \ x_{\x}$}] at (0,4-\x) {};
\node[label={left: \fs ${\color{red} \shortrightarrow} \ x_{\x}$}] at (0,\x+3) {};
\node[label={below: \fs ${\color{red} \begin{array}{c}  {\color{black} \b{x}_{\x}} \\ \shortuparrow \end{array} }$}] at (5-\x,4-\x) {};}

\node[label={left: \fs ${\color{red} \shortrightarrow} \ \b{x}_{3}$}] at (0,0) {};
\node[label={left: \fs ${\color{red} \shortrightarrow} \ \b{x}_{3}$}] at (0,7) {};
\node[label={below: \fs ${\color{red} \begin{array}{c}  {\color{black} x_3 } \\ \shortuparrow \end{array} }$}] at (1,0) {};


\foreach\y in {0,...,7}
\node at (0.5,\y) {\r};

\foreach\y in {2,...,6}
\node at (1.5,\y) {\r};
\node at (1.5,1) {\l};

\foreach\y in {2,...,5}
\node at (2.5,\y) {\r};

\foreach\x in {1,...,4}
\draw[thick, smooth] (\x,\x-1) -- (\x,8-\x);


\node at (1,0.5) {\d};

\foreach\y in {2,...,6}
{\foreach\x in {1,2}
\node at (\x,-0.5+\y) {\u};}

\node at (1,6.5) {\u};

\end{tikzpicture}
\end{tabular}
\caption{The two recursion relations satisfied by $Z_{\rm OO}$, in the case $n=2$. On the left, the freezing procedure which produces equation \eref{OO-rec1}. On the right, the freezing procedure which produces equation \eref{OO-rec2}.}
\label{fig:OO-freeze}
\end{figure}
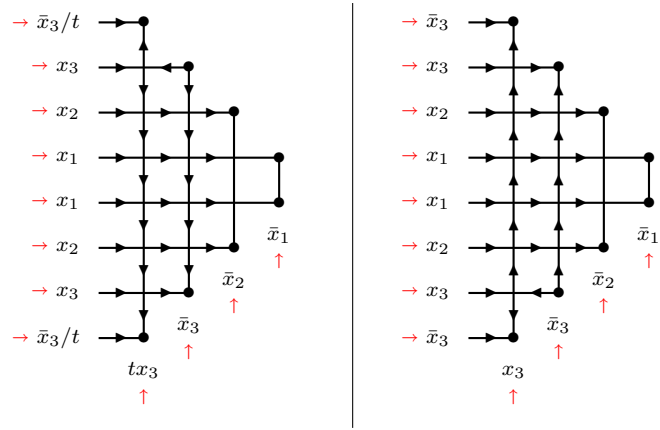
A similar argument applies to the bottom-most bulk vertex in Figure \ref{fig:OO}. After multiplying by $(1-x_{2n-1} x_{2n})^2$ and sending $x_{2n} \rightarrow \b{x}_{2n-1}$ this cannot be a $b_{+}$ vertex, meaning that it must be a $c_{+}$ vertex. This causes some of the vertices to freeze, as shown on the right of Figure 
\ref{fig:OO-freeze}, and they contribute a total weight of
\begin{align*}
(1-t)^2
\prod_{i=1}^{2n-2}
\frac{(1-tx_i \b{x}_{2n-1})^2}{(1-x_i \b{x}_{2n-1})^2}
\frac{(1-tx_i x_{2n-1})^2}{(1-x_i x_{2n-1})^2},
\end{align*}
with the non-frozen part of the lattice representing $Z^{(2n-2)}_{\rm OO}$. Hence we recover equation 
\eref{OO-rec2}.

\item Calculating the $n=1$ case explicitly, we find that
\begin{center}
\begin{tikzpicture}[scale=0.6]
\node at (-3,1.5) {$Z^{(2)}_{\rm OO}\ \ \ =$};
\foreach\x in {1,...,2}
{\draw[thick, smooth] (0,\x-1) -- (\x,\x-1);
\node[label={left: \fs ${\color{red} \shortrightarrow} \ x_{\x}$}] at (0,2-\x) {};
\draw[thick, smooth] (0,\x+1) -- (3-\x,\x+1);
\node[label={left: \fs ${\color{red} \shortrightarrow} \ x_{\x}$}] at (0,1+\x) {};
\node at (\x,\x-1) {$\bullet$}; \node at (3-\x,\x+1) {$\bullet$};
\node[label={below: \fs ${\color{red} \begin{array}{c}  {\color{black} \b{x}_{\x}} \\ \shortuparrow \end{array} }$}] at (3-\x,2-\x) {};}
\foreach\x in {1,...,4}
\node at (0.5,\x-1) {$\r$};
\node at (1.5,2) {\l};
\node at (1.5,1) {\r};
\foreach\x in {1,...,2}
\draw[thick, smooth] (\x,\x-1) -- (\x,4-\x);
\node at (1,0.5) {\d};
\node at (1,1.5) {\d};
\node at (1,2.5) {\u};
\node at (2,1.5) {\d};
\node at (3,1.5) {$+$};
\end{tikzpicture}
\quad
\begin{tikzpicture}[scale=0.6]
\foreach\x in {1,...,2}
{\draw[thick, smooth] (0,\x-1) -- (\x,\x-1);
\node[label={left: \fs ${\color{red} \shortrightarrow} \ x_{\x}$}] at (0,2-\x) {};
\draw[thick, smooth] (0,\x+1) -- (3-\x,\x+1);
\node[label={left: \fs ${\color{red} \shortrightarrow} \ x_{\x}$}] at (0,1+\x) {};
\node at (\x,\x-1) {$\bullet$}; \node at (3-\x,\x+1) {$\bullet$};
\node[label={below: \fs ${\color{red} \begin{array}{c}  {\color{black} \b{x}_{\x}} \\ \shortuparrow \end{array} }$}] at (3-\x,2-\x) {};}
\foreach\x in {1,...,4}
\node at (0.5,\x-1) {$\r$};
\node at (1.5,2) {\r};
\node at (1.5,1) {\l};
\foreach\x in {1,...,2}
\draw[thick, smooth] (\x,\x-1) -- (\x,4-\x);
\node at (1,0.5) {\d};
\node at (1,1.5) {\u};
\node at (1,2.5) {\u};
\node at (2,1.5) {\u};
\end{tikzpicture}
\end{center}
\noindent Substituting the explicit expression for the Boltzmann weights, we obtain
\begin{align*}
Z^{(2)}_{\rm OO}
=
\frac{(1-t)\sqrt{t}}{1-x_1 x_2}
+
\frac{1-tx_1 x_2}{1-x_1 x_2}
\frac{(1-t)}{1-x_1 x_2}
=
\frac{(1-t)(1+\sqrt{t})(1 - \sqrt{t} x_1 x_2)}
{(1-x_1 x_2)^2}.
\end{align*} 

\end{enumerate}

\end{proof}

\begin{thm}[Kuperberg]
The partition function on the doubly off-diagonally symmetric lattice is given by a product of Pfaffians:
\begin{multline}
\label{ZOO}
Z_{\rm OO}(x_1,\dots,x_{2n};t)
=
\\
\prod_{1 \leq i<j \leq 2n}
\frac{(1-t x_i x_j)^2}{(x_i-x_j)^2}
\pf_{1\leq i < j \leq 2n}
\left[
\frac{(1-t)(x_i-x_j)}{(1-x_i x_j)(1-t x_i x_j)}
\right]
\pf_{1\leq i < j \leq 2n}
\left[
\frac{(1+\sqrt{t} - (\sqrt{t} + t) x_i x_j)(x_i-x_j)}{(1-x_i x_j)(1-t x_i x_j)}
\right].
\end{multline}

\end{thm}

\begin{proof}
One needs only to check that \eref{ZOO} satisfies the four properties of Lemma \ref{lem:OO-prop}, since these properties uniquely determine $Z_{\rm OO}$. 
\end{proof}

\subsection{$u$-deformed Littlewood identity at Schur and Hall--Littlewood level}
\label{ssec:u-def-OSASM}

In this subsection we will refer to the following Littlewood identities for Schur and 
Hall--Littlewood polynomials \cite{mac}:
\begin{align}
\label{little1}
\sum_{\lambda'\ \text{even}}
s_{\lambda}(x_1,\dots,x_N)
&=
\prod_{1 \leq i<j \leq N}
\left(
\frac{1}
{1-x_i x_j}
\right),
\\
\label{little2}
\sum_{\lambda'\ \text{even}}
b^{\rm el}_{\lambda}(t)
P_{\lambda}(x_1,\dots,x_N;t)
&=
\prod_{1 \leq i<j \leq N}
\left(
\frac{1-t x_i x_j}
{1-x_i x_j}
\right),
\end{align}
where as always we take $N=2n$.

\begin{thm}
\label{u-def-little-s}
The Littlewood identity \eref{little1} for Schur polynomials can be doubly refined, by the introduction of two deformation parameters $t$ and $u$:
\begin{multline}
\label{OSASM-cor1}
\sum_{\lambda'\ {\rm even}}
\
\prod_{i=1}^{n}
(1 - u t^{\lambda_{2i} - 2i + 2n})
s_{\lambda}(x_1,\dots,x_{2n})
=
\\
\prod_{1 \leq i<j \leq 2n}
\frac{1}{(x_i - x_j)}
\pf_{1\leq i < j \leq 2n}
\left[
\frac{(x_i - x_j) (1-u + (u-t) x_i x_j)}
{(1-x_i x_j) (1-t x_i x_j)}
\right].
\end{multline}
\end{thm}

\begin{proof}
Using the Weyl determinant formula for $s_{\lambda}$ and multiplying equation \eref{OSASM-cor1} by the Vandermonde $\prod_{1 \leq i<j \leq 2n} (x_i - x_j)$, the left hand side may be written as
\begin{align*}
\sum_{\lambda'\ {\rm even}}
\
\prod_{i=1}^{n}
(1-ut^{\lambda_{2i} - 2i + 2n})
\det_{1\leq i,j \leq 2n}
\left[
x_j^{\lambda_i-i+2n}
\right]
=
\sum_{\substack{k_1 > \cdots > k_{2n} \geq 0 \\ k_{2i-1} = k_{2i}+1}}
\
\prod_{i=1}^{n}
(1-ut^{k_{2i}})
\det_{1\leq i,j \leq 2n}
\left[
x_j^{k_i}
\right],
\end{align*}
where we have made the change in summation indices $k_i = \lambda_i - i + 2n$. Owing to the Pfaffian factorization
\begin{align*}
\pf_{1\leq i < j \leq 2n}
\left[
\delta_{k_i,k_j+1}
(1-u t^{k_j})
\right]
=
\prod_{i = 2,4,6,\dots}^{2n}
\left(
\delta_{k_{i-1},k_i+1}
(1-u t^{k_i})
\right)
\end{align*}
we find that
\begin{multline*}
\sum_{\lambda'\ {\rm even}}
\
\prod_{i=1}^{n}
(1-ut^{\lambda_{2i} - 2i + 2n})
\det_{1\leq i,j \leq 2n}
\left[
x_j^{\lambda_i-i+2n}
\right]
=
\sum_{k_1 > \cdots > k_{2n} \geq 0}
\
\pf_{1\leq i < j \leq 2n}
\left[
\delta_{k_i,k_j+1}
(1-u t^{k_j})
\right]
\det_{1\leq i,j \leq 2n}
\left[
x_j^{k_i}
\right]
\\
=
\pf_{1\leq i < j \leq 2n}
\left[
\sum_{0 \leq k < l}
\delta_{k,l+1}
(1-ut^l)
(x_i^k x_j^l - x_i^l x_j^k)
\right]
=
\pf_{1\leq i < j \leq 2n}
\left[
\sum_{l=0}^{\infty}
(1-ut^l)
(x_i^{l+1} x_j^l - x_i^l x_j^{l+1})
\right], 
\end{multline*}
where we have used the Pfaffian analogue of the Cauchy--Binet identity to produce the second equality. Taking the formal power series expansion of $(x_i-x_j)(1-u+(u-t)x_i x_j)/((1-x_i x_j)(1-t x_i x_j))$, we obtain precisely the entries of the final Pfaffian.

\end{proof}

\begin{thm}
\label{u-def-little-hl}
The Littlewood identity \eref{little2} for Hall--Littlewood polynomials can be refined by the introduction of a single deformation parameter $u$\footnote{Equation \eref{OSASM-cor2} was originally conjectured by O. Warnaar in a private communication, after our first paper \cite{bw} appeared. This communication motivated much of the work that was performed in the current paper.}: 
\begin{multline}
\label{OSASM-cor2}
\sum_{\lambda'\ {\rm even}}
\
\prod_{j\ {\rm even}}^{m_0(\lambda)}
(1-u t^{j-2})
b^{\rm el}_{\lambda}(t)
P_{\lambda}(x_1,\dots,x_{2n};t)
=
\\
\prod_{1 \leq i<j \leq 2n}
\left(
\frac{1-t x_i x_j}{x_i - x_j}
\right)
\pf_{1\leq i < j \leq 2n}
\left[
\frac{(x_i - x_j) (1-u + (u-t) x_i x_j)}
{(1-x_i x_j) (1-t x_i x_j)}
\right].
\end{multline}
\end{thm}

\begin{proof}
The idea of the proof is similar to that of Theorem \ref{littlewood-thm}. For the sake of brevity, we will simply point out the places where the proof deviates from the scheme exposed in Section \ref{proof-OSASM}. 

We denote the left hand side of \eref{OSASM-cor2} by $\mathcal{G}_{2n}(x_1,\dots,x_{2n};u)$, and by comparing it with the proposed right hand side we find that necessarily:
\begin{enumerate}[label=\bf\arabic*.] 
\item $\mathcal{G}_{2n}$ is symmetric in $\{x_1,\dots,x_{2n}\}$. 
\item $\mathcal{G}_{2n} \times \prod_{1 \leq i<j \leq 2n} (1-x_i x_j) $ is a polynomial in $x_{2n}$ of degree $2n-1$.
\item $\mathcal{G}_{2n}|_{x_{2n} = 1/(t x_{2n-1})} = -ut^{2n-2} \mathcal{G}_{2n-2}$.
\item $\mathcal{G}_{2n}(0,\dots,0;u) = \prod_{i=1}^{n} (1-ut^{2i-2})$.
\item $\mathcal{G}_{2} = (1-u+(u-t)x_1 x_2)/(1-x_1 x_2)$.
\end{enumerate}
Each of these is an immediate property of the right hand side of \eref{OSASM-cor2}, with the exception of 
{\bf 4}, which at first glance seems to require taking a delicate limit. In fact property {\bf 4} can be quickly deduced by setting all $x_i = 0$ in 
(equation \eref{OSASM-cor1}) $\times \prod_{1 \leq i<j \leq 2n} (1-t x_i x_j)$. It is straightforward to show that these five properties uniquely determine $\mathcal{G}_{2n}$. We remark that the additional property {\bf 4} is needed here, because properties {\bf 1} and {\bf 3} determine $\mathcal{G}_{2n}$ at $2n-1$ values of $x_{2n}$, which only specifies it up to a constant. The value of the constant is fixed by property {\bf 4}.

Hence it is sufficient to show that the left hand side of \eref{OSASM-cor2} satisfies properties 
{\bf 1}--{\bf 5}. Properties {\bf 1} and {\bf 4} are trivial, while {\bf 5} follows from
\begin{align*}
\mathcal{G}_2(x_1,x_2;u)
=
(1-u) P_{(0,0)}
+
\sum_{k=1}^{\infty}
(1-t)
P_{(k,k)}(x_1,x_2;t)
&=
(1-u)
+
(1-t)
\sum_{k=1}^{\infty}
x_1^k x_2^k
=
\frac{1-u+(u-t)x_1 x_2}{1-x_1 x_2}. 
\end{align*}
Turning to property {\bf 2}, it suffices to show that $\prod_{1 \leq i<j \leq 2n} (1- x_i x_j) \mathcal{G}_{2n}$ is a degree $2n-1$ polynomial in $x_{2n}$ at $n+1$ different values of $u$ (since $\mathcal{G}_{2n}$ is a degree $n$ polynomial in $u$). The $n+1$ points that we choose are $u=0$ (for which the claim is trivial, since in that case we obtain the left hand side of the Littlewood identity \eref{littlewood-hl}) and $u=t^{2k-2n}$ for 
$1 \leq k \leq n$. For these latter values of $u$ we find that
\begin{align*}
\mathcal{G}_{2n}(x_1,\dots,x_{2n};t^{2k-2n})
=
\sum_{i=k}^{n}
\prod_{j=i+1}^{n}
(1-t^{2k-2j})
\sum_{\substack{\lambda:\ell(\lambda) = 2i \\ \lambda'\ {\rm even}}}
b_{\lambda}^{\rm el}(t)
P_{\lambda}(x_1,\dots,x_{2n};t),
\end{align*}
so it is sufficient to show that for all $1 \leq i \leq n$,
\begin{align*}
\prod_{1 \leq i<j \leq 2n} (1 - x_i x_j)
\sum_{\substack{\lambda:\ell(\lambda) = 2i \\ \lambda'\ {\rm even}}}
b_{\lambda}^{\rm el}(t)
P_{\lambda}(x_1,\dots,x_{2n};t)
\end{align*} 
is a degree $2n-1$ polynomial in $x_{2n}$. We treat this as a proposition and denote it by 
$\mathcal{P}_{=2i}$. Similarly, we let $\mathcal{P}_{\leq 2i}$ denote this proposition in the case where the sum is taken over partitions $\lambda$ satisfying 
$\ell(\lambda) \leq 2i$. As we showed in Section \ref{2-poly}, 
\begin{align*}
\mathcal{P}_{\leq 2n}\ {\rm true} \implies \mathcal{P}_{= 2n}\ {\rm true}. 
\end{align*}
In fact the arguments presented therein can be repeated (almost verbatim) to deduce that
\begin{align*}
\mathcal{P}_{\leq 2i}\ {\rm true} \implies \mathcal{P}_{= 2i}\ {\rm true} 
\end{align*}
for general $i$. Moreover if $\mathcal{P}_{\leq 2i}$ and $\mathcal{P}_{= 2i}$ are true, then
\begin{multline*}
\prod_{1 \leq i<j \leq 2n} (1 - x_i x_j)
\left(
\sum_{\substack{\lambda:\ell(\lambda) \leq 2i \\ \lambda'\ {\rm even}}}
-
\sum_{\substack{\lambda:\ell(\lambda) = 2i \\ \lambda'\ {\rm even}}}
\right)
b_{\lambda}^{\rm el}(t)
P_{\lambda}(x_1,\dots,x_{2n};t)
\\
=
\prod_{1 \leq i<j \leq 2n} (1 - x_i x_j)
\sum_{\substack{\lambda:\ell(\lambda) \leq 2i-2 \\ \lambda'\ {\rm even}}}
b_{\lambda}^{\rm el}(t)
P_{\lambda}(x_1,\dots,x_{2n};t)
\end{multline*}
is a degree $2n-1$ polynomial in $x_{2n}$, proving $\mathcal{P}_{\leq 2i-2}$ is true. Hence we are able to iterate this string of implications to deduce that 
$\mathcal{P}_{= 2i}$ holds for all $1\leq i \leq n$.

For the recursive property {\bf 3}, one repeats the procedure outlined in Section 
\ref{2-rec}, but with obvious modifications to the formulae to cater for the more general coefficients appearing in the sum \eref{OSASM-cor2}. The strategy is to expand the left hand side of \eref{OSASM-cor2} (evaluated at $x_{2n} = 1/(t x_{2n-1})$) using two applications of the branching rule, and to isolate the coefficients $\mathcal{D}(\nu)$ of 
$P_{\nu}(x_1,\dots,x_{2n-2};t)$ which arise from this expansion. In the case where $\nu$ has only even columns, one finds that 
\begin{align*}
\mathcal{D}(\nu)
=
\sum_{j=1}^{\infty}
\sum_{\delta_j \in \{0,1\}}
t^{-|\delta|}
\prod_{k\ {\rm even}}^{m_0(\lambda)}
(1-u t^{k-2})
b_{\lambda}^{\rm el}(t)
\prod_{\substack{\delta_k = 0 \\ \delta_{k+1} = 1}}
(1-t^{m_k(\nu)-1})
(1-t^{m_k(\nu)}),
\end{align*}
where $\lambda$ is also a partition with even columns, given by $\lambda' = \nu' + 2 \delta$, and by definition $m_0(\lambda) = 2n - \ell(\lambda)$. Calculating $\mathcal{D}(\nu)$ can be done recursively, via the partial coefficients
\begin{align*}
\mathcal{D}_{i,\delta_i}(\nu)
=
\sum_{j=1}^{i-1}
\sum_{\delta_j \in \{0,1\}}
t^{-\sum_{k=1}^{i} \delta_k}
\prod_{k\ {\rm even}}^{m_0(\lambda)}
(1-u t^{k-2})
\prod_{k=1}^{i-1}
\prod_{l\ {\rm even}}^{m_k(\lambda)}
(1-t^{l-1})
\prod_{\substack{1 \leq k \leq i-1 \\ \delta_k = 0 \\ \delta_{k+1} = 1}}
(1-t^{m_k(\nu)-1})
(1-t^{m_k(\nu)}) 
\end{align*}
which differ from those in equation \eref{D-partial-coeff} only by the additional factor 
$\prod_{k\ {\rm even}}^{m_0(\lambda)} (1-ut^{k-2})$. One finds that these coefficients satisfy the recurrences \eref{recur-3} and \eref{recur-4} without any alteration, but with the new initial condition
\begin{align*} 
\mathcal{D}_{1,0}(\nu) = \prod_{k\ {\rm even}}^{2n-\ell(\nu)} (1-ut^{k-2}),
\quad\quad
t \mathcal{D}_{1,1}(\nu) = \prod_{k\ {\rm even}}^{2n-2-\ell(\nu)} (1-ut^{k-2}).
\end{align*} 
The remaining steps in Section \ref{2-rec} then go through in the same way. In particular, it is still true that $\mathcal{D}(\nu) = \mathcal{D}_{\nu_1+1,0}(\nu) - t \mathcal{D}_{\nu_1+1,1}(\nu)$, and the recurrence \eref{recur-final2} remains valid. The initial condition is now 
$\mathcal{D}_{1,0}(\nu) - t \mathcal{D}_{1,1}(\nu) = 
-u t^{m_0(\nu)} \prod_{k\ {\rm even}}^{m_0(\nu)} (1-u t^{k-2})$, where $m_0(\nu) = 2n-2-\ell(\nu)$, so in solving \eref{recur-final2} one obtains
\begin{align*}
\mathcal{D}_{\nu_1+1,0}(\nu)
-
t
\mathcal{D}_{\nu_1+1,1}(\nu)
&=
-u t^{\sum_{i=0}^{\infty} m_i(\nu)}
\prod_{j\ {\rm even}}^{m_0(\nu)}
(1-ut^{j-2})
\prod_{i=1}^{\infty}
\prod_{j\ {\rm even}}^{m_i(\nu)}
(1-t^{j-1}),
\\
&=
-u t^{2n-2}
\prod_{j\ {\rm even}}^{m_0(\nu)}
(1-ut^{j-2})
b_{\nu}^{\rm el}(t),
\end{align*}
which is the required result. In the case where $\nu$ has a column of odd length, the arguments in Section \ref{2-rec} apply (without any change) to prove that $\mathcal{D}(\nu) = 0$. These two evaluations of $\mathcal{D}(\nu)$ complete the proof of property {\bf 3}.

\end{proof}

Theorems \ref{u-def-little-s} and \ref{u-def-little-hl} are important results, since they serve as checks of Conjecture \ref{mac-conj} at the particular values $q=t$ and $q=0$, respectively. Further specialization (of the parameter $u$) leads to various known results. For example in the case of \eref{OSASM-cor2}, setting $u=0$ yields the Littlewood identity for Hall--Littlewood polynomials, whereas setting $u=t$ gives rise to equation \eref{HL-OSASM}. In complete analogy with the previous section, when we set $u=-\sqrt{t}$ we recover one half of the factors in equation \eref{ZOO} for $Z_{\rm OO}$. The remaining factors in \eref{ZOO} are precisely those of the OSASM partition function, on the right hand side of \eref{HL-OSASM}.

\section{$u$-deformed $BC$-type Cauchy identity and double U-turn alternating sign matrices}
\label{sec:UUASM}

In this section we conclude our study of the relationship between refined Cauchy/Littlewood identities and partition functions of the six-vertex model. We present one final example, conjecturing a $u$-deformed version of equation \eref{HL-UASM} and showing that its right hand side contains half of the factors present in the partition function with U-turn boundaries on two sides of the lattice. The remaining factors are those of the UASM partition function, as given by the right hand side of \eref{HL-UASM}. The explicit formula for this partition function, as a product of two determinants, is again due to Kuperberg \cite{kup2}.    

\subsection{Redefinition of Boltzmann weights for bulk vertices}

Throughout this section, it turns out to be most convenient to adopt a more symmetric form for the Boltzmann weights:
\begin{align}
\label{boltz-sym}
a_{\pm}(x,y)
=
\frac{1-t x/y}{1-x/y},
\quad\quad
b_{\pm}(x,y)
=
\sqrt{t},
\quad\quad
c_{\pm}(x,y)
=
\frac{(1-t)\sqrt{x/y}}{1-x/y}.
\end{align}
The only difference between this choice and the previous one \eref{boltz} is that the $c_{\pm}$ vertices are now equal. The Yang--Baxter equation remains satisfied, since the two sets of Boltzmann weights \eref{boltz} and \eref{boltz-sym} are related by a diagonal conjugation of the corresponding $R$-matrix.

\subsection{U-turn vertices, reflection and fish equations}

We introduce a new set of boundary vertices, the {\it U-turn vertices}, as shown in Figure \ref{fig:U}. The 
U-turn Boltzmann weights (denoted $r_{\pm}$ and $t_{\pm}$, since they are situated on the {\it right} and {\it top} edges of the partition function that we subsequently study) are given explicitly by
\begin{align}
\label{U-boltz}
r_{+}(x)
=
\sqrt{p x} - \frac{1}{\sqrt{p x}},
\quad\quad
r_{-}(x)
=
\frac{\sqrt{p}}{\sqrt{x}} - \frac{\sqrt{x}}{\sqrt{p}},
\quad\quad
t_{+}(y)
=
\frac{\sqrt{p t}}{\sqrt{y}}
+
\frac{\sqrt{y}}{\sqrt{p t}},
\quad\quad
t_{-}(y)
=
-
\sqrt{p y}
-
\frac{1}{\sqrt{p y}}.
\end{align}

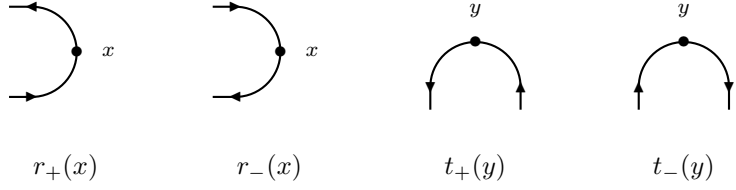
\begin{figure}
\begin{tabular}{cccc}
\begin{tikzpicture}[scale=0.6]
\draw[thick, smooth] ({2*sqrt(2) + 3 + 4},2) arc (90:-90:1);
\draw[thick, smooth] ({2*sqrt(2) + 3 + 3.5},2)--({2*sqrt(2) + 3 + 4},2);
\draw[thick, smooth] ({2*sqrt(2) + 3 + 3.5},0)--({2*sqrt(2) + 3 + 4},0);
\node[label={right: \fs $x$}] at ({2*sqrt(2) + 3 + 5},1) {$\bullet$};
\node at ({2*sqrt(2) + 3 + 4},0) {\r};
\node at ({2*sqrt(2) + 3 + 4},2) {\l};
\end{tikzpicture}
\quad\quad\quad&
\begin{tikzpicture}[scale=0.6]
\draw[thick, smooth] ({2*sqrt(2) + 3 + 4},2) arc (90:-90:1);
\draw[thick, smooth] ({2*sqrt(2) + 3 + 3.5},2)--({2*sqrt(2) + 3 + 4},2);
\draw[thick, smooth] ({2*sqrt(2) + 3 + 3.5},0)--({2*sqrt(2) + 3 + 4},0);
\node[label={right: \fs $x$}] at ({2*sqrt(2) + 3 + 5},1) {$\bullet$};
\node at ({2*sqrt(2) + 3 + 4},0) {\l};
\node at ({2*sqrt(2) + 3 + 4},2) {\r};
\end{tikzpicture}
\quad\quad\quad&
\begin{tikzpicture}[scale=0.6]
\draw[thick, smooth] (6,{0.5+2*sqrt(2)}) arc (180:0:1);
\draw[thick, smooth] (6,{2*sqrt(2)})--(6,{0.5+2*sqrt(2)});
\draw[thick, smooth] (8,{2*sqrt(2)})--(8,{0.5+2*sqrt(2)});
\node[label={above: \fs $y$}] at (7,{1.5+2*sqrt(2)}) {$\bullet$};
\node at (6,{2*sqrt(2)+0.5}) {\d};
\node at (8,{2*sqrt(2)+0.5}) {\u};
\end{tikzpicture}
\quad\quad\quad&
\begin{tikzpicture}[scale=0.6]
\draw[thick, smooth] (6,{0.5+2*sqrt(2)}) arc (180:0:1);
\draw[thick, smooth] (6,{2*sqrt(2)})--(6,{0.5+2*sqrt(2)});
\draw[thick, smooth] (8,{2*sqrt(2)})--(8,{0.5+2*sqrt(2)});
\node[label={above: \fs $y$}] at (7,{1.5+2*sqrt(2)}) {$\bullet$};
\node at (6,{2*sqrt(2)+0.5}) {\u};
\node at (8,{2*sqrt(2)+0.5}) {\d};
\end{tikzpicture}
\\ \\
$r_{+}(x)$
\quad\quad\quad&
$r_{-}(x)$
\quad\quad\quad&
$t_{+}(y)$
\quad\quad\quad&
$t_{-}(y)$
\end{tabular}
\caption{The U-turn vertices, with their Boltzmann weights indicated underneath. The Boltzmann weights are functions of the single rapidity variable passing through the U-turn, and of a further parameter $p$ which is arbitrary.}
\label{fig:U}
\end{figure} 
\noindent
Together with the ordinary Boltzmann weights \eref{boltz-sym}, the U-turn weights satisfy the Sklyanin reflection equation \cite{skl}. Since we have both $r$ and $t$-type boundary vertices, two types of reflection equation occur in this work. These are illustrated in Figure \ref{fig:U-YB}. The reflection equations allow us to establish the symmetry of the partition function $Z_{\rm UU}$ in its rapidity variables.


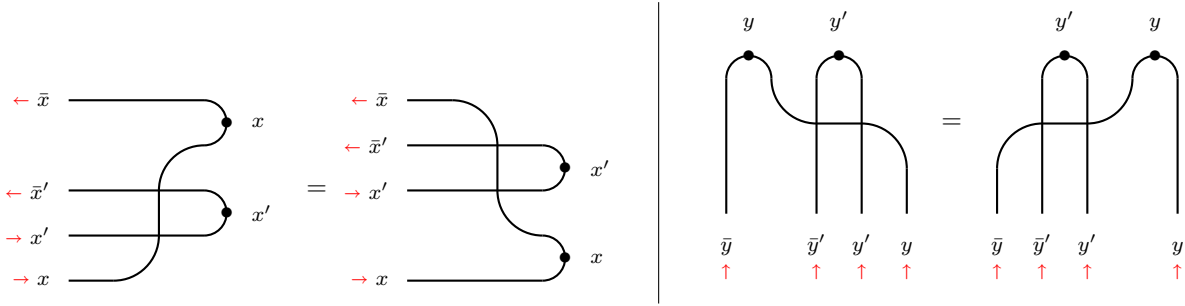
\begin{figure}[H]
\begin{tabular}{cc|cc}
\begin{tikzpicture}[scale=0.30]
\draw[thick, smooth] (0,0)--(2,0);
\draw[thick, smooth] (0,2)--(6,2);
\draw[thick, smooth] (0,4)--(6,4);
\draw[thick, smooth] (0,8)--(6,8);
\draw[thick, smooth] (6,2) arc (-90:90:1); 
\draw[thick, smooth] (6,6) arc (-90:90:1);
\draw[thick, smooth] (2,0) arc (-90:0:2);
\draw[thick, smooth] (4,4) arc (-180:-270:2);
\draw[thick, smooth] (4,4)--(4,2);
\node[label={left: \fs ${\color{red} \shortrightarrow} \ x$}] at (0,0) {};
\node[label={left: \fs ${\color{red} \shortrightarrow} \ x'$}] at (0,2) {};
\node[label={left: \fs ${\color{red} \shortleftarrow} \ \b{x}'$}] at (0,4) {};
\node[label={left: \fs ${\color{red} \shortleftarrow} \ \b{x}$}] at (0,8) {};
\node[label={right: \fs $x'$}] at (7,3) {$\bullet$};
\node[label={right: \fs $x$}] at (7,7) {$\bullet$};

\node at (11, 4) {$=$}; 

\draw[thick, smooth] (15,0)--(21,0);
\draw[thick, smooth] (15,4)--(21,4);
\draw[thick, smooth] (15,6)--(21,6);
\draw[thick, smooth] (15,8)--(17,8);
\draw[thick, smooth] (21,4) arc (-90:90:1); 
\draw[thick, smooth] (21,0) arc (-90:90:1); 
\draw[thick, smooth] (17,8) arc (90:0:2);
\draw[thick, smooth] (21,2) arc (270:180:2);
\draw[thick, smooth] (19,6)--(19,4);
\node[label={left: \fs ${\color{red} \shortrightarrow} \ x$}] at (15,0) {};
\node[label={left: \fs ${\color{red} \shortrightarrow} \ x'$}] at (15,4) {};
\node[label={left: \fs ${\color{red} \shortleftarrow} \ \b{x}'$}] at (15,6) {};
\node[label={left: \fs ${\color{red} \shortleftarrow} \ \b{x}$}] at (15,8) {};
\node[label={right: \fs $x$}] at (22,1) {$\bullet$};
\node[label={right: \fs $x'$}] at (22,5) {$\bullet$};
\end{tikzpicture}
&
&
&
\begin{tikzpicture}[scale=0.30]
\draw[thick, smooth] (0,0)--(0,6);
\draw[thick, smooth] (4,0)--(4,6);
\draw[thick, smooth] (6,0)--(6,6);
\draw[thick, smooth] (8,0)--(8,2);
\draw[thick, smooth] (2,6) arc (0:180:1); 
\draw[thick, smooth] (6,6) arc (0:180:1);
\draw[thick, smooth] (6,4) arc (90:0:2);
\draw[thick, smooth] (4,4) arc (270:180:2);
\draw[thick, smooth] (16,4) arc (-90:0:2);
\draw[thick, smooth] (6,4)--(4,4);
\node[label={below: \fs ${\color{red} \begin{array}{c} {\color{black} \b{y}} \\ \shortuparrow \end{array} }$}] at (0,0) {};
\node[label={below: \fs ${\color{red} \begin{array}{c} {\color{black} \b{y}'} \\ \shortuparrow \end{array} }$}] at (4,0) {};
\node[label={below: \fs ${\color{red} \begin{array}{c} {\color{black} y'} \\ \shortuparrow \end{array} }$}] at (6,0) {};
\node[label={below: \fs ${\color{red} \begin{array}{c} {\color{black} y} \\ \shortuparrow \end{array} }$}] at (8,0) {};
\node[label={above: \fs $y$}] at (1,7) {$\bullet$};
\node[label={above: \fs $y'$}] at (5,7) {$\bullet$};

\node at (10, 4) {$=$}; 

\draw[thick, smooth] (12,0)--(12,2);
\draw[thick, smooth] (14,0)--(14,6);
\draw[thick, smooth] (16,0)--(16,6);
\draw[thick, smooth] (20,0)--(20,6);
\draw[thick, smooth] (16,6) arc (0:180:1); 
\draw[thick, smooth] (20,6) arc (0:180:1);
\draw[thick, smooth] (12,2) arc (-180:-270:2);
\draw[thick, smooth] (16,4)--(14,4);
\node[label={below: \fs ${\color{red} \begin{array}{c} {\color{black} \b{y}} \\ \shortuparrow \end{array} }$}] at (12,0) {};
\node[label={below: \fs ${\color{red} \begin{array}{c} {\color{black} \b{y}'} \\ \shortuparrow \end{array} }$}] at (14,0) {};
\node[label={below: \fs ${\color{red} \begin{array}{c} {\color{black} y'} \\ \shortuparrow \end{array} }$}] at (16,0) {};
\node[label={below: \fs ${\color{red} \begin{array}{c} {\color{black} y} \\ \shortuparrow \end{array} }$}] at (20,0) {};
\node[label={above: \fs $y'$}] at (15,7) {$\bullet$};
\node[label={above: \fs $y$}] at (19,7) {$\bullet$};
\end{tikzpicture}
\end{tabular}
\caption{The Sklyanin reflection equations. In both equations, two U-turn vertices and two bulk vertices are involved. External edges are fixed to definite arrow configurations on both sides of the equation, while internal edges are summed over. We remark that the orientation of the bulk vertices is not the same in both equations.}
\label{fig:U-YB}
\end{figure}

We will make use of two further relations satisfied by the bulk and U-turn vertices. Following \cite{kup2} we refer to these as {\it fish equations}, and they are given in Figure \ref{fig:fish}.


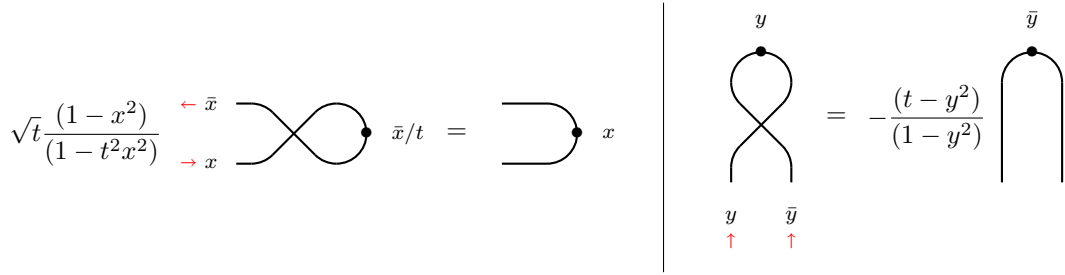
\begin{figure}
\begin{tabular}{cc|cc}
\begin{tikzpicture}[scale=0.4]
\def\dy{4}
\node at (-5.5, 1+\dy) {$\displaystyle{\sqrt{t} \frac{(1-x^2)}{(1-t^2 x^2)}}$};
\draw[thick,smooth] (-0.5,0+\dy)--(0,0+\dy);
\draw[thick,smooth] (-0.5,2+\dy)--(0,2+\dy);
\draw[thick, smooth] (0,2+\dy) arc (90:45:1);
\draw[thick, smooth] (0,0+\dy) arc (-90:-45:1);
\draw[thick, smooth] ({sqrt(2)/2},{1+sqrt(2)/2+\dy})--({3*sqrt(2)/2},{1-sqrt(2)/2+\dy});
\draw[thick, smooth] ({sqrt(2)/2},{1-sqrt(2)/2+\dy})--({3*sqrt(2)/2},{1+sqrt(2)/2+\dy});
\draw[thick, smooth] ({2*sqrt(2)},2+\dy) arc (90:135:1);
\draw[thick, smooth] ({2*sqrt(2)},0+\dy) arc (-90:-135:1);
\draw[thick, smooth] ({2*sqrt(2)},2+\dy) arc (90:-90:1);
\node[label={right: \fs $\b{x}/t$}] at ({2*sqrt(2)+1},1+\dy) {$\bullet$};
\node[label={left: \fs ${\color{red} \shortrightarrow} \ x$}] at (-0.5, 0+\dy) {};
\node[label={left: \fs ${\color{red} \shortleftarrow} \ \b{x}$}] at (-0.5, 2+\dy) {};

\node at ({2*sqrt(2) + 4}, 1+\dy) {$=$};
\node at (7,1) {${}$};

\draw[thick, smooth] ({2*sqrt(2) + 3 + 4},2+\dy) arc (90:-90:1);
\draw[thick, smooth] ({2*sqrt(2) + 3 + 2.5},2+\dy)--({2*sqrt(2) + 3 + 4},2+\dy);
\draw[thick, smooth] ({2*sqrt(2) + 3 + 2.5},0+\dy)--({2*sqrt(2) + 3 + 4},0+\dy);
\node[label={right: \fs $x$}] at ({2*sqrt(2) + 3 + 5},1+\dy) {$\bullet$};

\end{tikzpicture}
&
&
&
\begin{tikzpicture}[scale=0.4]
\draw[thick, smooth] (0,{-0.5 + 2*sqrt(2) + 1}) arc (-180:-135:1);
\draw[thick, smooth] (2,{-0.5 + 2*sqrt(2) + 1}) arc (0:-45:1);
\draw[thick, smooth] ({1 - sqrt(2)/2},{-0.5-sqrt(2)/2 + 2*sqrt(2) + 1})--({1 + sqrt(2)/2},{-0.5-3*sqrt(2)/2 + 2*sqrt(2) + 1});
\draw[thick, smooth] ({1 - sqrt(2)/2},{-0.5-3*sqrt(2)/2 + 2*sqrt(2) + 1})--({1 + sqrt(2)/2},{-0.5-sqrt(2)/2 + 2*sqrt(2) + 1});
\draw[thick, smooth] (0,0.5) arc (180:135:1);
\draw[thick, smooth] (2,0.5) arc (0:45:1);
\draw[thick, smooth] (0,{0.5+2*sqrt(2)}) arc (180:0:1);
\draw[thick, smooth]  (0,0)-- (0,0.5);
\draw[thick, smooth]  (2,0)-- (2,0.5);
\node[label={above: \fs $y$}] at (1,{1.5+2*sqrt(2)}) {$\bullet$};
\node[label={below: \fs ${\color{red} 
\begin{array}{c} {\color{black} y} \\ \shortuparrow \end{array} }$}] at (0,0) {};
\node[label={below: \fs ${\color{red} 
\begin{array}{c} {\color{black} \b{y}} \\ \shortuparrow \end{array} }$}] at (2,0) {};

\node at (3.5,{0.75 + sqrt(2)}) {$=$};

\node at (6.5,{0.75 + sqrt(2)}) {$-\displaystyle{\frac{(t-y^2)}{(1-y^2)}}$};
\draw[thick, smooth] (9,{0.5+2*sqrt(2)}) arc (180:0:1);
\draw[thick, smooth] (9,0)--(9,{0.5+2*sqrt(2)});
\draw[thick, smooth] (11,0)--(11,{0.5+2*sqrt(2)});
\node[label={above: \fs $\b{y}$}] at (10,{1.5+2*sqrt(2)}) {$\bullet$};
\end{tikzpicture}
\end{tabular}
\caption{The fish equations, which involve a single U-turn and bulk vertex. The factors indicate that one side of the equation is to be multiplied by that quantity. We point out that the argument of the U-turn vertex changes from one side of the equation to the other.}
\label{fig:fish}
\end{figure} 

\subsection{Partition function on double U-turn lattice}

Following \cite{kup2}, we now consider the partition function of the six-vertex model on a lattice with two reflecting boundaries, as shown in Figure \ref{fig:UU}. The corresponding alternating sign matrices are the so-called UUASMs \cite{kup2}.


\begin{figure}
\begin{tikzpicture}[scale=0.6]
\foreach\x in {1,...,6}
\draw[thick]
(0,\x) -- (6.5,\x);
\foreach\x in {1,...,3}
\node[label={left: \fs ${\color{red} \shortrightarrow} \ x_{\x}$}] at (0,7-2*\x) {};

\foreach\x in {1,...,3}
\node[label={left: \fs ${\color{red} \shortleftarrow} \ \b{x}_{\x}$}] at (0,8-2*\x) {};


\foreach\y in {1,...,6}
\node at (0.5,\y) {\r};

\foreach\x in {1,...,6}
\draw[thick]
(\x,0) -- (\x,6.5);
\foreach\x in {1,...,3}
\node[label={below: \fs ${\color{red} \begin{array}{c} {\color{black} y_{\x}} \\ \shortuparrow \end{array} }$}] at (8-2*\x,0) {};
\foreach\x in {1,...,3}
\node[label={below: \fs ${\color{red} \begin{array}{c} {\color{black} \b{y}_{\x}} \\ \shortuparrow \end{array} }$}] at (7-2*\x,0) {};


\foreach\x in {1,...,6}
\node at (\x,0.5) {\d};

\foreach\x in {1,...,3}
\draw[thick, smooth] (6.5,2*\x-1) arc (-90:90:0.5);
\foreach\x in {1,...,3}
\node[label={right: \fs $x_{\x}$}] at (7,8-2*\x-1/2) {$\bullet$};

\foreach\x in {1,...,3}
\draw[thick, smooth] (2*\x-1,6.5) arc (180:0:0.5);
\foreach\x in {1,...,3}
\node[label={above: \fs $y_{\x}$}] at (8-2*\x-1/2,7) {$\bullet$};

\end{tikzpicture}
\caption{The partition function $Z_{\rm UU}$ of the six-vertex model with 
doubly reflecting domain wall boundary conditions, in the case $n=3$. The orientations of the horizontal lines alternate between left-to-right and right-to-left, whereas the orientation of every vertical line is bottom-to-top.}
\label{fig:UU}
\end{figure}
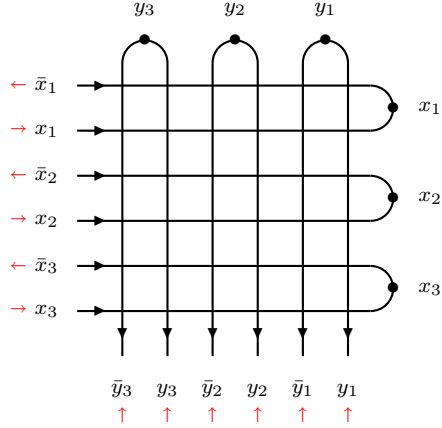

\begin{lem}
\label{lem:UU-prop}
The partition function 
$Z^{(n)}_{\rm UU} = Z_{\rm UU}(x_1,\dots,x_n;y_1,\dots,y_n;t)$ as defined in Figure \ref{fig:UU} satisfies six properties:
\begin{enumerate}[label=\bf\arabic*.]

\item Multiplying by $\prod_{i,j=1}^{n} (1-x_i y_j)^2 (1-x_i \b{y}_j)^2$, it is a polynomial in $x_n$ of degree $4n$.

\item It has zeros at $x_n = \pm 1$.

\item It is symmetric in $\{y_1,\dots,y_n\}$.

\item It is quasi-symmetric under $y_n \longleftrightarrow \b{y}_n$:
\begin{align}
\label{y-quasi-sym}
\b{y}_n
(1-t y_n^2)
Z_{\rm UU}
(x_1,\dots,x_n; y_1,\dots,y_{n-1},y_n;t)
=
y_n
(1-t\b{y}_n^2)
Z_{\rm UU}
(x_1,\dots,x_n; y_1,\dots,y_{n-1},\b{y}_n;t).
\end{align}
 
\item It obeys the recursion relations
\begin{align}
\label{UU-rec1}
\lim_{x_n \rightarrow \b{y}_n}
\Big(
(1-x_n y_n)^2
Z^{(n)}_{\rm UU}
\Big)
&=
\\
(1-t)^2  t^{n-1/2}
&
(\b{p} \b{y}_n-p y_n)
\frac{(1-t\b{y}_n^2)}{(1-\b{y}_n^2)}
\prod_{i=1}^{n-1}
\left[
\frac{(1-t x_i y_n)}{(1-x_i y_n)}
\frac{(1-ty_i \b{y}_n)}{(1-y_i \b{y}_n)}
\frac{(1-t\b{y}_i \b{y}_n)}{(1-\b{y}_i \b{y}_n)}
\right]^2  
Z^{(n-1)}_{\rm UU},
\nonumber
\\
\label{UU-rec2}
Z^{(n)}_{\rm UU}
\Big|_{x_n = y_n/t}
&= 
t^{3n-1/2}
(p y_n - \b{p} \b{y}_n)
\frac{(1-y_n^2/t^2)}{(1-y_n^2/t)}
\prod_{i=1}^{n-1}
\frac{(1-tx_i y_n)^2}{(1-x_i y_n)^2}
Z^{(n-1)}_{\rm UU}.
\end{align}

\item When $n=1$, it is given explicitly by
\begin{align*}
Z^{(1)}_{\rm UU}
=
\frac{
\sqrt{t}(1-t)(1-x_1^2)(y_1-t\b{y}_1) 
\left[
(pt+\b{p}) (x_1 y_1 + x_1 \b{y}_1)
-
(p+\b{p}) (1+tx_1^2)
\right]
}
{(1-x_1 y_1)^2 (1-x_1 \b{y}_1)^2}.
\end{align*}

\end{enumerate}
\end{lem}

\begin{proof}
The proof revolves around the lattice definition of $Z^{(n)}_{\rm UU}$, as well as using the Yang--Baxter and fish equations to prove various symmetries.

\begin{enumerate}[label=\bf\arabic*.]

\item Multiplying the entire partition function by 
$\prod_{i,j=1}^{n} (1-x_i y_j)^2 (1-x_i \b{y}_j)^2$ is equivalent to renormalizing the individual Boltzmann weights, such that each is a degree-1 polynomial in $x_i$ (except the $c_{\pm}$ weights, which go as $\sqrt{x_i}$). We focus our attention on the bottom two rows of $Z^{(n)}_{\rm UU}$, which is the sole place having $x_n$ dependence. In every legal configuration there must be an odd total number of $c_{\pm}$ vertices in these final two rows. Combining this with the explicit parametrization of the right U-turn vertices ensures that $Z^{(n)}_{\rm UU}$ is indeed a polynomial in $x_n$. Furthermore since there is minimally one 
$c_{\pm}$ vertex in these two rows, the degree of the polynomial is $4n$.

\item Starting from the U-turn vertex associated with the final two rows of 
$Z^{(n)}_{\rm UU}$, we can immediately use the fish equation on the left of Figure \ref{fig:fish} to introduce an extra vertex, internal to the lattice. This also produces an overall multiplicative factor of $\sqrt{t}(1-x_n^2)/(1-t^2 x_n^2)$. Using the Yang--Baxter equation repeatedly, it is possible to slide the new vertex horizontally through the lattice until it ultimately emerges from the left as a $b_{+}$ vertex, with Boltzmann weight $\sqrt{t}$. This process is illustrated in Figure \ref{fig:fish-exchange}. The two zeros at $x_n = \pm 1$ are due to the factor $(1-x_n^2)$ introduced at the start of this procedure.

\begin{figure}
\hspace*{3em}
\begin{tikzpicture}[scale=0.6]


\foreach\x in {1,...,6} \draw[thick, smooth] (0,\x) -- ({6.5},\x);
\foreach\x in {3,...,6} \draw[thick, smooth] (6.5,\x) -- ({6.5+sqrt(2)},\x);

\foreach\x in {1,...,3} \node[label={left: \fs ${\color{red} \shortrightarrow} \ x_{\x}$}] at (0,7-2*\x) {};
\foreach\x in {1,...,3} \node[label={left: \fs ${\color{red} \shortleftarrow} \ \b{x}_{\x}$}] at (0,8-2*\x) {};

\foreach\y in {1,...,6}
\node at (0.5,\y) {\r};

\foreach\x in {1,...,6}
\draw[thick]
(\x,0) -- (\x,6.5);
\foreach\x in {1,...,3}
\node[label={below: \fs ${\color{red} \begin{array}{c} {\color{black} y_{\x}} \\ \shortuparrow \end{array} }$}] at (8-2*\x,0) {};
\foreach\x in {1,...,3}
\node[label={below: \fs ${\color{red} \begin{array}{c} {\color{black} \b{y}_{\x}} \\ \shortuparrow \end{array} }$}] at (7-2*\x,0) {};

\foreach\x in {1,...,6}
\node at (\x,0.5) {\d};

\foreach\x in {2,...,3} \draw[thick, smooth] ({6.5+sqrt(2)},2*\x-1) arc (-90:90:0.5);
\foreach\x in {1,...,2} \node[label={right: \fs $x_{\x}$}] at ({7+sqrt(2)},8-2*\x-1/2) {$\bullet$};

\draw[thick, smooth] (6.5,2) arc (90:45:0.5);
\draw[thick, smooth] (6.5,1) arc (-90:-45:0.5);
\draw[thick, smooth] ({6.5 + sqrt(2)/4},{1.5+sqrt(2)/4})--({6.5 + 3*sqrt(2)/4},{1.5-sqrt(2)/4});
\draw[thick, smooth] ({6.5 + sqrt(2)/4},{1.5-sqrt(2)/4})--({6.5 + 3*sqrt(2)/4},{1.5+sqrt(2)/4});
\draw[thick, smooth] ({6.5 + sqrt(2)},2) arc (90:135:0.5);
\draw[thick, smooth] ({6.5 + sqrt(2)},1) arc (-90:-135:0.5);
\draw[thick, smooth] ({6.5 + sqrt(2)},2) arc (90:-90:0.5);
\node[label={right: \fs $\b{x}_3/t$}] at ({6.5 + sqrt(2)+0.5},1.5) {$\bullet$};

\foreach\x in {1,...,3}
\draw[thick, smooth] (2*\x-1,6.5) arc (180:0:0.5);
\foreach\x in {1,...,3}
\node[label={above: \fs $y_{\x}$}] at (8-2*\x-1/2,7) {$\bullet$};

\node at ({9+sqrt(2)},3.5) {$=$};


\foreach\x in {3,...,6} \draw[thick, smooth] ({11+sqrt(2)},\x) -- ({18.5+2*sqrt(2)},\x);
\foreach\x in {1,...,2} \draw[thick, smooth] ({12+2*sqrt(2)},\x) -- ({18.5+2*sqrt(2)},\x);
\foreach\x in {1,...,2} \draw[thick, smooth] ({11+sqrt(2)},\x) -- ({12+sqrt(2)},\x);

\foreach\x in {1,...,3} \node[label={left: \fs ${\color{red} \shortrightarrow} \ x_{\x}$}] at ({11+sqrt(2)},7-2*\x) {};
\foreach\x in {1,...,3} \node[label={left: \fs ${\color{red} \shortleftarrow} \ \b{x}_{\x}$}] at ({11+sqrt(2)},8-2*\x) {};

\draw[thick, smooth] ({12+sqrt(2)},2) arc (90:45:0.5);
\draw[thick, smooth] ({12+sqrt(2)},1) arc (-90:-45:0.5);
\draw[thick, smooth] ({12 + sqrt(2) + sqrt(2)/4},{1.5+sqrt(2)/4})--({12 + sqrt(2) + 3*sqrt(2)/4},{1.5-sqrt(2)/4});
\draw[thick, smooth] ({12 + sqrt(2) + sqrt(2)/4},{1.5-sqrt(2)/4})--({12 + sqrt(2) + 3*sqrt(2)/4},{1.5+sqrt(2)/4});
\draw[thick, smooth] ({12 + sqrt(2) + sqrt(2)},2) arc (90:135:0.5);
\draw[thick, smooth] ({12 + sqrt(2) + sqrt(2)},1) arc (-90:-135:0.5);

\foreach\y in {1,...,6} \node at ({11.5+sqrt(2)},\y) {\r};

\node at ({12.5+2*sqrt(2)},1) {\r};
\node at ({12.5+2*sqrt(2)},2) {\r};
%
\foreach\x in {1,...,6} \draw[thick, smooth] ({\x + 12 + 2*sqrt(2)},0) -- ({\x + 12 + 2*sqrt(2)},6.5);
\foreach\x in {1,...,3} \node[label={below: \fs ${\color{red} \begin{array}{c} {\color{black} y_{\x}} \\ \shortuparrow \end{array} }$}] at ({8-2*\x + 12 + 2*sqrt(2)},0) {};
\foreach\x in {1,...,3} \node[label={below: \fs ${\color{red} \begin{array}{c} {\color{black} \b{y}_{\x}} \\ \shortuparrow \end{array} }$}] at ({7-2*\x + 12 + 2*sqrt(2)},0) {};
%
\foreach\x in {1,...,6} \node at ({\x + 12 + 2*sqrt(2)},0.5) {\d};
%
\foreach\x in {1,...,3} \draw[thick, smooth] ({18.5+2*sqrt(2)},2*\x-1) arc (-90:90:0.5);
\foreach\x in {1,...,2} \node[label={right: \fs $x_{\x}$}] at ({19 + 2*sqrt(2)},8-2*\x-1/2) {$\bullet$};
\node[label={right: \fs $\b{x}_3/t$}] at ({19 + 2*sqrt(2)},2-1/2) {$\bullet$};

%
\foreach\x in {1,...,3} \draw[thick, smooth] ({2*\x-1 + 12 + 2*sqrt(2)},6.5) arc (180:0:0.5);
\foreach\x in {1,...,3} \node[label={above: \fs $y_{\x}$}] at ({8-2*\x-1/2 + 12 + 2*sqrt(2)},7) {$\bullet$};

\end{tikzpicture}
\caption{Using the fish equation to exchange the lowest two horizontal lines. The internal vertex thus introduced can be moved horizontally through the lattice, until it emerges from the left side, where it is forced to be in a $b_{+}$ configuration. The order of the two participating horizontal lines is switched, the right U-turn vertex now has argument $\b{x}_n/t$, and a total multiplicative factor of $t(1-x_n^2)/(1-t^2 x_n^2)$ is acquired.}
\label{fig:fish-exchange}
\end{figure}
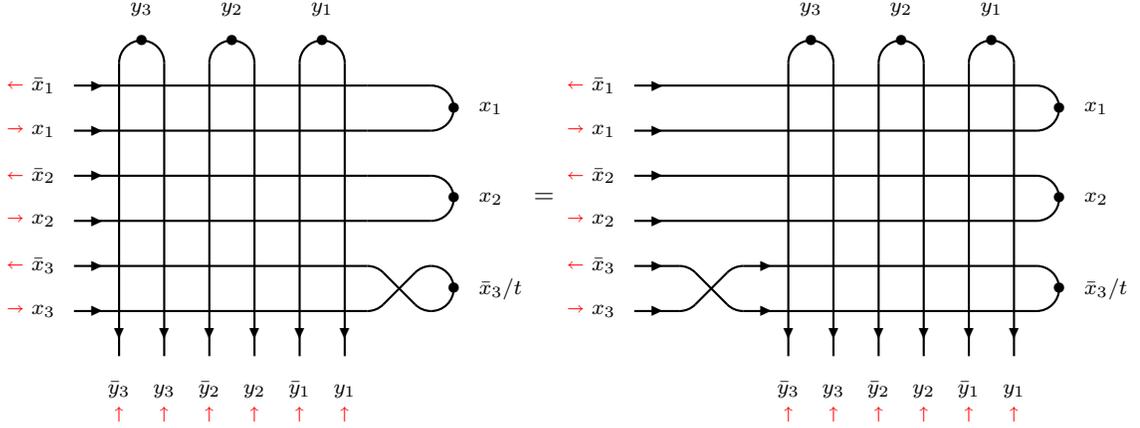

\item Using both the Yang--Baxter and reflection equation, it is possible to interchange the order of any two double columns bearing the rapidities $\{\b{y}_i,y_i\}$ and 
$\{\b{y}_j,y_j\}$. This is a standard argument used in models with a double-row transfer matrix; see \cite{kup2} and references therein for more details.

\item One can attach a single $a_{-}$ vertex at the base of the first two columns in $Z^{(n)}_{\rm UU}$, which is equivalent to multiplying the partition function by $(1-t y_n^2)/(1-y_n^2)$. The inserted vertex can then be moved vertically through the lattice (using the Yang--Baxter equation) until it reaches the U-turn vertex at the top of the double column. Applying the fish equation on the right of Figure \ref{fig:fish}, the internal crossing is removed and the order of the first two columns is interchanged, up to an overall factor of 
$-(t-y_n^2)/(1-y_n^2)$. Hence we find that
\begin{align*}
\frac{(1-t y_n^2)}{(1-y_n^2)}
Z_{\rm UU}
(x_1,\dots,x_n; y_1,\dots,y_{n-1},y_n;t)
=
-
\frac{(t-y_n^2)}{(1-y_n^2)}
Z_{\rm UU}
(x_1,\dots,x_n; y_1,\dots,y_{n-1},\b{y}_n;t).
\end{align*}
Trivial rearrangement of the prefactors gives the result \eref{y-quasi-sym}.

\item The recursion relation \eref{UU-rec1} follows from the original lattice representation of 
$Z^{(n)}_{\rm UU}$, in Figure \ref{fig:UU}. Multiplying the partition function by $(1-x_n y_n)^2$ and taking $x_n \rightarrow \b{y}_n$ forces the bottom left vertex of the lattice to be a $c_{+}$ vertex. This restriction causes a larger subset of the vertices in $Z^{(n)}_{\rm UU}$ to be in a frozen configuration, as shown on the left of Figure \ref{fig:UU-rec}.

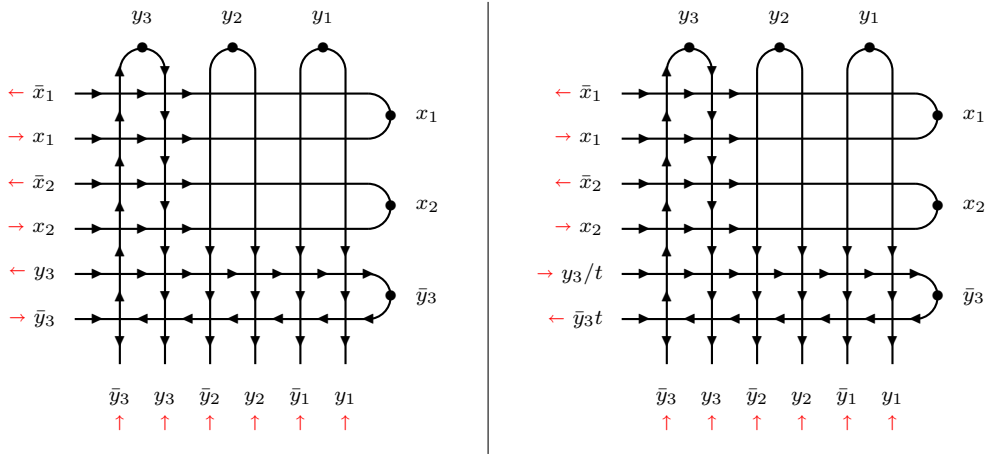
\begin{figure}
\begin{tabular}{cc|cc}
\begin{tikzpicture}[scale=0.6]
\foreach\x in {1,...,6}
\draw[thick]
(0,\x) -- (6.5,\x);
\foreach\x in {1,...,2}
\node[label={left: \fs ${\color{red} \shortrightarrow} \ x_{\x}$}] at (0,7-2*\x) {};
\foreach\x in {1,...,2}
\node[label={left: \fs ${\color{red} \shortleftarrow} \ \b{x}_{\x}$}] at (0,8-2*\x) {};
\node[label={left: \fs ${\color{red} \shortrightarrow} \ \b{y}_3$}] at (0,1) {};
\node[label={left: \fs ${\color{red} \shortleftarrow} \ y_3$}] at (0,2) {};
\foreach\y in {1,...,6}
\node at (0.5,\y) {\r};
\node at (1.5,1) {\l};
\foreach\y in {2,...,6}
\node at (1.5,\y) {\r};
\node at (2.5,1) {\l};
\foreach\y in {2,...,6}
\node at (2.5,\y) {\r};
\node at (3.5,1) {\l};
\node at (3.5,2) {\r};
\node at (4.5,1) {\l};
\node at (4.5,2) {\r};
\node at (5.5,1) {\l};
\node at (5.5,2) {\r};
\node at (6.5,1) {\l};
\node at (6.5,2) {\r};
\foreach\x in {1,...,6}
\draw[thick]
(\x,0) -- (\x,6.5);
\foreach\x in {1,...,3}
\node[label={below: \fs ${\color{red} \begin{array}{c} {\color{black} y_{\x}} \\ \shortuparrow \end{array} }$}] at (8-2*\x,0) {};
\foreach\x in {1,...,3}
\node[label={below: \fs ${\color{red} \begin{array}{c} {\color{black} \b{y}_{\x}} \\ \shortuparrow \end{array} }$}] at (7-2*\x,0) {};
\foreach\x in {1,...,6}
\node at (\x,0.5) {\d};
\node at (1,1.5) {\u};
\foreach\x in {2,...,6}
\node at (\x,1.5) {\d};
\node at (1,2.5) {\u};
\foreach\x in {2,...,6}
\node at (\x,2.5) {\d};
\node at (1,3.5) {\u};
\node at (2,3.5) {\d};
\node at (1,4.5) {\u};
\node at (2,4.5) {\d};
\node at (1,5.5) {\u};
\node at (2,5.5) {\d};
\node at (1,6.5) {\u};
\node at (2,6.5) {\d};
\foreach\x in {1,...,3}
\draw[thick, smooth] (6.5,2*\x-1) arc (-90:90:0.5);
\foreach\x in {1,...,2}
\node[label={right: \fs $x_{\x}$}] at (7,8-2*\x-1/2) {$\bullet$};
\node[label={right: \fs $\b{y}_3$}] at (7,2-1/2) {$\bullet$};
\foreach\x in {1,...,3}
\draw[thick, smooth] (2*\x-1,6.5) arc (180:0:0.5);
\foreach\x in {1,...,3}
\node[label={above: \fs $y_{\x}$}] at (8-2*\x-1/2,7) {$\bullet$};
\end{tikzpicture}
&
&
&
\begin{tikzpicture}[scale=0.6]
\foreach\x in {1,...,6}
\draw[thick]
(0,\x) -- (6.5,\x);
\foreach\x in {1,...,2}
\node[label={left: \fs ${\color{red} \shortrightarrow} \ x_{\x}$}] at (0,7-2*\x) {};
\foreach\x in {1,...,2}
\node[label={left: \fs ${\color{red} \shortleftarrow} \ \b{x}_{\x}$}] at (0,8-2*\x) {};
\node[label={left: \fs ${\color{red} \shortleftarrow} \ \b{y}_3 t$}] at (0,1) {};
\node[label={left: \fs ${\color{red} \shortrightarrow} \ y_3/t$}] at (0,2) {};
\foreach\y in {1,...,6}
\node at (0.5,\y) {\r};
\node at (1.5,1) {\l};
\foreach\y in {2,...,6}
\node at (1.5,\y) {\r};
\node at (2.5,1) {\l};
\foreach\y in {2,...,6}
\node at (2.5,\y) {\r};
\node at (3.5,1) {\l};
\node at (3.5,2) {\r};
\node at (4.5,1) {\l};
\node at (4.5,2) {\r};
\node at (5.5,1) {\l};
\node at (5.5,2) {\r};
\node at (6.5,1) {\l};
\node at (6.5,2) {\r};
\foreach\x in {1,...,6}
\draw[thick]
(\x,0) -- (\x,6.5);
\foreach\x in {1,...,3}
\node[label={below: \fs ${\color{red} \begin{array}{c} {\color{black} y_{\x}} \\ \shortuparrow \end{array} }$}] at (8-2*\x,0) {};
\foreach\x in {1,...,3}
\node[label={below: \fs ${\color{red} \begin{array}{c} {\color{black} \b{y}_{\x}} \\ \shortuparrow \end{array} }$}] at (7-2*\x,0) {};
\foreach\x in {1,...,6}
\node at (\x,0.5) {\d};
\node at (1,1.5) {\u};
\foreach\x in {2,...,6}
\node at (\x,1.5) {\d};
\node at (1,2.5) {\u};
\foreach\x in {2,...,6}
\node at (\x,2.5) {\d};
\node at (1,3.5) {\u};
\node at (2,3.5) {\d};
\node at (1,4.5) {\u};
\node at (2,4.5) {\d};
\node at (1,5.5) {\u};
\node at (2,5.5) {\d};
\node at (1,6.5) {\u};
\node at (2,6.5) {\d};
\foreach\x in {1,...,3}
\draw[thick, smooth] (6.5,2*\x-1) arc (-90:90:0.5);
\foreach\x in {1,...,2}
\node[label={right: \fs $x_{\x}$}] at (7,8-2*\x-1/2) {$\bullet$};
\node[label={right: \fs $\b{y}_3$}] at (7,2-1/2) {$\bullet$};
\foreach\x in {1,...,3}
\draw[thick, smooth] (2*\x-1,6.5) arc (180:0:0.5);
\foreach\x in {1,...,3}
\node[label={above: \fs $y_{\x}$}] at (8-2*\x-1/2,7) {$\bullet$};
\end{tikzpicture} 
\end{tabular}
\caption{Graphical representation of the two recursion relations satisfied by $Z_{\rm UU}$. On the left, the effect of setting $x_n = \b{y}_n$ in the original partition function, giving rise to equation \eref{UU-rec1}. On the right, the effect of setting $x_n = y_n/t$ in the partition function (after using the horizontal fish equation), giving rise to equation \eref{UU-rec2}.}
\label{fig:UU-rec}
\end{figure}

The second recursion relation \eref{UU-rec2} can be deduced from the lattice representation on the right hand side of Figure \ref{fig:fish-exchange}, obtained by a single application of the horizontal fish equation and repeated use of the Yang--Baxter equation. One starts by removing the frozen $b_{+}$ vertex from the left side of the lattice, then setting $x_n=y_n/t$ forces the bottom left vertex to be of type $c_{-}$. A subset of the vertices then freeze, as shown on the right of Figure \ref{fig:UU-rec}.

In both cases, by reading off the Boltzmann weights for the frozen vertices we deduce the prefactors in the recursion relations \eref{UU-rec1} and \eref{UU-rec2}. One must also be mindful of the overall multiplicative factors which are introduced in the derivation of Figure \ref{fig:fish-exchange}, and take these into account when arriving at equation \eref{UU-rec2}. In either case, the surviving (unfrozen) region represents the partition function of one size smaller, $Z^{(n-1)}_{\rm UU}$.

\item The $n=1$ case of the partition function can be computed as a sum of five terms:
\begin{center}
\begin{tikzpicture}[scale=0.6]
\node at (-3,0.5) {$Z^{(1)}_{\rm UU}\ \ \ =$};
\draw[thick, smooth] (0,0)--(5/2,0);
\draw[thick, smooth] (0,1)--(5/2,1);
\node[label={left: \fs ${\color{red} \shortrightarrow} \ x_1$}] at (0,0) {};
\node[label={left: \fs ${\color{red} \shortleftarrow} \ \b{x}_1$}] at (0,1) {};
\node at (0.5,0) {$\r$}; \node at (1.5,0) {\r}; \node at (2.5,0) {\l};
\node at (0.5,1) {$\r$}; \node at (1.5,1) {\l}; \node at (2.5,1) {\r};
\node at (3,0.5) {$\bullet$};
\draw[thick, smooth] (5/2,0) arc (-90:90:0.5);
\draw[thick, smooth] (1,-1)--(1,3/2);
\draw[thick, smooth] (2,-1)--(2,3/2);
\draw[thick, smooth] (2,3/2) arc (0:180:0.5);
\node[label={below: \fs ${\color{red} \begin{array}{c} {\color{black} \b{y}_1} \\ \shortuparrow \end{array} }$}] at (1,-1) {};
\node[label={below: \fs ${\color{red} \begin{array}{c} {\color{black} y_1} \\ \shortuparrow \end{array} }$}] at (2,-1) {};
\node at (1,-0.5) {$\d$}; \node at (1,0.5) {\d}; \node at (1,1.5) {\u};
\node at (2,-0.5) {$\d$}; \node at (2,0.5) {\u}; \node at (2,1.5) {\d};
\node at (1.5,2) {$\bullet$};
\node at (4,0.5) {$+$};
\end{tikzpicture}
\begin{tikzpicture}[scale=0.6]
\draw[thick, smooth] (0,0)--(5/2,0);
\draw[thick, smooth] (0,1)--(5/2,1);
\node[label={left: \fs ${\color{red} \shortrightarrow} \ x_1$}] at (0,0) {};
\node[label={left: \fs ${\color{red} \shortleftarrow} \ \b{x}_1$}] at (0,1) {};
\node at (0.5,0) {$\r$}; \node at (1.5,0) {\r}; \node at (2.5,0) {\r};
\node at (0.5,1) {$\r$}; \node at (1.5,1) {\l}; \node at (2.5,1) {\l};
\node at (3,0.5) {$\bullet$};
\draw[thick, smooth] (5/2,0) arc (-90:90:0.5);
\draw[thick, smooth] (1,-1)--(1,3/2);
\draw[thick, smooth] (2,-1)--(2,3/2);
\draw[thick, smooth] (2,3/2) arc (0:180:0.5);
\node[label={below: \fs ${\color{red} \begin{array}{c} {\color{black} \b{y}_1} \\ \shortuparrow \end{array} }$}] at (1,-1) {};
\node[label={below: \fs ${\color{red} \begin{array}{c} {\color{black} y_1} \\ \shortuparrow \end{array} }$}] at (2,-1) {};
\node at (1,-0.5) {$\d$}; \node at (1,0.5) {\d}; \node at (1,1.5) {\u};
\node at (2,-0.5) {$\d$}; \node at (2,0.5) {\d}; \node at (2,1.5) {\d};
\node at (1.5,2) {$\bullet$};
\node at (4,0.5) {$+$};
\end{tikzpicture}
\begin{tikzpicture}[scale=0.6]
\draw[thick, smooth] (0,0)--(5/2,0);
\draw[thick, smooth] (0,1)--(5/2,1);
\node[label={left: \fs ${\color{red} \shortrightarrow} \ x_1$}] at (0,0) {};
\node[label={left: \fs ${\color{red} \shortleftarrow} \ \b{x}_1$}] at (0,1) {};
\node at (0.5,0) {$\r$}; \node at (1.5,0) {\l}; \node at (2.5,0) {\l};
\node at (0.5,1) {$\r$}; \node at (1.5,1) {\r}; \node at (2.5,1) {\r};
\node at (3,0.5) {$\bullet$};
\draw[thick, smooth] (5/2,0) arc (-90:90:0.5);
\draw[thick, smooth] (1,-1)--(1,3/2);
\draw[thick, smooth] (2,-1)--(2,3/2);
\draw[thick, smooth] (2,3/2) arc (0:180:0.5);
\node[label={below: \fs ${\color{red} \begin{array}{c} {\color{black} \b{y}_1} \\ \shortuparrow \end{array} }$}] at (1,-1) {};
\node[label={below: \fs ${\color{red} \begin{array}{c} {\color{black} y_1} \\ \shortuparrow \end{array} }$}] at (2,-1) {};
\node at (1,-0.5) {$\d$}; \node at (1,0.5) {\u}; \node at (1,1.5) {\u};
\node at (2,-0.5) {$\d$}; \node at (2,0.5) {\d}; \node at (2,1.5) {\d};
\node at (1.5,2) {$\bullet$};
\node at (4,0.5) {$+$};
\end{tikzpicture}
\begin{tikzpicture}[scale=0.6]
\draw[thick, smooth] (0,0)--(5/2,0);
\draw[thick, smooth] (0,1)--(5/2,1);
\node[label={left: \fs ${\color{red} \shortrightarrow} \ x_1$}] at (0,0) {};
\node[label={left: \fs ${\color{red} \shortleftarrow} \ \b{x}_1$}] at (0,1) {};
\node at (0.5,0) {$\r$}; \node at (1.5,0) {\r}; \node at (2.5,0) {\l};
\node at (0.5,1) {$\r$}; \node at (1.5,1) {\r}; \node at (2.5,1) {\r};
\node at (3,0.5) {$\bullet$};
\draw[thick, smooth] (5/2,0) arc (-90:90:0.5);
\draw[thick, smooth] (1,-1)--(1,3/2);
\draw[thick, smooth] (2,-1)--(2,3/2);
\draw[thick, smooth] (2,3/2) arc (0:180:0.5);
\node[label={below: \fs ${\color{red} \begin{array}{c} {\color{black} \b{y}_1} \\ \shortuparrow \end{array} }$}] at (1,-1) {};
\node[label={below: \fs ${\color{red} \begin{array}{c} {\color{black} y_1} \\ \shortuparrow \end{array} }$}] at (2,-1) {};
\node at (1,-0.5) {$\d$}; \node at (1,0.5) {\d}; \node at (1,1.5) {\d};
\node at (2,-0.5) {$\d$}; \node at (2,0.5) {\u}; \node at (2,1.5) {\u};
\node at (1.5,2) {$\bullet$};
\node at (4,0.5) {$+$};
\end{tikzpicture}
\begin{tikzpicture}[scale=0.6]
\draw[thick, smooth] (0,0)--(5/2,0);
\draw[thick, smooth] (0,1)--(5/2,1);
\node[label={left: \fs ${\color{red} \shortrightarrow} \ x_1$}] at (0,0) {};
\node[label={left: \fs ${\color{red} \shortleftarrow} \ \b{x}_1$}] at (0,1) {};
\node at (0.5,0) {$\r$}; \node at (1.5,0) {\r}; \node at (2.5,0) {\r};
\node at (0.5,1) {$\r$}; \node at (1.5,1) {\r}; \node at (2.5,1) {\l};
\node at (3,0.5) {$\bullet$};
\draw[thick, smooth] (5/2,0) arc (-90:90:0.5);
\draw[thick, smooth] (1,-1)--(1,3/2);
\draw[thick, smooth] (2,-1)--(2,3/2);
\draw[thick, smooth] (2,3/2) arc (0:180:0.5);
\node[label={below: \fs ${\color{red} \begin{array}{c} {\color{black} \b{y}_1} \\ \shortuparrow \end{array} }$}] at (1,-1) {};
\node[label={below: \fs ${\color{red} \begin{array}{c} {\color{black} y_1} \\ \shortuparrow \end{array} }$}] at (2,-1) {};
\node at (1,-0.5) {$\d$}; \node at (1,0.5) {\d}; \node at (1,1.5) {\d};
\node at (2,-0.5) {$\d$}; \node at (2,0.5) {\d}; \node at (2,1.5) {\u};
\node at (1.5,2) {$\bullet$};
\end{tikzpicture}
\end{center}

\noindent
Using the expressions \eref{boltz-sym} and \eref{U-boltz} for the Boltzmann weights, we obtain the explicit sum
\begin{align*}
Z^{(1)}_{\rm UU}
&=
-
\frac{
\sqrt{t}(1-t)^3 x_1
(p - x_1)
(1+\b{p}\b{y}_1)
}
{
(1-x_1 y_1)
(1-x_1 \b{y}_1)^2
}
-
\frac{
t^{3/2}(1-t)
(p x_1 - 1)
(1 + \b{p} \b{y}_1)
}
{
(1-x_1 \b{y}_1)
}
\\
&
-
\frac{
\sqrt{t} (1-t) (1-t x_1 y_1) (1-t x_1 \b{y}_1)
(p - x_1)
(y_1 + \b{p})
}
{
(1-x_1 y_1)^2
(1-x_1 \b{y}_1)
}
+
\frac{
\sqrt{t}(1-t)(1-tx_1 \b{y}_1)
(p - x_1)
(t \b{y}_1 + \b{p})
}
{
(1-x_1 \b{y}_1)^2
}
\\
&
+
\frac{
\sqrt{t} (1-t) (1-t x_1 \b{y}_1)
(p x_1 - 1)
(t + \b{p} y_1)
}
{
(1-x_1 y_1)(1-x_1 \b{y}_1)
},
\end{align*}
which can be simplified to
\begin{align*}
Z^{(1)}_{\rm UU}
=
\frac{
\sqrt{t}(1-t)(1-x_1^2)(y_1-t\b{y}_1) 
\left[
(pt+\b{p}) (x_1 y_1 + x_1 \b{y}_1)
-
(p+\b{p}) (1+tx_1^2)
\right]
}
{(1-x_1 y_1)^2 (1-x_1 \b{y}_1)^2}.
\end{align*}

\end{enumerate}
\end{proof}

\begin{thm}[Kuperberg]
The partition function on the double U-turn lattice is given by a product of determinants:
\begin{multline}
\label{ZUU}
Z_{\rm UU}(x_1,\dots,x_n;y_1,\dots,y_n;t)
=
\frac{
(\sqrt{t})^{n^2} \prod_{i=1}^{n} (1-x_i^2) (y_i-t \b{y}_i) 
\prod_{i,j=1}^{n} (1- t x_i y_j)^2 (1- t x_i \b{y}_j)^2
}
{\prod_{1 \leq i<j \leq n} (x_i - x_j)^2 (y_i - y_j)^2 (1 - t x_i x_j)^2 (1-\b{y}_i \b{y}_j)^2}
\\
\times
\det_{1\leq i,j \leq n}
\left[
\frac{(1-t)}{(1-x_i y_j) (1-t x_i y_j) (1-x_i \b{y}_j) (1-t x_i \b{y}_j)}
\right]
\det_{1\leq i,j \leq n}
\left[
\frac{(pt+\b{p})(x_i y_j + x_i \b{y}_j) - (p+\b{p}) (1+t x_i^2)}
{(1-x_i y_j) (1-t x_i y_j) (1-x_i \b{y}_j) (1-t x_i \b{y}_j)}
\right]. 
\end{multline}

\end{thm}

\begin{proof}
It is a simple matter to verify that \eref{ZUU} satisfies the six properties of Lemma \ref{lem:UU-prop}. The fact that these properties uniquely determine $Z_{\rm UU}$ is again a consequence of Lagrange interpolation. Indeed, the recursion relations \eref{UU-rec1} and \eref{UU-rec2} (together with the symmetry property {\bf 3} and quasi-symmetry property {\bf 4}) determine the polynomial $Z_{\rm UU}$ at $4n$ values of $x_n$. Combined with the two known zeros at $x_n = \pm 1$, these are sufficiently many points to fully determine $Z_{\rm UU}$.  
\end{proof}

\subsection{$u$-deformed $BC_n$ Cauchy identity at Schur and Hall--Littlewood level} 

In this subsection we present a (conjectural) $u$-deformation of equation \eref{HL-UASM}, involving lifted Koornwinder polynomials \cite{rai}. We build up to this via a simpler result at the level of symplectic Schur polynomials, which we are able to prove using the Cauchy--Binet identity. In that sense, the two results presented here are direct analogues of equations \eref{ASM-cor1} and \eref{ASM-cor2} in Section 
\ref{ssec:u-def-ASM}, and \eref{OSASM-cor1} and \eref{OSASM-cor2} in Section \ref{ssec:u-def-OSASM}. In contrast with those other equations, we are currently unable to obtain \eref{UASM-def1} and \eref{UASM-def2} as the 
$q \rightarrow 0$ case of some identity at Macdonald level.

\begin{thm}
\label{symp-schur-thm}

The Cauchy identity for symplectic Schur polynomials can be doubly refined, by the introduction of two deformation parameters $t$ and $u$:
\begin{multline}
\label{UASM-def1}
\sum_{\lambda}
\prod_{i=1}^{n}
(1 - u t^{\lambda_i - i + n})
s_{\lambda}(x_1,\dots,x_n)
sp_{\lambda}(y_1^{\pm 1},\dots,y_n^{\pm 1})
=
\\
\frac{1}{\prod_{1 \leq i<j \leq n} (x_i-x_j) (y_i-y_j) (1-\b{y}_i \b{y}_j)}
\det_{1\leq i,j \leq n}
\left[
\frac{1-u + (u-t) (x_i y_j + x_i \b{y}_j) + (t^2 - u) x_i^2}
{(1-x_i y_j) (1-t x_i y_j)(1-x_i \b{y}_j) (1-t x_i \b{y}_j)}
\right].
\end{multline}

\end{thm}

\begin{proof}
Using the Weyl-type determinant expressions for both $s_{\lambda}$ and $sp_{\lambda}$, and multiplying equation \eref{UASM-def1} by $\prod_{1 \leq i<j \leq n} (x_i-x_j) (y_i-y_j) (1-\b{y}_i \b{y}_j) \prod_{i=1}^{n} (y_i-\b{y}_i)$, the left hand side may be written as
\begin{multline*}
\sum_{\lambda}
\prod_{i=1}^{n}
(1 - u t^{\lambda_i - i + n})
\det_{1\leq i,j \leq n}
\left[ x_i^{\lambda_j-j+n} \right]
\det_{1\leq i,j \leq n}
\left[ y_j^{\lambda_i-i+n+1} - \b{y}_j^{\lambda_i-i+n+1} \right]
\\
=
\sum_{k_1 > \cdots > k_n \geq 0}\ 
\prod_{i=1}^{n}
(1 - u t^{k_i})
\det_{1\leq i,j \leq n}
\left[ x_i^{k_j} \right]
\det_{1\leq i,j \leq n}
\left[ y_j^{k_i+1} - \b{y}_j^{k_i+1} \right],
\end{multline*}
where we have made the change of summation indices $\lambda_i - i + n = k_i$. Applying the Cauchy--Binet identity, we obtain
\begin{multline*}
\sum_{\lambda}
\prod_{i=1}^{n}
(1 - u t^{\lambda_i - i + n})
\det_{1\leq i,j \leq n}
\left[ x_i^{\lambda_j-j+n} \right]
\det_{1\leq i,j \leq n}
\left[ y_j^{\lambda_i-i+n+1} - \b{y}_j^{\lambda_i-i+n+1} \right]
\\
=
\det_{1\leq i,j \leq n}
\left[
\sum_{k=0}^{\infty}
(1-ut^k) x_i^k (y_j^{k+1} - \b{y}_j^{k+1})
\right]
=
\prod_{i=1}^{n} (y_i - \b{y}_i)
\det_{1\leq i,j \leq n}
\left[
\frac{1-u + (u-t) (x_i y_j + x_i \b{y}_j) + (t^2 - u) x_i^2}
{(1-x_i y_j) (1-t x_i y_j)(1-x_i \b{y}_j) (1-t x_i \b{y}_j)}
\right]
,
\end{multline*}
where the final equality follows from the formal power series expansion of the entries of the determinant. Keeping track of the multiplicative factors that we introduced at the outset, we recover the result 
\eref{UASM-def1}.
  
\end{proof}

\begin{conj}
\label{u-def-cau-bc-hl}
The Cauchy identity for lifted Koornwinder polynomials at $q=0$, $T=0$ can be refined by the introduction of a single deformation parameter $u$:
\begin{multline}
\label{UASM-def2}
\sum_{\lambda}
\prod_{i=1}^{m_0(\lambda)}
(1 - u t^{i-1})
b_{\lambda}(t)
P_{\lambda}(x_1,\dots,x_n;t)
\tilde{K}_{\lambda}(y_1^{\pm 1},\dots,y_n^{\pm 1}; 0, t, u t^{n-1}; t_0, t_1, t_2, t_3)
=
\\
\prod_{i=1}^{n}
\frac{(1- t_0 x_i )(1-t_1 x_i)(1-t_2 x_i)(1-t_3 x_i)}{(1-t x_i^2)}
\frac{\prod_{i,j=1}^{n} (1- t x_i y_j) (1- t x_i \b{y}_j)}
{\prod_{1 \leq i<j \leq n} (x_i-x_j) (y_i-y_j) (1 - t x_i x_j) (1-\b{y}_i \b{y}_j)}
\\
\times
\det_{1\leq i,j \leq n}
\left[
\frac{1-u + (u-t) (x_i y_j + x_i \b{y}_j) + (t^2 - u) x_i^2}
{(1-x_i y_j) (1-t x_i y_j)(1-x_i \b{y}_j) (1-t x_i \b{y}_j)}
\right],
\end{multline}
where $\tilde{K}_{\lambda}(y_1^{\pm 1},\dots,y_n^{\pm 1}; 0, t, u t^{n-1}; t_0, t_1, t_2, t_3)$ is a lifted Koornwinder polynomial with $q=0$, $T= u t^{n-1}$ (see Section 7 of \cite{rai} and Appendix A for more details).

\end{conj}

We discuss briefly some important specializations of equations \eref{UASM-def1} and \eref{UASM-def2}. The $u=0$ specialization of \eref{UASM-def1} gives rise to the equation
\begin{align*}
\sum_{\lambda}
s_{\lambda}(x_1,\dots,x_n)
sp_{\lambda}(y_1^{\pm 1},\dots,y_n^{\pm 1})
=
\frac{
\det_{1\leq i,j \leq n}
\left[
\frac{1}{(1-x_i y_j)(1-x_i \b{y}_j)}
\right]
}
{\prod_{1 \leq i<j \leq n} (x_i-x_j) (y_i-y_j) (1-\b{y}_i \b{y}_j)}
=
\frac{\prod_{1\leq i<j \leq n} (1-x_i x_j)}
{\prod_{i,j=1}^{n} (1-x_i y_j) (1-x_i \b{y}_j)},
\end{align*}
which is the Cauchy identity for symplectic Schur polynomials \cite{sun}. Setting $u=t$, \eref{UASM-def1} reduces to Theorem 3 of \cite{bw}, which is a simple $t$-refinement of the Cauchy identity. 

In a similar way, \eref{UASM-def2} interpolates between two identities which appeared previously in \cite{bw}. When $u=0$, the lifted Koornwinder polynomial appearing in \eref{UASM-def2} has its $T$ argument equal to zero. As is explained in \cite{rai} and Appendix A, the lifted Koornwinder polynomials at $T=0$ have many nice properties, including the Cauchy-type identity \eref{lift-koorn-cauchy}. Setting $u=0$ in 
\eref{UASM-def2}, we obtain
\begin{multline*}
\sum_{\lambda}
b_{\lambda}(t)
P_{\lambda}(x_1,\dots,x_n;t)
\tilde{K}_{\lambda}(y_1^{\pm 1},\dots,y_n^{\pm 1}; 0, t, 0; t_0, t_1, t_2, t_3)
=
\\
\prod_{i=1}^{n}
\frac{(1- t_0 x_i )(1-t_1 x_i)(1-t_2 x_i)(1-t_3 x_i)}{(1-t x_i^2)}
\prod_{1\leq i<j \leq n}
\frac{(1-x_i x_j)}{(1-t x_i x_j)}
\prod_{i,j=1}^{n}
\frac{(1-tx_i y_j)(1-tx_i \b{y}_j)}{(1-x_i y_j) (1-x_i \b{y}_j)}
\end{multline*}
as expected, this being the $q=0$ specialization of \eref{lift-koorn-cauchy}. On the other hand, when $u=t$ the lifted Koornwinder polynomial in \eref{UASM-def2} has its $T$ argument equal to $t^n$. In this case, it is equal to a Koornwinder polynomial with $q=0$ (see equation \eref{t^n-koorn}). Since the Koornwinder polynomials at $q=0$ are type $BC_n$ Hall--Littlewood polynomials \cite{ven}, we expect to recover equation 
\eref{HL-UASM} at this value of $u$. We find that this is indeed the case, after we additionally set $t_i=0$ for $0 \leq i \leq 3$, since all of these parameters were ignored in the original statement of \eref{HL-UASM} in \cite{bw}.

In analogy with previous sections, we wish to point out a further specialization of $u$ which leads to a connection with the partition function \eref{ZUU}. By choosing the boundary parameter in \eref{ZUU} to be $\b{p}=-\sqrt{t}$, and setting $u=-t$ in \eref{UASM-def2}, we obtain agreement between the determinants appearing in \eref{ZUU} and \eref{UASM-def2} up to an overall factor of $(\sqrt{t})^n$. Furthermore by specializing $t_0 = 1$, $t_1 = -1$ and $t_2 = \sqrt{t}$, $t_3 = -\sqrt{t}$, we find that all of prefactors present in \eref{UASM-def2} are also present in \eref{ZUU}. The leftover factors in \eref{ZUU} are those of the UASM partition function, given by the right hand side of equation \eref{HL-UASM}. Hence this is yet another example of a Cauchy-type identity that is closely related to a partition function appearing in \cite{kup2}. 

\section*{Acknowledgments}

We express our sincere thanks to Ole Warnaar, for many valuable insights and suggestions which motivated this work, and in particular for suggesting the idea of 
$u$-deformations of the original identities \eref{HL-ASM}--\eref{HL-UASM}; and to Eric Rains, for helping us in arriving at Conjecture 1. MW would like to thank Fr\'ed\'eric Jouhet for his interest in this work and for pointing out the reference \cite{agj}; and Jean-Christophe Aval, Philippe Nadeau and Eric Ragoucy for invitations to present related work at LaBRI, ICJ and LAPTh, respectively. We finally wish to acknowledge the open-source package Sage whose built-in functions for Hall--Littlewood and Macdonald polynomials were indispensible in verifying some of the conjectures. This work was done under the support of the ERC grant 278124, ``Loop models, integrability and combinatorics''.  

\appendix

\section{Lifted Koornwinder polynomials}

\subsection{$BC_n$-symmetric interpolation and Koornwinder polynomials}

Here we define interpolation and Koornwinder polynomials via Okounkov's binomial formula \cite{oko2,rai}. We mostly follow the exposition in \cite{rai}, while at the same time providing an approach that lends itself to computing these polynomials on the computer. We will need the following parameters (generically, complex numbers): $q, t, s, t_0, t_1, t_2, t_3$. Henceforth we let $\lambda$ be a partition of length 
$\ell(\lambda) \leq n$, $x$ be an indeterminate, and $\vec{x}_n$ be an $n$-tuple of indeterminates $(x_1,\dots,x_n)$. We let $n(\lambda):=\sum_{i} (i-1)\lambda_i$ and make use of the following notations:
\begin{align*}
&
C^+_\lambda(x;q,t) := \prod_{(i,j)\in \lambda} (1-q^{\lambda_i+j-1} t^{2-\lambda'_j-i} x),
\quad\quad 
C^-_\lambda(x;q,t) := \prod_{(i,j)\in \lambda} (1-q^{\lambda_i-j} t^{\lambda'_j-i} x), 
\\
&
(x;q,t)_{\lambda} := \prod_{(i,j)\in \lambda} (1-q^{j-1} t^{1-i} x), 
\quad\quad
(x_1,\dots,x_k;q,t)_{\lambda} := \prod_{i=1}^{k} (x_i;q,t)_{\lambda},
\\
&
k^0_\lambda(q,t,T;t_0{:}t_1,t_2,t_3) 
:= 
(t_0 T/t)^{-|\lambda|} t^{n(\lambda)} 
\frac{(T,T t_0 t_1/t,T t_0 t_2/t,T t_0 t_3/t;q,t)_{\lambda}} 
{C^-_\lambda(t;q,t) C^+_\lambda(T^2 t_0 t_1 t_2 t_3/(qt^2);q,t)}.
\end{align*}
Okounkov's {\it $BC_n$-symmetric interpolation polynomials} $P^{*}_{\lambda}$ can be defined uniquely via the following three conditions (the ground field is $\mathbb{C}(q,t,s)$):
\begin{itemize}
\item 
$P^{*}_{\lambda}(\vec{x}_n;q,t,s) 
= 
m_{\lambda}(\vec{x}_n) + \sum_{\mu < \lambda} c_{\lambda, \mu} m_{\mu}(\vec{x}_n)$, where $m_{\mu}$ is the $BC_n$-symmetric monomial symmetric polynomial in $n$ variables and $<$ denotes dominance ordering on partitions. In particular, 
$P^{*}_{\lambda}$ has $BC_n$-symmetry.
\smallskip

\item 
$P^{*}_{\lambda}(\langle \mu \rangle_{q,t,s};q,t,s) = 0$ for $\mu < \lambda$, where the specialization $\langle \mu \rangle_{q,t,s}$ stands for sending $x_i \mapsto t^{n-i} q^{\lambda_i} s$.
\smallskip

\item 
$P^{*}_\lambda(\langle \lambda \rangle_{q,t,s};q,t,s) 
= 
(q t^{n-1} s)^{-|\lambda|} t^{n(\lambda)} 
q^{-2n(\lambda')} C^-_\lambda(q;q,t) C^+_\lambda(t^{2n-2} s^2;q,t)$.
\end{itemize}
Using these one can define the following {\it generalized binomial coefficients:}
\begin{align*}
\binomQ{\lambda}{\mu}_{q,t,s} 
= 
\frac{P^{*}_\mu(\langle \lambda \rangle_{q,t,s};q,t,t^{1-n} s)} 
{P^{*}_\mu(\langle \mu \rangle_{q,t,s};q,t,t^{1-n} s)}.
\end{align*}
These binomial coefficients vanish unless $\mu \subseteq \lambda$ and we have $\binomQ{\lambda}{\lambda}_{q,t,s} = \binomQ{\lambda}{0}_{q,t,s} = 1$. We can now define the {\it Koornwinder polynomials} via Okounkov's binomial formula:
\begin{align}
K_\lambda(\vec{x}_n;q,t;t_0,t_1,t_2,t_3)
=
\sum_{\mu\subseteq\lambda}
\binomQ{\lambda}{\mu}_{q,t,t^{n-1} \hat{t}_0}
\frac{k^0_\lambda(q,t,t^n;t_0{:}t_1,t_2,t_3)}
{k^0_\mu(q,t,t^n;t_0{:}t_1,t_2,t_3)}
P^{*}_\mu(\vec{x}_n;q,t;t_0),
\label{eq:koorn:binomial}
\end{align}
where $\hat{t}_0=\sqrt{t_0t_1t_2t_3/q}$. When $q=0$ the Koornwinder polynomials are the Hall--Littlewood polynomials of type $BC$ \cite{ven}. 

\subsection{Symmetric polynomials from $BC_n$-symmetric interpolation and Koornwinder polynomials}

We begin by discussing a family of inhomogeneous symmetric functions (and in finitely many variables, polynomials) called \textit{lifted interpolation functions}, introduced by Rains in \cite{rai} (this section follows most of Sections 6 and 7 of \cite{rai}). In addition to the parameters already defined in the previous section, we will also need an extra parameter $T$. The lifted interpolation functions $\tilde{P}^{*}_{\lambda}$ are defined uniquely (over the base field $\mathbb{C}(q,t,s,T)$) via the following three conditions:
\begin{itemize}
\item 
$\tilde{P}^*_\lambda(\vec{x}_n;q,t,T;s)
=
m_\lambda(\vec{x}_n)
+
\sum_{\mu: \mu < \lambda \ \text{or}\ \mu \leq \kappa, \kappa \subset \lambda} c_{\lambda, \mu} m_{\mu}(\vec{x}_n)$, where $m_{\mu}$ is the usual (type $A$) monomial symmetric polynomial.
\smallskip

\item $\tilde{P}^*_\lambda(\langle\mu\rangle_{q,t,T;s};q,t,T;s)=0$ for $\mu < \lambda$, where the \textit{plethystic} specialization $\langle\mu\rangle_{q,t,T;s}$ (i.e., $f(\langle\mu\rangle_{q,t,T;s})$ for $f$ a symmetric function) stands for sending the $k$-th power sum $p_k$ (for every $k$) to
\begin{align*}
p_k \mapsto \sum_{i=1}^{\ell(\mu)} 
\left[
(q^{k\mu_i}-1) t^{-ki} (sT)^k 
+ 
(q^{-k\mu_i}-1) t^{ki} (sT)^{-k}
\right] 
+ 
s^k \frac{1-T^k}{1-t^k} 
+ 
s^{-k} \frac{1-T^{-k}}{1-t^{-k}}.
\end{align*}
\smallskip
 
\item $\tilde{P}^*_\lambda(\langle\lambda\rangle_{q,t,T;s};q,t,T;s)=(qsT/t)^{-|\lambda|} t^{n(\lambda)} q^{-2n(\lambda')} C^-_\lambda(q;q,t) C^+_\lambda((sT/t)^2;q,t)$.
\end{itemize}
When $T=t^n$ (with $x$-variables specialized appropriately) the lifted interpolation polynomials reduce to Okounkov's $BC_n$-symmetric interpolation polynomials: 
\begin{align*}
\tilde{P}^*_\lambda(x_1,x_1^{-1},\dots,x_n,x_n^{-1};q,t,t^n;s)
=
\begin{cases}
P^*_\lambda(x_1,\dots, x_n;q,t,s), 
& 
\ell(\lambda) \leq n,
\\
0, 
& 
\ell(\lambda) > n.
\end{cases}
\end{align*}
The \textit{lifted Koornwinder symmetric functions} $\tilde{K}_{\lambda}$ are defined from the lifted interpolation polynomials via a formula analogous to Okounkov's binomial formula \eref{eq:koorn:binomial} for Koornwinder polynomials:
\begin{align*}
\tilde{K}_\lambda(\vec{x}_n;q,t,T;t_0,t_1,t_2,t_3)
=
\sum_{\mu\subseteq\lambda}
\binomQ{\lambda}{\mu}_{q,t,(T/t)\hat{t}_0}
\frac{k^0_\lambda(q,t,T;t_0{:}t_1,t_2,t_3)}
{k^0_\mu(q,t,T;t_0{:}t_1,t_2,t_3)}
\tilde{P}^*_\mu(\vec{x}_n;q,t,T;t_0).
\end{align*}
When $T=t^n$ (and with $x$-variables appropriately specialized), we recover Koornwinder polynomials:
\begin{align}
\label{t^n-koorn} 
\tilde{K}_\lambda(x_1, x_1^{-1},\dots, x_n, x_n^{-1};q,t,t^n;t_0,t_1,t_2,t_3) 
= 
\begin{cases}
K_\lambda(x_1,\dots, x_n;q,t;t_0,t_1,t_2,t_3),
&
\ell(\lambda) \leq n,
\\
0,
&
\ell(\lambda) > n.
\end{cases}
\end{align}
When $T=0$, the lifted Koornwinder polynomials satisfy the following Cauchy-like identity:
\begin{multline}
\label{lift-koorn-cauchy}
\sum_\lambda
b_{\lambda}(q,t)
P_\lambda(\vec{x}_n;q,t)
\tilde{K}_\lambda(\vec{y}_n;q,t,0;t_0,t_1,t_2,t_3)
=
\\
\prod_{i,j=1}^{n} \frac{(t x_i y_j;q)}{(x_i y_j;q)}
\prod_{1\leq i<j \leq n} \frac{(x_i x_j;q)}{(t x_i x_j;q)}
\prod_{i=1}^{n} \frac{(t_0 x_i,t_1 x_i,t_2 x_i,t_3 x_i;q)}{(t x_i^2;q)},
\end{multline}
where $P_{\lambda}(\vec{x}_n;q,t)$ is a Macdonald polynomial.

\bibliographystyle{abbrv}
\bibliography{PPASM}

\end{document}